\documentclass{amsart}
\usepackage{graphicx} 
\usepackage[utf8]{inputenc}

\usepackage{amsmath,mathtools,amssymb,extarrows,mathrsfs,amsthm}
\usepackage{tikz-cd}
\usepackage[T1]{fontenc}

\usepackage[utf8]{inputenc}
\usepackage{hyperref}
\usepackage{enumitem}

\usepackage{color}

\theoremstyle{plain} 
\newtheorem{theorem}{Theorem}
\numberwithin{theorem}{section}
\newtheorem*{theorem*}{Theorem}

\newtheorem{prop}[theorem]{Proposition}
\newtheorem{lemma}[theorem]{Lemma}
\newtheorem{coro}[theorem]{Corollary}

\newtheorem{definition}[theorem]{Definition}
\theoremstyle{definition}

\newtheorem{example}{Example}
\theoremstyle{remark} 
\newtheorem{remark}[theorem]{Remark}
\theoremstyle{definition}

\newtheorem{mainthm}{Theorem}
\newtheorem{maincor}[mainthm]{Corollary}
\newtheorem{mainconj}[mainthm]{Conjecture}

\newcommand{\wt}{\widetilde}

\newcommand{\La}{\Lambda}

\newcommand{\be}{\begin{equation} }
\newcommand{\ee}{\end{equation} }

\newcommand{\A}{\mathbb{A}}

\newcommand{\C}{\mathbb{C}} 
\newcommand{\D}{\mathbb{D}}

\newcommand{\N}{\mathbb{N}}

\renewcommand{\P}{\mathbb{P}}
\newcommand{\Q}{\mathbb{Q}}

\newcommand{\Z}{\mathbb{Z}}

\newcommand{\cD}{\mathcal{D}}
\newcommand{\cE}{\mathcal{E}}
\newcommand{\cF}{\mathcal{F}}
\newcommand{\cG}{\mathcal{G}}
\newcommand{\cH}{\mathcal{H}}
\newcommand{\cI}{\mathcal{I}}

\newcommand{\cK}{\mathcal{K}}
\newcommand{\cL}{\mathcal{L}}
\newcommand{\cM}{\mathcal{M}}
\newcommand{\cN}{\mathcal{N}}
\newcommand{\cO}{\mathscr{O}}
\newcommand{\cP}{\mathcal{P}}
\newcommand{\cQ}{\mathcal{Q}}

\newcommand{\cS}{\mathcal{S}}
\newcommand{\cT}{\mathcal{T}}

\newcommand{\bR}{\mathbf{R}}

\newcommand{\bn}{\mathbf{n}}
\newcommand{\bm}{\mathbf{m}}

\newcommand{\gr}{\textup{gr}}

\newcommand{\sM}{\mathscr{M}}
\newcommand{\sN}{\mathscr{N}}

\newcommand{\sB}{\mathscr{B}}

\newcommand{\sA}{\mathscr{A}}

\newcommand{\sG}{\mathscr{G}}
\newcommand{\sS}{\mathscr{S}}

\newcommand{\shD}{\mathscr{D}}

\newcommand{\Spec}{\textup{Spec }}
\newcommand{\Specan}{\textup{Spec}^{\textup{an}}}

\newcommand{\DR}{\textup{DR}}

\newcommand{\Mod}{\textup{Mod}}
\newcommand{\coh}{\textup{coh}}

\newcommand{\an}{\textup{an}}

\newcommand{\id}{\textup{id}}

\newcommand{\supp}{\textup{supp}}
\newcommand{\codim}{\textup{codim}}

\newcommand{\bs}{{\bf s}}

\newcommand{\Ann}{{\textup{Ann}}}

\newcommand{\Ch}{\textup{Ch}}
\newcommand{\CC}{\textup{CC}}

\newcommand{\Rhom}{\textup{Rhom}}
\newcommand{\cRhom}{\mathcal Rhom}
\newcommand{\Exp}{\textup{Exp}}

\title{Generalized nearby cycles via relative and logarithmic $\mathscr{D}$-modules}
\author{Lei Wu}
\address{Lei Wu, School of Mathematical Sciences, Zhejiang University, Hangzhou, 310058, P.R.China}
\email{leiwu23@zju.edu.cn}

\begin{document}

\maketitle
\begin{abstract}
    For a regular map $F$ from a complex smooth affine variety $X$ to $\A^r_\C$, we construct generalized nearby-cycle modules of a regular holonomic $\shD_X$-module $\cM$ along log strata with the log structure induced by the graph of $F$, whose relative supports are infinite unions of translated linear subvarieties of $\C^r$ determined by the zero loci of Bernstein-Sato ideals along monoid ideals. For a fixed log stratum, the nearby-cycle module corresponds to the Sabbah specialization complex of $\DR(\cM)$ under the relative regular Riemann-Hilbert correspondence of Fiorot-Fernandes-Sabbah, which generalizes the classical comparison theorem of Kashiwara-Malgrange for Deligne's nearby cycles. As an application, when $\cM=\cO_X$, we give a topological interpretation of the zero loci of Bernstein-Sato ideals of $F$ along monoid ideals under the exponential map, which answers a question of Budur-Shi-Zuo. 
\end{abstract}

\tableofcontents

\section{Introduction}
Given a holonomic $\shD$-module on a complex manifold, Kashiwara \cite{KasV} and Malgrange \cite{MalV} constructed a filtration of the module along the hypersurface of a given holomorphic function, now standardly known as the Kashiwara-Malgrange $V$-filtration, by employing a $\shD$-module version of the deformation to the normal cone. By passing to the associated graded pieces,
they obtained the nearby and vanishing cycles of holonomic $\shD$-modules along the hypersurface. They further established a comparison theorem between Deligne’s topological nearby and vanishing cycles for perverse sheaves \cite{DelV} and their $\shD$-module counterparts under the classical Riemann–Hilbert correspondence. 
Later, Beilinson and Bernstein \cite{Beiglue,BBJats} gave an alternative construction for both frameworks using algebraic relative $\shD$-modules over one dimensional base spaces and Bernstein-Sato polynomials (see also \cite{BG}). Meanwhile, for a finite union of holomorphic functions, Sabbah \cite{Sab90} introduced Alexander complexes (also termed Sabbah specialization complexes) by considering normal deformations along smooth complete intersections and their toric birational modifications, thereby generalizing Deligne's nearby cycles.  

The purpose of this paper is to construct generalized nearby-cycle modules of regular holonomic $\shD$-modules in the algebraic setting, such that they correspond to the Sabbah specialization complexes under the relative regular Riemann-Hilbert correspondence of Fiorot-Fernandes-Sabbah \cite{FFS23}, drawing inspiration from the work of Sabbah and Beilinson-Bernstein. Our primary technical framework rests on
the theory of relative regular holonomic $\shD$-modules developed in 
\cite{FS13, FTM18, FS17, FFS19, FFS23}, and the theory of logarithmic $\shD$-modules. We emphasize that we work in the algebraic category primarily to simplify notation; the local analytic case can be treated analogously by the same methods.

\subsection{Log structures and Bernstein-Sato ideals}
Let $F=(f_1,f_2,\dots,f_r)\colon X\to \A_\C^r=\Spec \C[\N^r]$ be a morphism with $X=\Spec S$ an irreducible smooth affine algebraic variety over $\C$ and let $\cM$ be a regular holonomic $\shD_X$-module with $\shD_X$ the ring of (algebraic) differential operators on $X$. Then we consider the log structure on $Z=X\times \A_\C^r$ given by the prelog structure 
\[\alpha\colon\N^r\hookrightarrow S[\N^r]\simeq S[t_1,\dots,t_r].\]
We call it the log structure induced by the graph of $F$.

On the topological side, using the log structure, we give an alternative description of Sabbah specialization complexes. More precisely, for $K\subseteq \N^r$ a monoid ideal and for $\bm \in \N^r$, we have $\psi_{K_\bm}(\cF^\bullet)$, the \emph{Sabbah specialization complex} of a constructible complex $\cF^\bullet$ along $Z_{K_\bm}$, where $Z_{K_\bm}\subseteq Z$ is the locally closed subscheme defined by $\alpha(K_\bm)$ and $K_\bm$ is the localization of $K$ at $\bm$ (cf. \cite[I.1.4]{Ogusbook}); see \S\ref{subsect:SabSp} for details. 

On the $\shD$-module side, we consider a log $\shD$-module lattice (see \S\ref{subsect:logDmodule} for the definition) 
\[\cN_0\coloneqq\shD_X[s_1,\dots,s_r]\cH_0\cdot F^\bs\subseteq \cM[1/f,s_1,\dots,s_r]F^\bs,\]
where we write $F^\bs=\prod_{i=1}^rf_i^{s_i}$ and $f=\prod_{i=1}^rf_i$,
and $\cH_0\subseteq\cM[1/f]$ is a finitely generated $S$-module such that $\shD_X[1/f]\cdot \cH_0=\cM[1/f]$. Then, we give a logarithmic interpretation of Bernstein-Sato ideals (or polynomials) by defining 
\[B^{K_\bm}(\cN_0)\coloneqq \Ann_{\C[s_1,\dots,s_r]}\frac{\cN_0(*E_\bm)}{\alpha(K)\cdot\cN_0(*E_\bm) 
}.\]
See \S\ref{subsect:logBSideal} for more details.
Even when $\cM$ is irregular, $B^{K_\bm}(\cN_0)$ always contains a distinguished element in the form of products of linear polynomials (see Lemma \ref{lm:subbahgb}), based on a fundamental result of Sabbah \cite{Sab}. Taking $\cN_F=\shD_X[s_1,\dots,s_r]F^\bs$, we get the \emph{Bernstein-Sato ideals} of $F$ along $K_\bm$, 
\[B^{K_\bm}_F\coloneqq B^{K_\bm}(\cN_F)\subseteq \C[s_1,\dots,s_r].\]
When $\bm=e\coloneqq(0,0,\dots,0)$, the unit in $\N^r$, since $K_e=K$, we have $B^{K_e}_F=B^{K}_F$. In particular, when $K$ is the monoid ideal generated by $(1,1,\dots,1)\in \N^r$, we have $B_F\coloneqq B_F^{(1,1,\dots,1)}$, the classical Bernstein-Sato ideal of $F$. More generally,
when $K\subseteq \N^r$ is generated by $\bm_1,\dots,\bm_p$, we have 
\[B_F^K=B_F^{\bm_1,\dots,\bm_p}\]
where the later is introduced and studied by Budur \cite[\S4]{Budur}. The Bernstein-Sato ideal $B_F$ has been widely studied from both theoretical and computational perspectives in literature, to just list a few \cite{May97, BMMBSI, Bahl, CaLie, Mai, Bath, OkTak, UchaCas, BaOK,DGPS}.

The logarithmic interpretation of Bernstein-Sato ideals enables us to prove that $B^{K_\bm}(\cN_0)\subseteq\C[s_1,\dots,s_r]$ is a non-zero proper ideal under some geometric condition.  In particular, we obtain that if $(f_1=f_2=\cdots=f_r=0)\subseteq X$ is not empty, then $B_F^K$ is a non-zero proper ideal for every monoid ideal $K\subsetneq \N^r$. See Theorem \ref{thm:nozeroproperBSI} and Proposition \ref{prop:localizeBF} for details.

Our first main result is a comparison between the zero loci of Bernstein-Sato ideals $Z(B^{K_\bm}(\cN_0))$ and the relative supports of the Sabbah specialization complexes $\supp\big(\psi_{K_\bm}(\DR(\cM))\big)$, where $\DR(\cM)$ denotes the de Rham complex of $\cM$.
\begin{mainthm}\label{thm:mainA}
    With notation as above, let $\cM$ be a regular holonomic $\shD_X$-module. Then, for $K\subseteq \N^r$ and $\bm\in \N^r$, 
    \[\textup{Exp}\big(Z(B^{K_\bm}(\cN_0))\big)=\supp\big(\psi_{K_\bm}(\DR(\cM))\big)\subseteq(\C^*)^r\]
    is a finite union of translated subtori, where Exp$:\C^r\to (\C^*)^r$ is the universal covering map.
    Furthermore, if $\cM$ underlies a $\Q$-mixed Hodge module, then 
    \[\textup{Exp}\big(Z(B^{K_\bm}(\cN_0))\big)=\supp\big(\psi_{K_\bm}(\DR(\cM))\big)\subseteq(\C^*)^r\]
    is a finite union of torsion translated subtori.
\end{mainthm}
Taking $\cN_0=\cN_F$, we immediately obtain:
\begin{maincor}\label{cor:mainB}
    \[\textup{Exp}\big(Z(B^{K_\bm}_F)\big)=\supp\big(\psi_{K_\bm}(\C_X)\big).\]
\end{maincor}
Motivated by a classical theorem of Kashiwara and Malgrange (cf. \cite[Thm.1.4.1(i)]{BVWZ}), Budur \cite{Budur} raised a question regarding how to topologically interpret zero loci of Bernstein-Sato ideals of $F$ (see also \cite[Remark 6.8.(2)]{BSZ25}). While this question was partially addressed by the main results in \cite{BVWZ,BVWZ2},
Corollary \ref{cor:mainB} gives a complete answer. 

In the context of multivariable monodromy zeta functions, Sabbah \cite[0.2]{Sab90} established that the relative support of the Sabbah specialization complex of $\C_X$ is contained in a finite union of translated subtori. Subsequently, Budur and Wang \cite{BuWa} utilized cohomology jumping loci of rank-one local systems to prove that the relative supports of the germs of the Sabbah specialization of $\C_X$ along
$F$ are finite unions of torsion translated subtori.
Their approach, however, relies on some approximation theorem to reduce the problem to the case where all $f_i$ are polynomials defined over $\bar\Q$.
In contrast, Theorem A and its proof offer a distinct approach under more general assumptions. Our new perspective reveals that studying Bernstein–Sato ideals from a log-$\shD$-module standpoint is both natural and essential.

\subsection{Generalized nearby-cycle modules}To prove Theorem \ref{thm:mainA}, we construct in \S\ref{subsect:rmel} the \emph{relative minimal extensions} of regular holonomic $\shD$-modules (and more generally relative regular holonomic $\shD$-modules) along divisors supported in the log boundary by using graph embeddings and the log structure induced by the graphs, which naturally generalize the  
Deligne-Goresky-MacPherson minimal extensions for perverse sheaves (cf. \cite[\S8.2.1]{HTT}) and the $j_!$-extension of Beilinson and Bernstein along a regular function (cf. \cite[\S4.1]{BBJats} and \cite[\S3.8]{Gil}). Leveraging these minimal extension, we further construct $\wt{i^!_{K_\bm}}\cN_0$, the \emph{generalized nearby-cycle module} of the lattice $\cN_0$ along $K_\bm$, which is shown to be a relative regular holonomic $\shD_{X^\an\times \C^r/\C^r}$-module. See \S\ref{subsect:r!pullback} and Definition \ref{def:grnearby} for details. These constructions depend on two main ingredients:
\begin{enumerate}
    \item a logarithmic characteristic cycle formula for lattices of relative regular holonomic $\shD$-modules in Theorem \ref{thm:relGins-Sab},
    \item a logarithmic Artin-Rees theorem developed in \S\ref{subsect:logArtinRees}. 
\end{enumerate}

The logarithmic characteristic cycle formula is a relative generalization of \cite[Thm.A1.2]{Gil1}. Following the approach of Brian\c{c}on-Maisonobe-Merle \cite{BMM}, we derive the characteristic cycle formula as a direct consequence of Sabbah's relative characteristic variety formula (Theorem A.1 in Appendix). This result allows us to control the logarithmic characteristic varieties of lattices, and consequently the relative characteristic varieties of lattices of graph-embedding type (Definition \ref{def:getype}), under specialization along divisorial log boundaries (Theorem \ref{thm:mainsuppfiberation}). Such control is crucial for performing induction along log strata.


Furthermore, while it is well known that the $\shD$-module analogue of the Artin-Rees Lemma from commutative algebra does not hold, we observe that ideals in the ring of logarithmic differential operators generated by
$\alpha(K)$ for $K\subseteq \N^r$ are two-sided (Lemma \ref{lm:twosidedideal}). This observation ensures that the Artin–Rees property remains valid for log $\shD$-modules with respect to such ideals
(Theorem \ref{thm:artin-reeslog}). We apply this logarithmic Artin–Rees theorem to prove that the generalized 
nearby cycles $\wt{i^!_{K_\bm}}\cN_0$ constitute a module rather than merely a complex (Theorem \ref{thm:upper!DI}).

More explicitly, the generalized nearby cycles of $\cN_0$ along $K$ is given by the following projective limit:
\be\label{eq:eqlimprojmain}
\wt{i^!_{K}}\cN_0\simeq \lim_{\stackrel{\longleftarrow}{j} }\frac{\wt{\cN_0}}{\sum_{i=1}^p\wt{\alpha(v_i)^j\cN_0}}
\ee
as $j\to+\infty$, where $\sim$ denotes the analytic sheafification functor (see \S\ref{subsect:relregularhol}), and $v_1,\dots,v_p\in \N^r$ generate $K$. Motivated by the double-limit construction in \cite{Beiglue}, we establish a double limit giving the largest generalized nearby cycles:
\be\label{eq:doublelimtmain}
\Psi_{K}(\cM)\coloneqq \wt{i^!_{K}}\cN_0(*E)\simeq \lim_{\longleftrightarrow}\frac{\wt{\alpha(\bn)^{-k}\cN_0}}{\sum_{i=1}^p\wt{\alpha(v_i)^j\cN_0}}, 
\ee
as $\alpha(\bn)^{-k}\cdot\cN_0\to \cN_0(*E)$ for any $\bn\in \N_{>0}^r$, where $E\subseteq Z$ is the divisor defined by $(t_1\cdots t_r=0)$.
For $\wt{i^!_{K_\bm}}\cN_0$ and $\Psi_{K_\bm}(\cM)\coloneqq \wt{i^!_{K_\bm}}\cN_0(*E)$ in general, one can first localize the log structure at $\bm$, and then get similar limits. See Proposition \ref{prop:doublelimitnearby} and Remark \ref{rmk:localizingupper!}. 

Consequently, under the relative regular Riemann–Hilbert correspondence, we establish a comparison theorem (Theorem \ref{thm:sabbspcmp})
\be\label{eq:compmain}
\pi_*^{\sG_r}\DR_{X^\an\times\C^r/\C^r}\big(\Psi_{K_\bm}(\cM)\big)\simeq \wt{\psi}_{K_\bm}\big(\DR(\cM)\big),
\ee
which generalizes the comparison of Kashiwara \cite{KasV} and Malgrange \cite{MalV} for Deligne nearby cycles. Here, $\sG_r\simeq \Z^r$ denotes the free abelian group generated by $t_1,\dots, t_r$ associated with the log structure, and $\pi_*^{\sG_r}$ is the equivariant direct image functor (cf. \cite[Part I.0.3.]{BLequiv}). It is worth noting that our construction of generalized nearby cycles extends to relative regular holonomic $\shD$-modules of graph-embedding type in a more general setting (see \S\ref{subsect:rmel} and \S\ref{subsect:r!pullback}).

The double-limit in Eq.\eqref{eq:doublelimtmain} indicates that algebraic relative  $\shD$-modules can be utilized to locally approximate the analytic relative $\shD$-module $\Psi_K(\cM)$. Furthermore, upon taking the equivariant direct image, the comparison in
Eq.\eqref{eq:compmain} ensures that
the $\shD_{X^{\an}\times (\C^*)^r/(\C^*)^r}$-module $\pi_*^{\sG_r}\big(\Psi_K(\cM)\big)$ is $dR$-algebraic (Definition \ref{def:dralgebraic}). This algebraicity allows us to invoke an Ax–Lindemann-type criterion developed by Budur and Wang 
(Lemma \ref{lm:BudurWang}). Combining these ingredients, we arrive at our second main result:
\begin{mainthm}\label{thm:mainC}
With notation as above, let $\cM$ be a regular holonomic $\shD_X$-module. Then, for $K\subseteq \N^r$ and $\bm\in \N^r$ we have:
    \begin{enumerate}[label=$(\arabic*)$]
        \item the $\sG_r$-invariant subset $\supp(\wt{i^!_{K_\bm}}\cN_0(*E))\subseteq\C^r$ is an infinite (countable) union of translated linear subvarieties such that 
        \[\supp(\psi_{K_\bm}(\DR(\cM)))=\textup{Exp}\big(\supp(\wt{i^!_{K_\bm}}\cN_0(*E))\big)\subseteq (\C^*)^r\]
        is a finite union of translated subtori,  
        \item for every $\bn\in \N^r$, $\supp(\wt{i^!_{K_\bm}}\cN_0(*E_\bn))\subseteq\C^r$ is an infinite (countable) union of translated linear subvarieties,
        \item if $\sN$ is a coherent $\shD_{X^\an\times \C^r/\C^r}$-submodule of $\wt{i^!_{K_\bm}}\cN_0(*E_\bn)$ for $\bn\in \N^r$, then  $\supp(\sN)\subseteq \C^r$ is a countable union of translated linear subvarieties.
    \end{enumerate}
\end{mainthm}
The third part of the aforementioned theorem asserts that every irreducible components of the relative support of $\wt{i^!_{K_\bm}}\cN_0(*E_\bn)$, including any ``embedded" component, is a translated linear subvariety of $\C^r$. Theorem \ref{thm:mainC} thus implies that the relative supports of generalized nearby cycles exhibit desirable geometric linearity, echoing the behavior of jumping indices in the $\C$-indexed Kashiwara-Malgrange $V$-filtrations. In cases where $\cM$ underlies a $\Q$-mixed Hodge modules (for instance $\cM=\cO_X$), Theorem \ref{thm:mainC} can be further refined as follows:


\begin{mainthm}\label{thm:mainD}
If $\cM$ is a regular holonomic $\shD_X$-module underlying a $\Q$-mixed Hodge module, then for $K\subseteq \N^r$ and $\bm\in \N^r$ we have:
    \begin{enumerate}[label=$(\arabic*)$]
        \item the $\sG_r$-invariant subset $\supp(\wt{i^!_{K_\bm}}\cN_0(*E))\subseteq\C^r$ is an infinite (countable) union of $\Q$-translated linear subvarieties such that 
        \[\supp(\psi_{K_\bm}(\DR(\cM)))=\textup{Exp}\big(\supp(\wt{i^!_{K_\bm}}\cN_0(*E))\big)\subseteq (\C^*)^r\]
        is a finite union of torsion translated subtori,  
        \item for every $\bn\in \N^r$, $\supp(\wt{i^!_{K_\bm}}\cN_0(*E_\bn))\subseteq\C^r$ is an infinite (countable) union of $\Q$-translated linear subvarieties,
        \item if $\sN$ is a coherent $\shD_{X^\an\times \C^r/\C^r}$-submodule of $\wt{i^!_{K_\bm}}\cN_0(*E_\bn)$ for $\bn\in \N^r$, then $\supp(\sN)\subseteq \C^r$ is a countable union of $\Q$-translated linear subvarieties.
    \end{enumerate}
\end{mainthm}
The strategy for proving the above theorem relies on finding $\Q$-points within each irreducible component of the relative support of the nearby-cycle module, a task facilitated by Theorem \ref{thm:mainC}. We first use the $\Q$-specializability of $\Q$-mixed Hodge modules (see \S\ref{sect:Qspecial}) to generate such $\Q$-points in the case where the log boundary is a divisor. Subsequently, we utilize the fact that the property of being $\Q$-mixed Hodge modules is stable under the nearby cycles functor along a hypersurfaces, allowing for induction across different log strata.
This inductive process depends on two technical results, Theorem \ref{thm:mainsuppfiberation} and Theorem \ref{thm:equivmainsuppfiberation}, both of which are derived from the log characteristic cycle formula in Theorem \ref{thm:relGins-Sab}.   
\subsection{Geometric Budur-type conjecture}
Based on computational examples, Budur conjectured that for $\bm_1,\dots,\bm_p\in\N^r$ the ideal $B_F^{\bm_1,\dots,\bm_p}$ is generated by products of linear polynomials of type $c_1s_1+\cdots+c_rs_r+c_0$ with $c_i\in \N$ and $c_0>0$ (cf. \cite{Budur} and \cite[Remark.6.8.(1)]{BSZ25}). 
Given a monoid ideal $K\subseteq\N^r$, Eq.\eqref{eq:eqlimprojmain} yields a surjection
\[\wt{i^!_{K}}\cN_0\twoheadrightarrow \frac{\wt\cN_0}{\wt{K\cdot\cN_0}}\]
and specifically, $\wt{i^!_{K}}\cN_F\twoheadrightarrow \frac{\wt\cN_F}{\wt{K\cdot\cN_F}}.$ Since $\sim$-functor is faithfully flat, following Eq.\eqref{eq:suppbideal=} we get
\[Z(B^K(\cN_0))\subseteq \supp(\wt{i^!_{K}}\cN_0).\]
We have shown that $\supp(\wt{i^!_{K}}\cN_0)$ is a union of translated linear subvarieties (or a union of $\Q$-translated linear subvarieties if $\cM$ underlies a $\Q$-mixed Hodge module).
In particular, $Z(B^K_F)$ is contained in a union of $\Q$-translated linear subvarieties, thereby providing geometric evidence of Burdur's conjecture. In our setting, it is then natural to ask whether a Budur-type conjecture holds for $\cN_0$ more generally, at least in a geometric sense.
\begin{mainconj}[Geometric Budur-type Conjecture]
    With notation as above, let $\cM$ be a regular holonomic $\shD_X$-module. For $K\subseteq \N^r$ and $\bm\in\N^r$, 
    $$Z(B^{K_\bm}(\cN_0))\subseteq \C^r$$
    is a finite union of translated linear subvarieties. Moreover, if $\cM$ underlies a $\Q$-mixed Hodge module, then   $Z(B^{K_\bm}(\cN_0))\subseteq \C^r$ is a finite union of $\Q$-translated linear subvarieties.
\end{mainconj}

Lastly, we outline the structure of the paper. Section \ref{sect:logdmodule} provides the necessary preliminaries on the theory of log $\shD$-modules in the normal crossing setting. In Section \ref{sect:GAGArel}, we discuss algebraic and analytic relative constructible sheaves and relative $\shD$-modules. Section \ref{sect:Qspecial} is devoted to the $\Q$-specializability of $\shD$-modules. Section \ref{sect:rmegnc} constitutes the core of the paper, where we construct relative minimal extensions and generalized nearby cycles modules. In Section \ref{subsect:logBSideal}, we provide a logarithmic interpretation of Bernstein–Sato ideals and establish our main results. Finally, Appendix I, written by Sabbah, establishes a formula for the relative characteristic variety, while Appendix II applies Sabbah’s formula to derive the logarithmic characteristic cycle formula.




\noindent{\bf Acknowledgement.}
The author would like to thank Nero Budur and Botong Wang for many helpful discussions and for answering various questions. A large part of this paper was prepared during the author’s visit to Harvard University; he is grateful to his former advisor, Mihnea Popa, for the invitation, his hospitality, and insightful discussions. Furthermore, the author deeply thanks Claude Sabbah for answering questions.

The author was supported by a start-up grant from Zhejiang University.

\section{Logarithmic $\shD$-modules in the normal crossing case}\label{sect:logdmodule}
\subsection{Basic set-up}\label{subsect:logDmodule}
We recall some terminology from the theory of logarithmic $\shD$-modules for the log structure associated with a normal crossing divisor; see \cite{WZ}, and also \cite{WuHohl} for the general theory.

The standard smooth log structure on
\[
\A_\C^r=\Spec \C[\N^r]\simeq \Spec \C[t_1,\dots,t_r]
\]
is given by the pre-log structure
\[
\N^r\hookrightarrow \C[\N^r]\simeq \C[t_1,\dots,t_r],
\]
where $\N=\{0,1,2,\dots\}$. Pulling this standard log structure back to
$Z=X\times \A_\C^r$, where $X=\Spec S$ is an irreducible smooth affine complex variety of dimension $n$, we obtain a log structure on
$Z=\Spec S[\N^r]$ given by
\[
\alpha\colon \N^r\hookrightarrow S[\N^r]\simeq S[t_1,\dots,t_r].
\]
Let $K\subseteq \N^r$ be a monoid ideal. Then $K$ defines an idealized log subscheme
\[
Z_K=\Spec \frac{S[\N^r]}{\alpha(K)\cdot S[\N^r]}\subseteq Z
\]
see \cite[III.1.3]{Ogusbook}.

Let $E=(\prod_{i=1}^r t_i=0)$ be the normal crossing divisor on $Z$. The above log structure on $Z$ is precisely the log structure associated with $E$. We denote by $\shD_{Z,E}$ the ring of logarithmic differential operators with respect to this log structure; explicitly,
\[
\shD_{Z,E}
=
\shD_X[t_1,\dots,t_r]\langle t_1\partial_{1},\dots,t_r\partial_r\rangle,
\]
with defining relations
\[
[t_i\partial_i,t_j]=\delta_{i,j}\cdot t_i
\]
cf. \cite[\S2]{WZ}. Thus $\shD_{Z,E}$ is a $\C$-subalgebra of $\shD_Z$, the ring of differential operators on $Z$.

The rank of the characteristic sheaf of the log structure determines the log stratification
\[
Z=\bigsqcup_k E^k,
\]
where the rank of the characteristic sheaf on $E^k$ is exactly $k$; cf. \cite{Ogusbook}. Let $Y$ be the closure of an irreducible component of $E^k$ for some $k\geq 0$. After relabeling the indices, $Y$ is the idealized log subvariety defined by the submonoid generated by $e_{r-k+1},\dots,e_r$, where $e_i$ denotes the $i$-th standard basis vector of $\N^r$. Since the monoid ideal defining $Y$ is the complement of a face of $\N^r$ \cite[Def.1.4.1]{Ogusbook}, $Y$ inherits a smooth log structure. More explicitly, we have a log closed embedding
\[
i_Y\colon Y=X\times \A_\C^l\hookrightarrow Z=X\times \A_\C^r,
\quad
(x,t_1,\dots,t_l)\mapsto (x,t_1,\dots,t_l,0,\dots,0),
\]
where $l=r-k$, and the induced normal crossing divisor on $Y$ is
\[
D=\left(\prod_{i=1}^l t_i=0\right).
\]
Notice that when $k=r$, the variety $Y$ is simply $X$ with empty log boundary.

We write
\[
R^{<l>}_k=\C[s_{l+1},\dots,s_{k+l}]
\quad \text{and} \quad
R_r=\C[s_1,\dots,s_r].
\]
We then set
\[
\sA^k_Y\coloneqq \shD_Y\otimes_\C R^{<l>}_k
\simeq \shD_Y\otimes_\C R_k,
\quad
\sA_X^r\coloneqq \shD_X\otimes_\C R_r,
\quad \textup{and} \quad
\sA^k_{Y,D}\coloneqq \shD_{Y,D}\otimes_\C R^{<l>}_k,
\]
where $\shD_{Y,D}$ denotes the ring of logarithmic differential operators with respect to the induced log structure associated with the normal crossing divisor $D\subseteq Y$. More generally, throughout the paper we let $R$ be a regular commutative finite-type $\C$-algebra domain and set
\[
\sA^R_Y\coloneqq \shD_Y\otimes_\C R,
\]
together with its subring
\[
\sA^R_{Y,D}\coloneqq \shD_{Y,D}\otimes_\C R\subseteq \sA^R_Y.
\]

The decomposition $Z=X\times \A_\C^l\times \A_\C^k$
gives $\C$-algebra inclusions
\[
\shD_X\hookrightarrow \shD_{Y,D}\hookrightarrow \shD_{Z,E},
\]
and hence induced inclusions of $\C$-algebras
\begin{equation}\label{eq:uncanincl}
    \sA_X^r\hookrightarrow \sA_{Y,D}^k\hookrightarrow \shD_{Z,E},
\end{equation}
where $s_i$ is identified with $-t_i\partial_{t_i}-1$ for $i=1,\dots,r$. More generally, analogously to \eqref{eq:uncanincl}, we also have
\begin{equation}\label{eq:uncaningen}
\sA^{R\otimes_\C R_l}_X
=
\sA_X^{R[s_1,\dots,s_l]}
\hookrightarrow
\sA^R_{Y,D},
\quad
s_i\mapsto -t_i\partial_i-1,
\quad
i=1,\dots,l.
\end{equation}
The inclusion of centers
\[
R\hookrightarrow R\otimes_\C R_l
\]
induces the smooth projection
\[
p_{\log}\colon \Spec R\otimes_\C R_l\longrightarrow \Spec R.
\]

Considering the associated analytic space of $\Spec R_l$, which we identify with $\C^l$, we have the universal covering map
\[
\operatorname{Exp}\colon \C^l\to (\C^*)^l,\quad
(\alpha_1,\dots,\alpha_l)\mapsto
(e^{-2\pi i\alpha_1},\dots,e^{-2\pi i\alpha_l}),
\]
where $(\C^*)^l$ is the affine complex torus. Geometrically, $\C^l$ is naturally the complex cotangent space to the fibers of the Kato--Nakayama space $Z^{\log}$ over $E^k$; cf. \cite[V.1.2]{Ogusbook}.

We will also use the algebraic structure $\Spec T_l$ on $(\C^*)^l$, where $T_l$ is the Laurent polynomial ring
\[
T_l=\C[\chi_1^{\pm1},\chi_2^{\pm1},\dots,\chi_l^{\pm1}].
\]
An irreducible algebraic subvariety $V$ of $\Spec R_l$ is called a \emph{linear subvariety} if $\operatorname{Exp}(V)$ is a subtorus of $(\C^*)^l$. Since ideals of subtori are generated by binomials in $T_l$, one sees that linear subvarieties are integer translates of linear subspaces of $\C^l$ defined over $\Q$. In particular, linear subvarieties are defined over $\Q$.

An algebraic subvariety $W\subseteq \C^l$ is called a \emph{translated linear subvariety} if $W=V+a$ for some linear subvariety $V$ and some $a\in \C^l$. Moreover, $W$ is called a \emph{$\Q$-translated linear subvariety} if, in addition, $a\in \Q^l$. Thus the $\operatorname{Exp}$-image of a translated linear subvariety is a translated subtorus, and the $\operatorname{Exp}$-image of a $\Q$-translated linear subvariety is a torsion-translated subtorus.

We denote by $F_\bullet\shD_Y$ the order filtration on $\shD_Y$. It induces the order filtration on $\shD_{Y,D}$:
\[
F^{\log}_\bullet\shD_{Y,D}
=
F_\bullet\shD_Y\cap \shD_{Y,D}.
\]
Then we have the induced relative filtrations on both $\sA^R_Y$ and $\sA^R_{Y,D}$:
\[
F_\bullet\sA_Y^R
=
F_\bullet\shD_Y\otimes R
\quad \textup{and} \quad
F^\dagger_\bullet\sA_{Y,D}^R
=
F^{\log}_\bullet\shD_{Y,D}\otimes R,
\]
so that
\[
\gr_\bullet\sA_Y^R
=
\gr_\bullet\shD_Y\otimes R
\quad \textup{and} \quad
\gr^\dagger_\bullet\sA_{Y,D}^R
=
\gr^{\log}_\bullet\shD_{Y,D}\otimes R.
\]
It follows that $\sA_Y^R$ and $\sA^R_{Y,D}$ are noetherian rings; cf. \cite[A.III 1.27]{BJ}. As in the classical theory of $\shD$-modules, for a finitely generated $\shD_{Y,D}$-module $\cM$, a good filtration with respect to $F^{\log}_\bullet\shD_{Y,D}$ gives the logarithmic characteristic variety
\[
\Ch^{\log}(\cM)\subseteq T^*(Y,D);
\]
cf. \cite[\S 2]{WZ}.

Similarly, one can consider good filtrations with respect to $F_\bullet\sA_Y^R$ and $F^\dagger_\bullet\sA_{Y,D}^R$, thereby obtaining characteristic varieties for $\sA^R_Y$-modules and $\sA^R_{Y,D}$-modules, respectively. For a finitely generated left or right $\sA_Y^R$-module, respectively $\sA_{Y,D}^R$-module, $\cP$, we use
\[
\Ch(\cP)
=
\Ch_{\sA_Y^R}(\cP)
\subseteq T^*Y\times \Spec R
\]
and
\[
\Ch^\dagger(\cP)
=
\Ch^\dagger_{\sA_{Y,D}^R}(\cP)
\subseteq T^*(Y,D)\times \Spec R
\]
to denote the characteristic variety, with respect to $F_\bullet\sA_Y^R$ and $F^\dagger_\bullet\sA_{Y,D}^R$ respectively. We denote the corresponding characteristic cycles by replacing $\Ch$ with $\CC$.

Let $\cM$ be a left or right $\sA^R_Y$-module, and let $\cN\subseteq \cM$ be a nonzero finitely generated $\sA^R_{Y,D}$-submodule. We say that $\cN$ is an $\sA^R_{Y,D}$-\emph{lattice} of $\cM$ if
\[
\cM\simeq \cN(*D)
\coloneqq
\cN\otimes_{\cO_Y}\cO_Y(*D),
\]
where $\cO_Y$ is the coordinate ring of $Y$ and $\cO_Y(*D)$ denotes its localization along $D$. Equivalently, $\cN$ is an $\sA^R_{Y,D}$-lattice if $\cN$ has no nonzero $\sA^R_{Y,D}$-submodules supported on $D$.

Moreover, if $H$ is a divisor supported on $D$, then its fractional ideal $\cO_Y(H)\subseteq \cO_Y(*D)$ is a left $\shD_{Y,D}$-module. Hence, by \cite[Lem.2.1]{WZ},
\[
\cN(H)\coloneqq \cN\otimes_{\cO_Y}\cO_Y(H)
\]
is a $\shD_{Y,D}$-module, and therefore also an $\sA^R_{Y,D}$-module, since $R$ lies in the center of $\sA^R_{Y,D}$. Consequently, $\cN(H)$ is also an $\sA^R_{Y,D}$-\emph{lattice} whenever $\cN$ is one. Moreover, if $H$ is effective, then we have a direct system
\[
\cN(pH)\hookrightarrow \cN((p+1)H),
\quad p\in \Z,
\]
such that, if $\supp(H)=D$, then
\[
\varinjlim_{p\to\infty}\cN(pH)=\cM.
\]
In particular, if $\cN\subseteq \cM$ is a lattice, then
\[
t_i^p\cdot \cN=\cN(-pD_i)\subseteq \cM
\]
is also a lattice for $i=1,\dots,l$ and $p\in \Z$.

\begin{lemma}\label{lm:lattictwist}
Let $\cN$ be an $\sA_{Y,D}^k$-lattice. Let
\[
H_1=\sum_{i=1}^l a_iD_i
\quad \text{and} \quad
H_2=\sum_{i=1}^l b_iD_i
\]
be two divisors supported on $D$, with $H_1$ effective and $D_i=(t_i=0)$. Then
\[
\frac{\cN}{\cN(-H_1)}(H_2)
\simeq
\frac{\prod_{i=1}^l t_i^{-b_i}\cdot \cN}
{\prod_{i=1}^l t_i^{a_i-b_i}\cdot \cN}.
\]
\end{lemma}

\begin{proof}
We have a short exact sequence
\[
0
\to
\cN(H_2-H_1)
\longrightarrow
\cN(H_2)
\longrightarrow
\frac{\cN}{\cN(-H_1)}(H_2)
\to
0.
\]
Since $\cN$ is a lattice, we have
\[
\cN(H_2)
\simeq
\prod_{i=1}^l t_i^{-b_i}\cdot \cN
\quad \textup{and} \quad
\cN(H_2-H_1)
\simeq
\prod_{i=1}^l t_i^{a_i-b_i}\cdot \cN.
\]
The assertion follows.
\end{proof}

\begin{lemma}\label{lm:twosidedideal}
If $K\subseteq \N^r$ is a monoid ideal, then
$\shD_{Z,E}\cdot \alpha(K)\subseteq \shD_{Z,E}$ is a two-sided ideal.
\end{lemma}

\begin{proof}
For $\gamma\in K$, $\alpha(\gamma)\in S[t_1,\dots,t_r]$ is a monomial in the $t_i$. By definition of $\shD_{Z,E}$, for every $P\in \shD_{Z,E}$ there exists $Q\in \shD_{Z,E}$ such that
\[
P\cdot \alpha(\gamma)=\alpha(\gamma)\cdot Q.
\]
Thus, $\shD_{Z,E}\cdot \alpha(K)$ is a two-sided ideal.
\end{proof}

\begin{lemma}\label{lm:logsp}
We have a natural isomorphism of $\C$-algebras
\[
i_Y^*(\shD_{Z,E})
=
\frac{\shD_{Z,E}}{\shD_{Z,E}\cdot(t_{l+1},\dots,t_r)}
\simeq
\sA^k_{Y,D},
\]
such that the inclusion $\sA^k_{Y,D}\hookrightarrow \shD_{Z,E}$ in \eqref{eq:uncanincl} is a section of the natural $\C$-algebra morphism
\[
\shD_{Z,E}\longrightarrow i_Y^*(\shD_{Z,E}).
\]
Moreover, we have the induced filtration on $\sA^k_{Y,D}$:
\[
F^\sharp_\bullet\sA^k_{Y,D}
\simeq
i_Y^*(F^{\log}_\bullet\shD_{Z,E}),
\]
so that
\[
\gr^\sharp_\bullet\sA^k_{Y,D}
\simeq
\gr^{\log}_\bullet\shD_{Y,D}\otimes_\C \C[s_{l+1},\dots,s_r].
\]
In particular, when $k=r$, we have
\[
i_X^*(\shD_{Z,E})\simeq \sA_X^r
\quad \textup{and} \quad
F^\sharp_p\sA_X^r
=
\sum_{q+|\mathbf{i}|=p}F_q\shD_X\cdot \mathbf{s}^{\mathbf{i}},
\]
where
\[
i_X\colon X\hookrightarrow Z,\quad x\mapsto (x,0,\dots,0)
\]
is the closed embedding.
\end{lemma}

\begin{proof}
The first isomorphism is obtained by identifying $s_i$ with
$-t_i\partial_{t_i}-1$ for $i=l+1,\dots,r$, following the convention in \eqref{eq:uncanincl}. The quotient $i_Y^*(\shD_{Z,E})$ is a ring because the ideal
\[
\shD_{Z,E}\cdot(t_{l+1},\dots,t_r)
\]
is two-sided by Lemma \ref{lm:twosidedideal}. The statements concerning the induced $\sharp$-filtration on $\sA^k_{Y,D}$ follow similarly from the construction.
\end{proof}

By the above lemma, if $\cP$ is a finitely generated $\sA^k_{Y,D}$-module, then taking a good filtration of $\cP$ over $F^\sharp_\bullet\sA^k_{Y,D}$ defines its $\sharp$-characteristic variety
\[
\Ch^\sharp(\cP)
\subseteq
T^*(Y,D)\times \A_\C^k
=
T^*(Z,E)|_Y.
\]
When $k=0$ and $D$ is empty, $\Ch^\sharp$, $\Ch^\dagger$, and $\Ch^{\log}$ all reduce to the usual characteristic varieties of $\shD_Y$-modules.

Let us remark that, for $\sA_X^{R\otimes_\C R_l}$-modules $\cN$, we can consider two types of good filtrations:
\[
F^\dagger_\bullet \sA_X^{R\otimes_\C R_l}
=
F^\dagger_\bullet\sA^R_{Y,D}\cap \sA_X^{R\otimes_\C R_l},
\quad
F_\bullet \sA_X^{R\otimes_\C R_l}
=
F_\bullet\shD_X\otimes_\C (R\otimes_\C R_l),
\]
so that
\[
\gr^\dagger_\bullet(\sA_X^{R\otimes_\C R_l})
\simeq
\gr^\sharp_\bullet(\sA_X^l)\otimes_\C R,
\quad
\gr_\bullet(\sA_X^{R\otimes_\C R_l})
\simeq
\gr_\bullet(\shD_X)\otimes_\C (R\otimes_\C R_l).
\]
Thus, we obtain two characteristic varieties,
\[
\Ch^\dagger_{\sA_X^{R\otimes_\C R_l}}(\cN)
\quad \textup{and} \quad
\Ch_{\sA_X^{R\otimes_\C R_l}}(\cN),
\]
both of which will be used later.

Recall that if $\cM$ is a nonzero finitely generated $\sA^R_{Y,D}$-module, then the grade of $\cM$ is
\[
j(\cM)
\coloneqq
\min\left\{
q \mid
\operatorname{Ext}^q_{\sA^R_{Y,D}}(\cM,\sA^R_{Y,D})\neq 0
\right\}.
\]
If $\cM=0$, then we define $j(\cM)=+\infty$ by convention.

\begin{lemma}\label{lm:ausdual}
Let $\cM$ be a nonzero finitely generated $\sA^R_{Y,D}$-module. Then:
\begin{enumerate}[label=$(\arabic*)$]
    \item $j(\cM)+\dim\Ch^\dagger(\cM)=2(n+l)+\dim R$;
    \item $j\big(\operatorname{Ext}^q_{\sA^R_{Y,D}}(\cM,\sA^R_{Y,D})\big)\ge q$ for all $q$.
\end{enumerate}
Moreover, if $\cM$ is a nonzero finitely generated $\sA^k_{Y,D}$-module, then
\[
j(\cM)+\dim\Ch^\sharp(\cM)=2(n+l)+k.
\]
\end{lemma}

\begin{proof}
Since
\[
\gr^\dagger_\bullet(\sA^R_{Y,D})
=
\gr^{\log}_\bullet(\shD_{Y,D})\otimes_\C R
\]
is a regular commutative ring, \cite[A.IV 4.15]{BJ} implies that $\sA^R_{Y,D}$ is Auslander regular. Hence part $(2)$ follows. Part $(1)$ follows from \cite[A.IV 4.10 and 3.5]{BJ}, and the same argument applies to $\Ch^\sharp$.
\end{proof}

\begin{lemma}\label{lm:twistCHsharp}
If $\cP$ is a finitely generated $\sA^R_{Y,D}$-module and $H$ is a divisor supported on $D$, then
\[
\Ch^\dagger(\cP)
=
\Ch^\dagger\big(\cP(H)\big).
\]
In particular, if $\cN_1$ and $\cN_2$ are two $\sA^R_{Y,D}$-lattices such that
$\cN_1(*D)\simeq \cN_2(*D),$
then
\[
\Ch^\dagger(\cN_1)=\Ch^\dagger(\cN_2).
\]
If, moreover, $\cP$ is a finitely generated $\sA^k_{Y,D}$-module, then we also have
\[
\Ch^\sharp(\cP)
=
\Ch^\sharp\big(\cP(H)\big).
\]
\end{lemma}

\begin{proof}
Since $\cO_Y(H)$ is locally free, if $F^\dagger_\bullet\cP$ is a good filtration of a finitely generated $\sA^R_{Y,D}$-module $\cP$, then
\[
F^\dagger_\bullet\cP\otimes_{\cO_Y}\cO_Y(H)
\]
is a good filtration of $\cP(H)$, and we have
\[
\gr^\dagger_\bullet\big(\cP(H)\big)
=
\gr^\dagger_\bullet(\cP)\otimes_{\cO_Y}\cO_Y(H).
\]
Thus,
\[
\Ch^\dagger(\cP)
=
\Ch^\dagger\big(\cP(H)\big).
\]
The same argument applied to $F^\sharp_\bullet$ gives the corresponding statement for $\Ch^\sharp$.

To prove
\[
\Ch^\dagger(\cN_1)=\Ch^\dagger(\cN_2),
\]
it suffices to observe that, for some $m\gg 0$,
\[
\cN_1(-mD)\subseteq \cN_2\subseteq \cN_1(mD).
\]
The assertion then follows from the first part of the lemma.
\end{proof}

We need the following key example of log $\shD$-modules.

\begin{example}\label{ex:maingrebex}
Let
\[
F=(f_1,f_2,\dots,f_r)\colon X\to \A_\C^r
\]
be a morphism. Write
\[
f=\prod_{i=1}^r f_i
\quad \textup{and} \quad
F^{\bs}=\prod_{i=1}^r f_i^{s_i}.
\]
Then we have a left $\sA_X^r$-module
\[
\cS_F
\coloneqq
\cO_X[1/f,s_1,\dots,s_r]\cdot F^{\bs}
\]
and submodules
\[
\cN_F
\coloneqq
\shD_X[s_1,\dots,s_r]\cdot \prod_{i=1}^r f_i^{s_i}
\quad \textup{and} \quad
\cN_F^v
\coloneqq
\shD_X[s_1,\dots,s_r]\cdot \prod_{i=1}^r f_i^{s_i+v_i}
\]
for $v=(v_1,\dots,v_r)\in \Z^r$. Clearly,
\[
\cN_F^{v'}\subseteq \cN_F^v
\]
if $v_i\le v_i'$ for all $i$. We also have the graph embedding
\[
\iota_F\colon X\rightarrow Z=X\times \A_\C^r,
\quad
x\mapsto (x,f_1(x),\dots,f_r(x)).
\]
Following \cite{MalV}; see also \cite[Prop.5.4]{ChDMu}, we have a $\shD_Z$-module isomorphism
\begin{equation}\label{be:graphemD}
\iota_{F,*}
\big(
\cO_X[1/f,s_1,\dots,s_r]\cdot F^{\bs}
\big)
\simeq
\iota_{F,+}\big(\cO_X[1/f]\big),
\end{equation}
under which $s_i$ is identified with $-\partial_{t_i}t_i$.\footnote{Some authors take $s_i=-t_i\partial_{t_i}$, for instance, \cite[Prop.A3.1]{Gil1}. This makes no substantial difference.}
Here, $\iota_{F,+}$ denotes the $\shD$-module pushforward functor, while $\iota_{F,*}$ denotes the sheaf pushforward functor. Since $\iota_F$ is a closed embedding, we will, to simplify notation, not distinguish between $\iota_{F,*}(\star)$ and $\star$ when $\star$ is a sheaf defined on $X$; the same convention will be used for other closed embeddings hereafter. Consequently,
\[
\cN_F^v
=
\prod_i t_i^{v_i}\cdot \cN_F
\]
is a $\shD_{Z,E}$-lattice of $\cS_F$.
\end{example}
It is worth pointing out that, when we want to emphasize the use of the closed embedding
$i_Y\colon Y\hookrightarrow Z$, modules over $i_Y^*(\shD_{Z,E})$ and modules over
$\sA^k_{Y,D}$ should be conceptually distinguished, although we have a ring
isomorphism
\[
i_Y^*(\shD_{Z,E})\simeq \sA^k_{Y,D}
\]
by Lemma \ref{lm:logsp}. For instance, $\cN_F$ is naturally an $\sA_X^r$-module,
but not an $i_X^*(\shD_{Z,E})$-module, since it is supported on the graph of $F$,
rather than on the image of $i_X$. This motivates the following definition.

\begin{definition}\label{def:getype}
Let $\cM$ be an $\sA^R_{Y,D}$-module, and let
\[
G=(g_1,\dots,g_l)\colon X\rightarrow \A_\C^l
\]
be a regular map. We say that $\cM$ is of \emph{graph-embedding type} with
respect to $G$ (or of \emph{GE-type} for short) if $\cM$ is supported on
$\iota_G(X)$, the graph of $G$, where
\[
\iota_G\colon X\hookrightarrow Y=X\times \A_\C^l
\]
is the graph embedding.
\end{definition}

If $\cM$ is an $\sA^R_{Y,D}$-module of graph-embedding type, then $\cM(*D)$ is
an $\sA_Y^R$-module supported on the graph of $G$. Since $\sA_Y^R$-modules are,
in particular, $\shD_Y$-modules, Kashiwara's equivalence
\cite[Thm.1.6.1]{HTT} (see also \cite[Thm.1.5]{FS17} for the relative
version), gives, similarly to \eqref{be:graphemD}, $\sA_Y^R$-module isomorphisms
\begin{equation}\label{eq:getypeD}
\cM(*D)
\simeq
\iota_{G,+}\cS
\simeq
\cS\left[1/\prod_{i=1}^l g_i,s_1,\dots,s_l\right]\cdot
\prod_{i=1}^l g_i^{s_i},
\end{equation}
where
\[
\cS
=
\bigcap_{i=1}^l
\ker\left(
\cM(*D)\xrightarrow{\,t_i-g_i\,}\cM(*D)
\right).
\]
Here $\cS$ is an $\sA_X^R$-module satisfying
\[
\cS=\cS\left[1/\prod_{i=1}^l g_i\right].
\]
Since $s_i=-\partial_{t_i}t_i$, the isomorphism
\[
\cM(*D)
\simeq
\cS\left[1/\prod_{i=1}^l g_i,s_1,\dots,s_l\right]\cdot
\prod_{i=1}^l g_i^{s_i}
\]
is also an $\sA_X^{R\otimes_\C R_l}$-module isomorphism via the inclusion
\[
\sA_X^{R\otimes_\C R_l}\hookrightarrow \sA^R_{Y,D}.
\]

\begin{lemma}\label{lm:GE-type}
Let $\cN$ be an $\sA^R_{Y,D}$-lattice of GE-type with respect to $G$.
Then $\cN$ is finitely generated over $\sA_X^{R\otimes_\C R_l}$. Moreover, under
the embedding $d\iota_G$, we have
\[
d\iota_G\left(
\Ch^\dagger_{\sA_X^{R\otimes_\C R_l}}(\cN)
\right)
=
\Ch^\dagger_{\sA^R_{Y,D}}(\cN).
\]
\end{lemma}

\begin{proof}
Using \eqref{eq:getypeD}, choose elements $e_1,\dots,e_m\in \cS$ that generate
$\cN|_{Y\setminus D}$ over $\sA^R_{Y\setminus D}$. These elements generate
another $\sA^R_{Y,D}$-lattice, say $\cN_1$. Since the action of $t_i$ on $\cS$
can be replaced by the action of $g_i$, the elements $e_1,\dots,e_m$ also
generate $\cN_1$ over $\sA_X^{R\otimes_\C R_l}$.

Define
\[
F_p^\dagger\cN_1
=
\sum_i F_p^\dagger\sA_X^{R\otimes_\C R_l}\cdot e_i
=
\sum_i F_p^\dagger\sA_{Y,D}^R\cdot e_i.
\]
This gives a good filtration both over
$F_\bullet^\dagger\sA_X^{R\otimes_\C R_l}$ and over
$F_\bullet^\dagger\sA_{Y,D}^R$.

The graph embedding $\iota_G$ induces the map
\[
d\iota_G\colon
T^*X\times \A_s^l\times \Spec R
\hookrightarrow
T^*(Y,D)\times \Spec R
\simeq
T^*X\times \A_t^l\times \A_{\tilde\xi}^l\times \Spec R
\]
given by
\[
(x,\xi,s,a)\mapsto (x,\xi,G(x),-s,a),
\]
where $\tilde\xi$ denotes the fiber coordinates of $T^*(Y,D)$ corresponding to
the symbols of $t_i\partial_{t_i}$. By construction, we therefore have
\[
d\iota_G\left(
\Ch^\dagger_{\sA_X^{R\otimes_\C R_l}}(\cN_1)
\right)
=
\Ch^\dagger_{\sA^R_{Y,D}}(\cN_1).
\]

Since $\cN$ and $\cN_1$ are both $\sA^R_{Y,D}$-lattices in $\cM$, there exists
$k\gg 0$ such that
$\cN\subseteq \cN_1(kD).$
Using the relation
$t_is_i=(s_i+1)t_i,$
we see that $\cN_1(kD)$ is finitely generated over
$\sA_X^{R\otimes_\C R_l}$. Since $\sA_X^{R\otimes_\C R_l}$ is noetherian, it
follows that $\cN$ is finitely generated over $\sA_X^{R\otimes_\C R_l}$.
Finally, the equality of characteristic varieties follows from
Lemma \ref{lm:twistCHsharp}.
\end{proof}

\subsection{Pure filtrations of Gabber--Kashiwara}

Let $\cM$ be a finitely generated $\sA^R_{Y,D}$-module. Since
$\sA^R_{Y,D}$ is Auslander regular, \cite[A.IV 2.11]{BJ} implies that there
exists a unique finite filtration
\[
\cM=T_0\cM\supseteq T_1\cM\supseteq \cdots
\supseteq T_p\cM\supseteq T_{p+1}\cM=0
\]
such that, whenever $T_i\cM/T_{i+1}\cM\neq 0$, the quotient
$T_i\cM/T_{i+1}\cM$ is $i$-pure. This filtration is called the
\emph{pure filtration of Gabber--Kashiwara}, or simply the \emph{pure filtration}.

\begin{lemma}\label{lm:purfilcodim}
With notation as above, we have
\[
T_i\cM
=
\left\{
m\in \cM
\ \middle|\
\codim\big(\Ch^\dagger(\sA^R_{Y,D}\cdot m)\big)\ge i
\right\}.
\]
\end{lemma}

\begin{proof}
This is the logarithmic analogue of \cite[Thm.2.19]{Kasbook}. With the help of
Lemma \ref{lm:ausdual}, the same argument gives the required statement.
\end{proof}

Since characteristic varieties are compatible with localization, the preceding
lemma implies that taking the pure filtration of $\cM$ commutes with localization
as $\cO_Y$-modules. Using Lemma \ref{lm:purfilcodim}, we also see that $\cM$ is
$j$-pure if and only if
\[
T_i\cM=\cM \quad \text{for } i\le j,
\qquad
T_i\cM=0 \quad \text{for } i>j.
\]
We immediately obtain the following consequence.

\begin{coro}\label{cor:sesk-pure}
Let
\[
0\to \cM_1\to \cM\to \cM_2\to 0
\]
be a short exact sequence of finitely generated $\sA_{Y,D}^R$-modules. If
$\cM_1$ and $\cM_2$ are $k$-pure, then $\cM$ is also $k$-pure.
\end{coro}

Similarly to \cite[Thm.2.23]{KSbook}, we have the following alternative
description of the pure filtration, obtained by the same argument.

\begin{lemma}\label{lm:dualalPF}
We have
\[
T_i\cM
=
\operatorname{Ext}^{0}_{\sA^R_{Y,D}}
\left(
\tau^{\ge i}
\Rhom_{\sA^R_{Y,D}}(\cM,\sA^R_{Y,D}),
\sA^R_{Y,D}
\right).
\]
\end{lemma}
\subsection{Maximal tame pure extensions of lattices}\label{subsect:mtamepureex}

Let $\cN$ be a pure $\sA_{Y,D}^R$-lattice with grade number $j(\cN)=p$.
For $\bm\in\N^l$, we denote by $D_\bm$ the divisor defined by
$\alpha(\bm)=0$. We consider the family of all finitely generated
$\sA_{Y,D}^R$-submodules $\cN'$ of $\cN(*D_\bm)$ such that
\[
\cN'\supseteq \cN
\quad \text{and} \quad
j(\cN'/\cN)\ge p+2.
\]
By \cite[A.IV 2.10]{BJ}, this family has a unique maximal element
$\cL=\cL_\bm$, called the \emph{maximal tame pure extension} of $\cN$
inside $\cN(*D_\bm)$. The following lemma is the logarithmic generalization
of \cite[Prop.15]{Mai}; see also \cite[Lem.4.4.1]{BVWZ2}.

\begin{lemma}\label{lm:tamepurechv}
Let $H$ be a divisor supported on $D$. For nonzero $\bm\in\N^l$, the module
$\cL(H)$ is the maximal tame pure extension of $\cN(H)$ inside
$\cN(H)(*D_\bm)$. Moreover, if
$j\big(\cN/\alpha(\bm)\cdot \cN\big)=p+1$,
then 
$\cL(H)/\alpha(\bm)\cdot \cL(H)$
is $(p+1)$-pure, and the top-dimensional part of $\Ch^\dagger\big(\cN/\alpha(\bm)\cdot \cN\big)$
is
$\Ch^\dagger\big(\cL/\alpha(\bm)\cdot \cL\big)$.
\end{lemma}

\begin{proof}
It is clear that $\cN(H)$ is also $p$-pure. Let $\cL'$ denote the maximal
tame pure extension of $\cN(H)$ inside $\cN(H)(*D_\bm)$. By
Lemma \ref{lm:lattictwist} and Lemma \ref{lm:twistCHsharp}, we have
\[
\Ch^\dagger\big(\cL'(-H)/\cN\big)
=
\Ch^\dagger\big(\cL'/\cN(H)\big).
\]
Then Lemma \ref{lm:ausdual}(1) gives
\[
j\big(\cL'(-H)/\cN\big)\ge p+2.
\]
By maximality, we obtain $\cL'(-H)\subseteq \cL$. By symmetry, we also have
$\cL(H)\subseteq \cL'$, and hence $\cL'=\cL(H)$. In particular,
\[
j\big(\cL(H)/\cN(H)\big)\ge p+2.
\]

We next prove purity modulo $\alpha(\bm)$. Suppose that
$\cL/\alpha(\bm)\cdot\cL$ is not $(p+1)$-pure. Then there exists an
intermediate submodule
$\alpha(\bm)\cdot\cL
\subsetneq
\cL''
\subsetneq
\cL$
such that
\[
j\big(\cL''/\alpha(\bm)\cdot\cL\big)\ge p+2.
\]
By \cite[A.IV 2.3]{BJ}, we then have
\[
j\big(\cL''/\alpha(\bm)\cdot\cN\big)\ge p+2.
\]
Thus $\alpha(\bm)^{-1}\cL''$ belongs to the family defining $\cL$ and
strictly contains $\cL$, contradicting the maximality of $\cL$. Therefore
$\cL/\alpha(\bm)\cdot\cL$ is $(p+1)$-pure. Twisting by $H$, we conclude that
$\cL(H)/\alpha(\bm)\cdot\cL(H)$
is also $(p+1)$-pure.

It remains to prove that the top-dimensional part of
$\Ch^\dagger\big(\cN/\alpha(\bm)\cdot\cN\big)$
is
\[
\Ch^\dagger\big(\cL/\alpha(\bm)\cdot\cL\big).
\]
To simplify notation, assume that $\bm=e_i$, so that $\alpha(\bm)=t_i$.
The general case is similar.

Consider the short exact sequence
\[
0
\to
\frac{t_i\cL\cap \cN}{t_i\cN}
\longrightarrow
\frac{\cN}{t_i\cN}
\longrightarrow
\frac{\cN}{t_i\cL\cap \cN}
\to
0.
\]
Since $t_i\cL$ is the tame pure extension of $t_i\cN$, we have
\[
j\left(\frac{t_i\cL\cap \cN}{t_i\cN}\right)\ge p+2.
\]
By Lemma \ref{lm:ausdual}, the above short exact sequence implies that the
top-dimensional part of $\Ch^\dagger(\cN/t_i\cN)$ is the same as that of
$\Ch^\dagger\left(\frac{\cN}{t_i\cL\cap \cN}\right).$

Now consider the diagram
\[
\begin{tikzcd}
0\arrow[r] &
t_i\cN\arrow[r]\arrow[d] &
t_i\cL \arrow[r]\arrow[d] &
t_i\cL/t_i\cN\arrow[r]\arrow[d] &
0\\
0\arrow[r] &
\cN\arrow[r] &
\cL \arrow[r] &
\cL/\cN\arrow[r] &
0 .
\end{tikzcd}
\]
It induces a short exact sequence
\[
0
\to
\frac{\cN}{t_i\cL\cap\cN}
\longrightarrow
\frac{\cL}{t_i\cL}
\longrightarrow
\cQ
\to
0.
\]
By the snake lemma, we also have an exact sequence
\[
t_i\cL/t_i\cN
\longrightarrow
\cL/\cN
\longrightarrow
\cQ
\to
0.
\]
Since $j(\cQ)\ge p+2$, Lemma \ref{lm:ausdual} again implies that the
top-dimensional part of $\Ch^\dagger(\cL/t_i\cL)$ is the same as that of
\[
\Ch^\dagger\left(\frac{\cN}{t_i\cL\cap\cN}\right).
\]
Since $\cL/t_i\cL$ is $(p+1)$-pure, the assertion follows from
\cite[A.IV 4.11 and 3.7]{BJ}.
\end{proof}

\section{Analytic and algebraic relative regular holonomic $\shD$-modules}
\label{sect:GAGArel}

\subsection{Relative $\C$-constructible sheaves}
\label{subsect:relconstrsh}

We keep the notation of \S \ref{subsect:logDmodule}. We use
$D^b_{\C-c}(X^\an,R)$ to denote the derived category of
$\C$-constructible complexes of $R$-modules on the underlying complex
manifold $X^\an$ of $X$; cf. \cite[Def.8.5.6]{KSbook}. A complex
$F\in D^b_{\C-c}(X^\an,R)$ is called \emph{algebraic} if its singular
support
\[
\operatorname{SS}(F)\subseteq T^*X
\]
is an algebraic conic Lagrangian subvariety. We denote by $D^b_c(X_R)$ the
full triangulated subcategory consisting of algebraic $\C$-constructible
complexes of $R$-modules.

For $F\in D^b_{\C-c}(X^\an,R)$, we use
\[
\supp(F)=\supp_X^R(F)\subseteq \Spec R
\]
to denote its $R$-module support. By
\cite[Thm.8.5.5]{KSbook}, this support is locally algebraic, that is, for a relatively compact
neighborhood $U\subseteq X^\an$, the support
\[
\supp(F|_U)\subseteq \Spec R
\]
is an algebraic subvariety. If $F\in D^b_c(X_R)$, then
$\supp(F)\subseteq \Spec R$ is globally algebraic by definition.

For $F\in D^b_{\C-c}(X^\an,R)$, we use $\wt F$ to denote its analytic
sheafification; cf. \cite[\S 2.2]{WuRHA}. By construction and
\cite[Properties 2.20(5)]{FS13}, the complex $\wt F$ is
$\Specan R$-$\C$-constructible.

\subsection{Analytic relative regular holonomic $\shD$-modules}
\label{subsect:relregularhol}

Let $\cM$ be an $\sA_X^R$-module. We use $\widetilde\cM$ to denote the
associated analytic sheaf of $\cM$. More precisely, since $\cM$ is, in
particular, an $\cO_X\otimes_\C R$-module, we first regard it as the
associated algebraic sheaf on
\[
X\times_{\Spec \C}\Spec R
\]
as in \cite[Ch. II, Prop. 5.1]{Hart}; then $\widetilde\cM$ is the analytic
sheaf associated with this algebraic sheaf by GAGA. See also
\cite[\S 2.3]{WuRHA}. For instance,
\[
\widetilde{\sA_X^r}\simeq \shD_{X\times \C^r/\C^r},
\]
where $\shD_{X\times \C^r/\C^r}$ is the sheaf of analytic relative
differential operators over $\C^r$; cf. \cite{FS17}. Thus
$\widetilde\cM$ is a
\[
\widetilde{\sA_X^R}
\simeq
\shD_{X\times \Specan R/\Specan R}
\]
module, where $\Specan R$ and $X\times \Specan R$ denote the analytic
spaces associated with $\Spec R$ and $X\times_{\Spec\C}\Spec R$,
respectively.

Let $\Omega=\Spec R_\Omega\subseteq \Spec R$ be an affine open subset.
The \emph{algebraic relative localization} of $\cM$ over $\Omega$ is
\[
\cM|_\Omega
\coloneqq
\cM\otimes_R R_\Omega,
\]
which is an $\sA_X^{R_\Omega}$-module. On the other hand, if
$U\subseteq \Specan R$ and $U_X\subseteq X^\an$ are open subsets in the
Euclidean topology, then the \emph{analytic relative localization} of
$\cM$ over $U$ is the $\shD_{U_X\times U/U}$-module
\[
\widetilde\cM|_{U_X\times U}.
\]

The distinction between these two localization procedures is important in
this paper. For instance, $\cS_F$ from Example \ref{ex:maingrebex}, as well
as its algebraic relative localization over any affine open neighborhood
$\Omega$, is not finitely generated if $f$ is not invertible. However,
$\widetilde{\cS_F}$ and all of its analytic relative localizations are
coherent; see \cite[Remark 3.4]{WuRHA}. Motivated by this example, we say
that a coherent $\wt{\sA_X^R}$-module $\sM$ is \emph{algebraic} if
$\sM\simeq \wt{\cN}$
for some finitely generated $\sA_X^R$-module $\cN$.
Furthermore, we say that an $\sA_X^R$-module $\cM$ is
\emph{analytic relative regular holonomic} if $\widetilde\cM$ is a regular
holonomic $\widetilde{\sA_X^R}$-module; cf. \cite[\S 3.a]{FFS23}. The
subcategory of relative regular holonomic $\sA_X^R$-modules is a thick
abelian subcategory by \cite[Prop.3.2(iii)]{FFS23}. By definition, a
relative regular holonomic $\sA_X^R$-module $\cM$ is not necessarily
finitely generated over $\sA_X^R$; equivalently, $\wt{\cM}$ is not
necessarily \emph{algebraic}. Therefore, relative regular holonomic
$\sA_X^R$-modules are not necessarily algebraic relative holonomic
$\sA_X^R$-modules in the sense of \cite[Def.3.2.3]{BVWZ}.

We also introduce the following weak form of algebraicity for relative
regular holonomic modules, using the relative regular Riemann--Hilbert
correspondence of \cite{FFS23}.

\begin{definition}\label{def:dralgebraic}
Let $\sM$ be a relative regular holonomic $\wt{\sA_X^R}$-module. We say that
$\sM$ is \emph{dR-algebraic}, respectively \emph{weakly dR-algebraic}, if
the relative de Rham complex $\DR(\sM)$ satisfies
\[
\DR(\sM)\simeq \wt F
\]
for some $F\in D^b_c(X_R)$, respectively for some
$F\in D^b_{\C-c}(X^\an,R)$.
\end{definition}

Recall the duality functor from \cite{KSbook}:
\[
\mathbf D_R(F)
\coloneqq
\cRhom_R(F,R[2n]),
\]
for $F\in D^b_{\C-c}(X^\an,R)$. We also have the duality functor
$\mathbf D$ for $(\Spec R)$-$\C$-constructible complexes; cf.
\cite[\S 2.6]{FS13}. By construction, we have the following.

\begin{lemma}\label{lm:twistdual}
For $F\in D^b_{\C-c}(X^\an,R)$, there is a natural isomorphism
\[
\mathbf D(\wt F)\simeq \wt{\mathbf D_R(F)}.
\]
\end{lemma}

If $\sM$ is a coherent $\wt{\sA_X^R}$-module, then we denote by
\[
\Ch(\sM)
=
\Ch_{\wt{\sA_X^R}}(\sM)
\subseteq
T^*X^\an\times \Specan R
\]
its analytic relative characteristic variety, and by
\[
\supp(\sM)
=
\supp_{\Specan R}(\sM)
\subseteq
\Specan R
\]
its analytic relative support; cf. \cite[Prop.-Def.5.7]{WuRHA}. If,
moreover, $\sM$ is relative holonomic over $\wt{\sA_X^R}$, then
\[
\supp(\sM)
=
\operatorname{pr}\big(\Ch(\sM)\big)
\subseteq
\Specan R
\]
is locally an analytic subvariety by \cite[Cor.5.10]{WuRHA}.

\begin{remark}\label{rmk:analyticpurefil}
Lemma \ref{lm:purfilcodim} and Lemma \ref{lm:dualalPF} also hold in the
analytic setting; that is, they hold for coherent $\wt{\sA_{Y,D}^R}$-modules
and coherent $\wt{\sA_X^R}$-modules.
\end{remark}

\begin{lemma}\label{lm:dralg}
If $\sM$ is relative regular holonomic over $\wt{\sA_X^R}$ and
dR-algebraic, then $\supp(\sM)\subseteq \Specan R$ is an algebraic
subvariety. If $\sM$ is only weakly dR-algebraic, then $\supp(\sM)$ is
locally algebraic over $X^\an$.
\end{lemma}

\begin{proof}
Assume that
\[
\DR(\sM)\simeq \wt F
\]
for some $F\in D^b_{\C-c}(X^\an,R)$. By the relative regular
Riemann--Hilbert correspondence \cite{FFS23}, we have
\[
\supp(\sM)
=
\supp(\wt F)
\coloneqq
\left\{
m\in \Specan R
\,\middle|\,
\wt F_{(x,m)}\neq 0
\text{ for some } x\in X^\an
\right\}.
\]
Since the analytification functor is faithfully flat by \cite[Cor.1]{Mai}
and \cite[Lem.5.6]{WuRHA}, we have
$\supp(\wt F)=\supp(F).$
The claim follows because $\supp(F)$ is locally algebraic by
constructibility, and globally algebraic if $F\in D^b_c(X_R)$.
\end{proof}

Let
\[
G=(g_1,\dots,g_l)\colon X\to \A_\C^l
\]
be a regular map. We denote by
\[
\wt{\iota_G}\colon
X\times \Specan R
\hookrightarrow
Y\times \Specan R,
\quad
(x,a)\mapsto (x,G(x),a),
\]
the associated analytic map. Let $\sN$ be a coherent
$\wt{\sA_{Y,D}^R}$-module such that $\sN(*\wt D)$ is a coherent
$\wt{\sA_Y^R}$-module, where
\[
\wt D=D\times \Specan R
\subseteq
Y\times \Specan R
\]
is the associated relative divisor. We say that $\sN$ is a
$\wt{\sA_{Y,D}^R}$-\emph{lattice} if $\sN$ is a submodule of
$\sN(*\wt D)$.

Similarly to Definition \ref{def:getype}, we say that $\sN$ is of GE-type
with respect to $G$ if it is supported on the image of $\wt{\iota_G}$. By
relative Kashiwara equivalence \cite[Thm.1.5]{FS17}, and similarly to
\eqref{eq:getypeD}, we have
\[
\sN(*\wt D)
\simeq
\wt{\iota_G}_+\sS
\simeq
\wt{\iota_G}_*
\left(
\sS[s_1,\dots,s_l]\cdot \prod_{i=1}^l g_i^{s_i}
\right),
\]
where
\[
\sS
=
\bigcap_{i=1}^l
\ker
\left(
\sN(*\wt D)
\xrightarrow{\,t_i-g_i\,}
\sN(*\wt D)
\right).
\]
Here $\sS$ is a $\wt{\sA_X^R}$-module satisfying
$\sS=\sS\left[1/\prod_{i=1}^l g_i\right].$
Since $s_i=-\partial_{t_i}t_i$ and $\wt{\iota_G}$ is a closed embedding, the
isomorphism
\begin{equation}\label{eq:relgetype}
\sN(*\wt D)
\simeq
\sS\left[
1/\prod_{i=1}^l g_i,
s_1,\dots,s_l
\right]\cdot
\prod_{i=1}^l g_i^{s_i}
\end{equation}
is also an $\wt{\sA_X^R}\otimes_\C R_l$-module isomorphism.

\begin{theorem}[\cite{FFS23}, Example 4.36]\label{thm:mainFFS}
Let
\[
F=(f_1,\dots,f_r)\colon X\to \A_\C^r
\]
be a regular map as in Example \ref{ex:maingrebex}. If $\cG$ is a regular
holonomic $\shD_X$-module, then
$\cG\otimes_{\cO_X}\cS_F$
is a relative regular holonomic $\sA_X^r$-module. Consequently, as a
submodule of $\cG\otimes_{\cO_X}\cS_F$,
\[
\cG\otimes_{\cO_X}\cN_F
\]
is also relative regular holonomic.
\end{theorem}

For $\bn\in \N^r$, denote by $E_\bn$ the divisor defined by
$\alpha(\bn)=0$. If $\cN$ is a $\shD_{Z,E}$-module, then
$\cN(*E_\bn)$ is naturally a $\shD_{Z,E-\supp(E_\bn)}$-module. We call this
the \emph{partial localization} of $\cN$ along $\bn$. For instance,
\[
\cN_F\left[1/\prod_{i=1}^l f_i\right]
\]
is a partial localization.

\begin{coro}\label{cor:factorthroughgraph}
With notation as in Theorem \ref{thm:mainFFS}, let $\cG$ be a regular
holonomic $\shD_X$-module. Then
\[
\cG\otimes_{\cO_X}\cS_F
\quad \text{and} \quad
\cG\otimes_{\cO_X}
\cN_F\left[1/\prod_{i=1}^l f_i\right]
\]
are relative regular holonomic $\sA_Y^k$-modules for every $0\le k\le r$,
where $l=r-k$.
\end{coro}

\begin{proof}
Write
\[
G_1=(f_1,\dots,f_l)
\quad \text{and} \quad
G_2=(\tilde f_{l+1},\dots,\tilde f_r),
\]
where $\tilde f_i$ is the pullback of $f_i$ under the natural projection
\[
Y=X\times \A_\C^l\to X,
\]
with $l=r-k$. We have the graph embedding
$\iota_{G_1}\colon X\longrightarrow Y.$
Then, set
\[
\cH
=
\iota_{G_1,+}
\left(
\cG\left[1/\prod_{i=1}^l f_i\right]
\right).
\]
Equivalently, under the graph-embedding description,
\[
\cH
\simeq
\cG\left[
1/\prod_{i=1}^l f_i,
s_1,\dots,s_l
\right]\cdot
\prod_{i=1}^l f_i^{s_i}.
\]
Then we have the decomposition
\[
\cG\otimes_{\cO_X}\cS_F
\simeq
\cH\otimes_{\cO_Y}\cS_{G_2}.
\]
Applying Theorem \ref{thm:mainFFS} to $\cH$ and $G_2$ gives the result for
$\cG\otimes_{\cO_X}\cS_F$. Now, the statement for
\[
\cG\otimes_{\cO_X}
\cN_F\left[1/\prod_{i=1}^l f_i\right]
\] 
follows because it is a submodule of the corresponding relative regular
holonomic module.
\end{proof}

\subsection{Essential fibers of relative holonomic $\shD$-modules}

Let $\cM$ be a relative holonomic $\sA_X^R$-module such that
$j(\cM)=n+l$ for some $0\le l\le \dim R$. We denote by
\[
\supp(\cM)=\supp_X(\cM)=\supp_X^R(\cM)\subseteq \Spec R
\]
its support as an $R$-module. Recall that if $\cN$ is an $\sA_X^R$-module, then
\[
B_\cN\coloneqq \Ann_R(\cN)\subseteq R
\]
is called the \emph{Bernstein--Sato ideal} of $\cN$. By
\cite[Lem.3.4.1]{BVWZ}, we have
\begin{equation}\label{eq:suppbideal=}
    Z(B_\cM)=\supp(\cM)=p_2\big(\Ch_{\sA_X^R}(\cM)\big).
\end{equation}
By relative holonomicity and Lemma \ref{lm:ausdual}, we have
\[
\dim \supp(\cM)=\dim R-l.
\]

\begin{lemma}\label{lm:genericCMness}
With notation and assumptions above, there exists a closed subvariety
$V\subsetneq \supp(\cM)$
such that $\cM$ is $(n+l)$-Cohen--Macaulay away from $V$. In other words,
$\cM$ is generically Cohen--Macaulay. Moreover, if $\cM$ is $(n+l)$-pure, then
\[
\dim V\le \dim \supp(\cM)-2.
\]
\end{lemma}

\begin{proof}
Set $d=\dim R$. By Auslander regularity, we have
\[
j\big(\operatorname{Ext}^q_{\sA_X^R}(\cM,\sA_X^R)\big)\ge q
\]
for every $q$. By Lemma \ref{lm:ausdual}, it follows that
\[
\dim \Ch\big(\operatorname{Ext}^q_{\sA_X^R}(\cM,\sA_X^R)\big)
\le 2n+d-q.
\]
By \cite[Lem.3.4.1]{BVWZ}, we therefore have
\[
\dim \supp\big(\operatorname{Ext}^q_{\sA_X^R}(\cM,\sA_X^R)\big)
\le n+d-q.
\]
We then take
\[
V=
\bigcup_{q>n+l}
\supp\big(\operatorname{Ext}^q_{\sA_X^R}(\cM,\sA_X^R)\big)
\subseteq \supp(\cM).
\]
By the above dimension estimate, $V$ is a proper closed subset of
$\supp(\cM)$. Hence, if
\[
\Omega=\Spec R_\Omega\subseteq \Spec R\setminus V,
\]
then $\cM|_\Omega$ is Cohen--Macaulay; see \cite[Def.3.3.1]{BVWZ}.

If $\cM$ is $(n+l)$-pure, then by \cite[A.IV 2.6]{BJ} we have
\[
j\big(\operatorname{Ext}^q_{\sA_X^R}(\cM,\sA_X^R)\big)\ge q+1
\]
for $q>n+l$. Thus,
\[
\dim V\le \dim \supp(\cM)-2.
\]
\end{proof}

\begin{prop}\label{prop:essfib}
Let $p\in \supp(\cM)\setminus V$ be a closed point, and let
$W\subseteq \Spec R$
be an $l$-dimensional smooth closed subvariety containing $p$ and intersecting
$\supp(\cM)$ properly. Then the algebraic relative localization
\[
\left(\frac{\cM}{\cI_W\cdot\cM}\right)\bigg|_{\Omega}
=
\frac{\cM}{\cI_W\cdot\cM}\otimes_R R_\Omega
\]
is a holonomic $\shD_X$-module, where $\cI_W$ is the ideal of $W$ and
$\Omega=\Spec R_\Omega$ is a sufficiently small affine Zariski open
neighborhood of $p\in \Spec R$ such that
\[
W\cap \supp(\cM)\cap \Omega=\{p\}.
\]
We call
\[
\left(\frac{\cM}{\cI_W\cdot\cM}\right)\bigg|_{\Omega}
\]
the \emph{essential fiber} of $\cM$ at $p$ along $W$.
\end{prop}

\begin{proof}
Set $d=\dim R$. After shrinking $\Omega$ if necessary, choose a regular
sequence
\[
(x_1,\dots,x_{d-l})
\]
generating $\cI_W|_\Omega$. We then apply
\cite[Lem.3.4.2]{BVWZ}. Since $(x_1,\dots,x_{d-l})$ is a regular sequence,
the arguments in the proof of \cite[Prop.3.4.3]{BVWZ} show that
\[
\frac{\cM|_\Omega}{(x_1,\dots,x_i)\cdot \cM|_\Omega}
\]
is Cohen--Macaulay for every $i\le d-l$. In particular,
\[
\left(\frac{\cM}{\cI_W\cdot\cM}\right)\bigg|_{\Omega}
\simeq
\frac{\cM|_\Omega}{(x_1,\dots,x_{d-l})\cdot \cM|_\Omega}
\]
is $(n+d)$-Cohen--Macaulay and relative holonomic over
$\sA_X^{R_\Omega}$. Since its support over $\Spec R_\Omega$ is the closed
point $p$, forgetting the $R_\Omega$-module structure yields a holonomic
$\shD_X$-module.
\end{proof}

\begin{remark}
\begin{enumerate}[label=$(\arabic*)$]
    \item If $\sM$ is a relative holonomic $\wt{\sA_X^R}$-module, then locally
    on $X^\an\times \Specan R$, the support $\supp(\sM)$ is a locally closed
    analytic subvariety of $\Specan R$. Thus, after replacing $V$ by a locally
    closed analytic subvariety, Lemma \ref{lm:genericCMness} still holds for
    $\sM$. Moreover, after replacing algebraic relative localization by
    analytic relative localization, Proposition \ref{prop:essfib} also holds
    for $\sM$. Hence one can define the essential fibers of $\sM$ similarly,
    and these essential fibers are holonomic $\shD_{X^\an}$-modules.

    \item The essential fiber at $p$ may depend on the choice of $W$, but its
    characteristic cycle does not.
\end{enumerate}
\end{remark}

\section{$\Q$-specializable $\shD$-modules}\label{sect:Qspecial}

Let $f\in \cO_X$ be a regular function, and let
$\iota_f\colon X\rightarrow X\times \A_\C^1$
be the graph embedding. For a $\shD_X$-module $\cM$, we use
$V_\bullet\cM_f$ to denote the $\Z$-indexed increasing
Kashiwara--Malgrange filtration of
$\cM_f\coloneqq \iota_{f,+}\cM$
along $X\times\{0\}$; see, for instance, \cite[\S 4.1]{Wuch} for the
construction. We say that $\cM$ is \emph{specializable} along $f$ if
$V_\bullet\cM_f$ exists. We say that $\cM$ is \emph{$\Q$-specializable}
along $f$ if, moreover, there exists a $b$-function
$b_f(s)\in \C[s]$
whose roots all lie in $\Q$ and such that
$b_f(t\partial_t+k)$
annihilates
\[
\gr^V_k\cM_f=V_k\cM_f/V_{k-1}\cM_f
\]
for every $k\in \Z$. If $\cM$ is $\Q$-specializable along $f$, then for
$\alpha\in \Q$ we use
$\gr^V_\alpha \cM_f$
to denote the generalized $\alpha$-eigenspace, with respect to the
$t\partial_t$-operation, of
\[
V_k\cM_f/V_{k-1}\cM_f,
\qquad
k=\lfloor \alpha\rfloor+1.
\]

More generally, let
$L\in \Z_{\ge 0}^l$
be a slope with respect to the smooth complete intersection
\[
X\times\{0\}\subseteq Y=X\times \A_\C^l.
\]
For a $\shD_Y$-module $\cP$, we use ${}^L V_\bullet\cP$ to denote the
$\Z$-indexed increasing Kashiwara--Malgrange filtration along the slope $L$,
following Sabbah \cite{Sab}; see also \cite[Def.4.14]{Wuch}. We say that
$\cP$ is \emph{$\Q$-$L$-specializable} if ${}^L V_\bullet\cP$ exists and if
there exists a $b$-function
$b_L(s)\in \C[s]$
whose roots all lie in $\Q$ and such that
$b_L(v_L+k)$
annihilates
\[
\gr^{{}^LV}_k\cP={}^L V_k\cP/{}^L V_{k-1}\cP
\]
for every $k\in \Z$, where $v_L$ is the Euler vector field along $L$; cf.
\cite[Remark 4.13]{Wuch}.

Let
$G=(g_1,\dots,g_l)\colon X\rightarrow \A_\C^l$
be a regular map, and let
$\iota_G\colon X\to X\times \A_\C^l$
be the graph embedding. For a $\shD_X$-module $\cM$, we say that $\cM$ is
\emph{$\Q$-$L$-specializable} along $G$ if
$\cM_G\coloneqq \iota_{G,+}\cM$
is $\Q$-$L$-specializable. Finally, we say that $\cM$ is
\emph{$\Q$-specializable} if it is $\Q$-$L$-specializable along every $G$ and
for every slope $L$.

\begin{lemma}\label{lm:sqstQsp}
We have:
\begin{enumerate}[label=$(\arabic*)$]
    \item subquotients of $\Q$-specializable $\shD_X$-modules are
    $\Q$-specializable;
    \item extensions of $\Q$-specializable holonomic $\shD_X$-modules are
    $\Q$-specializable;
    \item $\shD_X$-modules underlying $\Q$-mixed Hodge modules are
    $\Q$-specializable.
\end{enumerate}
\end{lemma}

\begin{proof}
By uniqueness of $V$-filtrations along $G$ and for a slope $L$, if $\cN$ is a submodule, respectively a
quotient module, of a $\Q$-specializable $\shD_X$-module $\cM$, then the
${}^LV$-filtration of $\cM$ induces the ${}^LV$-filtration of $\cN$. Moreover,
$\gr^{{}^LV}_\bullet\cN_G$ is respectively a submodule or a quotient module of
$\gr^{{}^LV}_\bullet\cM_G$. Part $(1)$ follows.

To prove part $(2)$, consider a short exact sequence of holonomic
$\shD_X$-modules
\[
0\to \cM'\to \cM\to \cM''\to 0
\]
such that $\cM'$ and $\cM''$ are $\Q$-specializable. Since $\cM$ is
holonomic, the Kashiwara--Malgrange filtration of $\cM$ along any slope $L$
exists. By uniqueness, this filtration induces the Kashiwara--Malgrange
filtrations of $\cM'$ and $\cM''$ along $L$, and we obtain short exact sequences
\[
0\to
\frac{{}^L V_k\cM'_G}{{}^L V_{k-1}\cM'_G}
\to
\frac{{}^L V_k\cM_G}{{}^L V_{k-1}\cM_G}
\to
\frac{{}^L V_k\cM_G''}{{}^L V_{k-1}\cM''_G}
\to 0
\]
for all $k$. Taking a common multiple of the corresponding $b$-functions for
$\cM'$ and $\cM''$, we see that $\cM$ is also $\Q$-specializable.

Finally, graph embeddings and nearby cycles (along $L$) preserve $\Q$-mixed Hodge
modules. Hence part $(3)$ follows from the definition of $\Q$-mixed Hodge
modules; cf. \cite{MHM90} and \cite[Lem.4.16]{Wuch}.
\end{proof}

\begin{theorem}\label{thm:latticemtqsupp}
Let $\cM$ be a $($respectively
$\Q$-specializable$)$ regular holonomic $\shD_Y$-module, supported on the graph of
$G=(g_1,\dots,g_l)$
and satisfying $\cM=\cM(*D)$. Let $\cN\subseteq \cM$ be a
$\shD_{Y,D}$-lattice, and let $\cL=\cL_{\bm}$ be the maximal tame pure extension
of $\cN$ inside $\cN(*D_{\bm})$ for some nonzero $\bm\in \N^l$. Then:
\begin{enumerate}[label=$(\arabic*)$]
    \item $\cN$ and $\cL$ are relative regular holonomic $\sA_X^l$-modules.

    \item The top-dimensional part of
    \[
    \supp_X\big(\cN/\alpha(\bm)\cdot\cN\big)
    \subseteq
    \Spec R_l\simeq \C^l
    \]
    is a nonempty finite union of translated, respectively $\Q$-translated,
    linear subvarieties of codimension one, each of the form
    \[
    \left(\sum_{i=1}^l a_i s_i+b=0\right),
    \]
    for some nonzero primitive slope vector
    $(a_1,\dots,a_l)\in \Z_{\ge 0}^l$
    such that
    \[\bm\cdot (a_1,\dots,a_l)\neq 0.\]
    \item The support
    \[
    \supp_X\big(\cL/\alpha(\bm)\cdot\cL\big)
    \subseteq
    \Spec R_l
    \]
    is a nonempty finite union of translated, respectively $\Q$-translated,
    linear subvarieties of codimension one.
\end{enumerate}
\end{theorem}

\begin{proof}
Since
$\cN|_{Y\setminus D}=\cM|_{Y\setminus D}$
is a holonomic $\shD_{Y\setminus D}$-module, and since submodules of holonomic
modules are holonomic, $\cN|_{Y\setminus D}$ is $(n+l)$-pure over $\shD_{Y\setminus D}$.
Since taking pure filtrations commutes with localization, $T_j(\cN)\subseteq
\cN$ is supported on $D$ for $j>n+l$. Since $\cN$ is a lattice, however, we have
\[
T_j(\cN)=0
\qquad
\text{for } j>n+l.
\]
Thus $\cN$ is $(n+l)$-pure over $\shD_{Y,D}$, so its maximal tame pure extension
$\cL$ is well-defined.

By Kashiwara equivalence, we have
\[
\cM
\simeq
\cG[s_1,\dots,s_l]\prod_{i=1}^l g_i^{s_i},
\]
where
$\cG=\cG\left[1/\prod_{i=1}^l g_i\right]$
is a regular holonomic $\shD_X$-module. By Theorem \ref{thm:mainFFS},
$\cM$
is a relative regular holonomic $\sA_X^l$-module. Since $\cL$ and $\cN$ are
finitely generated $\sA_X^l$-submodules of this module, they are also relative
regular holonomic. This proves $(1)$.

Write
$t=\prod_{i=1}^l t_i.$
By \cite[Thm.3.1 and Cor.3.15]{WuRHA}, the codimension-one part of
$\supp_X(\cN/t\cN)$
is a finite union of translated linear subvarieties. Since $\cN/t_i\cN$ is a
quotient of $\cN/t\cN$, the codimension-one part of
$\supp_X(\cN/t_i\cN)$
is also a union of translated linear subvarieties. By Lemma
\ref{lm:lattictwist}, the quotient
$\cN/\alpha(\bm)\cdot\cN$
can be inductively decomposed into extensions of modules of the form
$\cN(-H)/t_i\cN(-H),$
for effective divisors $H$ supported on $D$. Hence the codimension-one part of
\[
\supp_X\big(\cN/\alpha(\bm)\cdot\cN\big)
\]
is also a union of translated linear subvarieties.

By Lemma \ref{lm:purfilcodim},
$\cL/\alpha(\bm)\cdot\cL$
is $(n+1)$-pure over $\sA_X^l$. Therefore, by Lemma \ref{lm:tamepurechv},
\[
\supp_X\big(\cL/\alpha(\bm)\cdot\cL\big)
\]
is a union of translated linear subvarieties of codimension one. In the
$\Q$-specializable case,
$\supp_X(\cL/t\cL)$
is a union of $\Q$-translated linear subvarieties; cf. \cite{Sab} and
\cite[\S 4.9]{Wuch}. Hence
\[
\supp_X\big(\cL/\alpha(\bm)\cdot\cL\big)
\]
is also a union of $\Q$-translated linear subvarieties. This proves $(3)$ and
the translated-linearity assertion in $(2)$.

It remains to prove the stated condition on the slope vector in $(2)$. By
Lemma \ref{lm:tamepurechv}, it is enough to prove this condition for the
corresponding codimension-one components of
\[
\supp_X\big(\cL/\alpha(\bm)\cdot\cL\big).
\]
Suppose, for contradiction, that
\[
\left(\sum_{i=1}^l a_i s_i+b=0\right)
\]
is such a component and that
$\bm\cdot(a_1,\dots,a_l)=0.$
Let
$I=\{i\mid \bm_i\neq 0\}\subseteq \{1,\dots,l\}$ and
let $\bar I$ denote the complement of $I$ in $\{1,\dots,l\}$. We write $D_I$
for the divisor defined by
$\prod_{i\in I}t_i=0,$
and similarly for $D_{\bar I}$. Since $a_i\ge 0$ for all $i$, the equality
$\bm\cdot(a_1,\dots,a_l)=0$
implies that
$a_i=0$ for every  $i\in I$.

Since characteristic varieties are compatible with localization, we have
\[
\Ch_{\sA_X^{R_l}}
\big(\cN/\alpha(\bm)\cdot\cN\big)
\big|_{X\setminus \iota_G^{-1}(D_{\bar I})}
=
\Ch_{\sA_{X\setminus \iota_G^{-1}(D_{\bar I})}^{R_l}}
\left(
\cN(*D_{\bar I})/
\alpha(\bm)\cdot\cN(*D_{\bar I})
\right).
\]
Moreover,
\[
\cN(*D_{\bar I})/
\alpha(\bm)\cdot\cN(*D_{\bar I})
\]
is a $\shD_{Y_I}$-module supported on the graph of
$G_{\bar I}=(g_i)_{i\notin I}.$
By Kashiwara equivalence, we obtain an isomorphism
\[
\cN(*D_{\bar I})/
\alpha(\bm)\cdot\cN(*D_{\bar I})
\simeq
\cT\otimes_\C \C[s_i]_{i\notin I}
\prod_{i\notin I}g_i^{s_i}
\]
for some
$\sA^{R_I}_{X\setminus \iota_G^{-1}(D_{\bar I})}
\text{-module } \cT.$
Consequently,
\[
\supp_{\C[s_1,\dots,s_l]}
\left(
\cN(*D_{\bar I})/
\alpha(\bm)\cdot\cN(*D_{\bar I})
\right)
=
\supp_{R_I}(\cT)\times \Spec \C[s_i]_{i\notin I}.
\]
Since $a_i=0$ for all $i\in I$, the hyperplane
\[
\left(\sum_{i=1}^l a_i s_i+b=0\right)
\]
imposes a nontrivial condition only on the variables $s_i$ with $i\notin I$.
Thus, by relative holonomicity and \eqref{eq:suppbideal=}, the components of
\[
\Ch_{\sA_X^{R_l}}\big(\cN/\alpha(\bm)\cdot\cN\big)
\]
lying over this hyperplane can only be supported over
$X\cap \iota_G^{-1}(D_{\bar I}\setminus D_I).$
But
$\cN/\alpha(\bm)\cdot\cN$
is supported on
$X\cap \iota_G^{-1}(D_I)$
by definition. Hence
\[
\Ch_{\sA_X^{R_l}}\big(\cN/\alpha(\bm)\cdot\cN\big)
\]
has no component over
$X\cap \iota_G^{-1}(D_{\bar I}\setminus D_I),$
which is a contradiction. Therefore
\[
\bm\cdot(a_1,\dots,a_l)\neq 0.
\]
The theorem follows.
\end{proof}


\section{Relative minimal extensions and generalized nearby cycles}
\label{sect:rmegnc}

In this section, we first construct relative minimal extensions of lattices of GE-type. We then use these minimal extensions to construct generalized nearby cycles.

\subsection{Geometry of relative supports}

The following is a relative generalization of Ginsburg's result \cite[Thm.A1.2]{Gil1}; see also \cite[Thm.1.6]{Wuch}. Its proof is given in Appendix II, using a relative characteristic variety formula of Sabbah from Appendix I.

\begin{theorem}\label{thm:relGins-Sab}
Let $\cN$ be an $\sA^R_{Y,D}$-lattice such that $\cN(*D)$ is a relative regular holonomic $\sA_Y^R$-module. Then
\[
\CC^{\dagger}(\cN)=\overline{\CC(\cN|_{U})}^{\log},
\]
where $U=Y\setminus D$, and $\overline{\CC(\cN|_{U})}^{\log}$ denotes the closure of
\[
{\CC}(\cN)|_{U}\subseteq T^*U\times \Spec R
\]
inside
$T^*(Y,D)\times \Spec R.$
\end{theorem}

\begin{coro}\label{cor:purefillattice}
If $\cN$ is an $\sA_{Y,D}^R$-lattice such that $\cN(*D)$ is a relative regular holonomic $\sA_Y^R$-module, then $T_i\cN$ and $T_i\cN/T_{i+1}\cN$ are $\sA^R_{Y,D}$-lattices for every $i$.
\end{coro}

\begin{proof}
Set
$T'_i\cN=\cN\cap T_i(\cN|_{Y\setminus D}).$
Then we have inclusions
\[
T'_i\cN\hookrightarrow T_i(\cN|_{Y\setminus D})
\]
and
\[
T'_i\cN/T'_{i+1}\cN
\hookrightarrow
T_i(\cN|_{Y\setminus D})/T_{i+1}(\cN|_{Y\setminus D}).
\]
Thus $T'_i\cN$ and $T'_i\cN/T'_{i+1}\cN$ are $\sA^R_{Y,D}$-lattices for every $i$. By Theorem \ref{thm:relGins-Sab}, the dimension of
\[
\Ch^\dagger(T'_i\cN/T'_{i+1}\cN)
\]
is the same as that of
\[
\Ch\big(T_i(\cN|_{Y\setminus D})/T_{i+1}(\cN|_{Y\setminus D})\big).
\]
Thus, if $T'_i\cN/T'_{i+1}\cN\neq 0$, then Lemma \ref{lm:purfilcodim} gives
$j(T'_i\cN/T'_{i+1}\cN)=i.$
Any nonzero submodule of $T'_i\cN/T'_{i+1}\cN$ of grade greater than $i$ would be supported on $D$, which is impossible because $T'_i\cN/T'_{i+1}\cN$ is a lattice. Therefore, if $T'_i\cN/T'_{i+1}\cN\neq 0$, then it is $i$-pure. By uniqueness of the pure filtration, we obtain
$T'_i\cN=T_i\cN.$
\end{proof}

\begin{coro}\label{cor:puregradequot}
If $\cN$ is an $\sA_{Y,D}^R$-lattice such that $\cN(*D)$ is a relative regular holonomic $\sA_Y^R$-module, then
\[
\Ch^\dagger(\cN/\alpha(\bm)\cdot\cN)
=
\Ch^\dagger(\cN)|_{(\alpha(\bm)=0)}.
\]
In particular, if the support of $\cN$, as an $\cO_Y$-module, intersects $(\alpha(\bm)=0)$ nontrivially for $\bm\in \N^l$, then
\[
j(\cN/\alpha(\bm)\cdot\cN)=j(\cN)+1.
\]
\end{coro}

\begin{proof}
For simplicity, take $\bm=e_i$. The general case is similar. By Corollary \ref{cor:purefillattice} and induction on the pure filtration, it is enough to assume that $\cN$ is pure.

Since $\cN$ is a lattice, we have a short exact sequence
\[
0\to \cN\xrightarrow{t_i}\cN\to \cN/t_i\cN\to0.
\]
By \cite[A.IV 4.10 and 4.11]{BJ}, or equivalently \cite[Prop.4.4.1]{BVWZ}, there exists a good filtration $F_\bullet^\dagger\cN$ over $F^\dagger_\bullet\sA_{Y,D}^R$ such that $\gr^{\dagger}_\bullet\cN$ is pure over $\gr^{\dagger}_\bullet\sA_{Y,D}^R$.

We then obtain a filtered morphism
\[
F^\dagger_\bullet\cN\xrightarrow{t_i}F^\dagger_\bullet\cN,
\]
and hence a convergent spectral sequence whose $E^1$-page is given by
\[
\varphi\colon
\gr^{\dagger}_\bullet\cN
\xrightarrow{t_i}
\gr^{\dagger}_\bullet\cN.
\]
By Theorem \ref{thm:relGins-Sab}, $\Ch^\dagger(\cN)$ and $(t_i=0)$ intersect properly. Thus $\varphi$ is injective by purity, which implies that the spectral sequence degenerates at the $E^1$-page. Therefore,
\[
\gr^\dagger_\bullet(\cN/t_i\cN)
\simeq
\frac{\gr^{\dagger}_\bullet\cN}
{t_i\gr^{\dagger}_\bullet\cN}.
\]
By \cite[Prop.1.5.3]{Gil}, we have
\[
\Ch^\dagger(\cN/t_i\cN)
=
\Ch^\dagger(\cN)|_{(t_i=0)}.
\]
If the support of $\cN$, as an $\cO_Y$-module, intersects $(t_i=0)$ nontrivially, then this characteristic variety has dimension
\[
2(n+l)+\dim R-j(\cN)-1,
\]
since the $\cO_Y$-support of $\cN$ is $\pi(\Ch^\dagger(\cN))$, where
$\pi\colon T^*(Y,D)\times \Spec R\to Y$
is the natural projection. Lemma \ref{lm:ausdual} then gives
\[
j(\cN/t_i\cN)=j(\cN)+1.
\]
\end{proof}

\begin{coro}\label{cor:ccresintprop}
Let $\cN$ be a $j$-pure $\sA_{Y,D}^R$-lattice such that $\cN(*D)$ is a relative regular holonomic $\sA^R_Y$-module, and let $b\in R$ be a noninvertible polynomial. If the hypersurface $(b=0)\subseteq \Spec R$ intersects
\[
\supp^R_{Y\setminus D}(\cN|_{Y\setminus D})\subseteq \Spec R
\]
properly, then both maps
\[
\cN\xrightarrow{b}\cN
\quad \text{and} \quad
\cN(*D)\xrightarrow{b}\cN(*D)
\]
are injective. Moreover,
\[
\CC^\dagger(\cN/b\cdot\cN)
=
\CC^\dagger(\cN)|_{b=0}.
\]
\end{coro}

\begin{proof}
By \cite[Lem. 3.4.2]{BVWZ}, the map
\[
\cN|_{Y\setminus D}\xrightarrow{b}\cN|_{Y\setminus D}
\]
is injective. Since neither $\cN$ nor $\cN(*D)$ has nonzero torsion submodules supported on $D$, the maps
\[
\cN\xrightarrow{b}\cN
\quad \text{and} \quad
\cN(*D)\xrightarrow{b}\cN(*D)
\]
are injective.

By relative holonomicity and purity, we can write
\[
\CC_{Y\setminus D}(\cN|_{Y\setminus D})
=
\sum_p \La_p\times Z_p,
\]
where the $\La_p$ are conic Lagrangian cycles in $T^*(Y\setminus D)$ and the $Z_p\subseteq \Spec R$ are pure-dimensional of dimension
$n+l+\dim R-j.$
By \eqref{eq:suppbideal=}, each $Z_p$ intersects $(b=0)$ properly. By Theorem \ref{thm:relGins-Sab}, we have
\[
\CC^\dagger(\cN)
=
\sum_p \overline{\La}^{\log}_p\times Z_p,
\]
and hence $\CC^\dagger(\cN)$ intersects $(b=0)$ properly.

By \cite[Prop.4.4.1]{BVWZ}, there exists a good filtration $F^\dagger_\bullet\cN$ such that $\gr^\dagger_\bullet\cN$ is $j$-pure and such that
\[
\gr^\dagger_\bullet\cN
\xrightarrow{b}
\gr^\dagger_\bullet\cN
\]
is injective. Thus
\[
\gr^\dagger_\bullet(\cN/b\cN)
\simeq
\frac{\gr^\dagger_\bullet\cN}
{b\cdot\gr^\dagger_\bullet\cN}.
\]
Since the $\gr^\dagger_\bullet\sA^R_{Y,D}$-support of $\gr^\dagger_\bullet\cN$ is $\CC^\dagger(\cN)$ by construction, \cite[Prop.1.5.3]{Gil} gives
\[
\CC^\dagger(\cN/b\cdot\cN)
=
\CC^\dagger(\cN)|_{b=0}
=
\sum_p \overline{\La}^{\log}_p\times (Z_p|_{b=0}).
\]
\end{proof}

\begin{lemma}\label{lm:injimplyproint}
Let $\cN$ be a relative holonomic $\sA_X^R$-module, and let $b\in R$ be an irreducible polynomial. If multiplication by $b$,
\[
\cN\xrightarrow{b}\cN,
\]
is injective, then both $\supp(\cN)$ and $\Ch(\cN)$ intersect the divisor $(b=0)$ properly.
\end{lemma}

\begin{proof}
By relative holonomicity and \eqref{eq:suppbideal=}, it is enough to prove that $\supp(\cN)$ intersects the divisor $(b=0)\subseteq \Spec R$ properly. Suppose that an irreducible component of $\supp(\cN)$ does not intersect $(b=0)$ properly. Then this irreducible component is contained in $(b=0)$. Localizing in an affine open neighborhood of a general point of this component, we may assume that $\supp(\cN)$ is irreducible and contained in $(b=0)$.

By \eqref{eq:suppbideal=} again, we have $b^q\in B_\cN$ for some $q\gg 0$. Hence there exists $e\in \cN$ such that
\[
b^m\cdot e=0
\quad \text{but} \quad
b^{m-1}\cdot e\neq 0
\]
for some $m\ge 1$, contradicting the injectivity of multiplication by $b$ on $\cN$.
\end{proof}

Suppose that $\cQ$ is a finitely generated $\sA^R_{Y,D}$-module which is also relative holonomic over $\sA_X^{R\otimes_\C R_l}$, via the inclusion
\[
\sA_X^{R\otimes_\C R_l}\subseteq \sA_{Y,D}^R.
\]
If
$H=\sum_{i=1}^l a_iD_i$
is a divisor supported on $D$, then $\cQ(H)$ is again an $\sA^R_{Y,D}$-module. Since $s_i=-\partial_{t_i}t_i$, we have
\[
s_i\cdot(q\otimes e)
=
(s_i\cdot q)\otimes e+q\otimes(s_i\cdot e)
=
((s_i+a_i-1)q)\otimes e
\]
for $q\in \cQ$, and
$P\cdot(q\otimes e)=(P\cdot q)\otimes e$
for $P\in \shD_X$, where
$e=\prod_{i=1}^l t_i^{-a_i}$
is the generator of $\cO_Y(H)$. Hence $\cQ(H)$ is also relative holonomic over $\sA_X^{R\otimes_\C R_l}$.

Moreover, since
$t_is_i=(s_i+1)t_i,$
we have an induced $t_i$-action on both $R_l$ and $R\otimes_\C R_l$ given by
\[
t_i\cdot s_j=s_j+\delta_{ij},
\qquad
t_i\cdot b=b
\]
for $i,j=1,\dots,l$ and $b\in R$, where $\delta_{ij}$ is the Kronecker delta. Geometrically, $t_i$ thus acts on
\[
\Spec(R\otimes_\C R_l).
\]
If we identify $t_i$ with a loop around the origin in the $i$-th component of $(\C^*)^l$, then this action agrees with the $\sG_l=\pi_1((\C^*)^l)$-action on the covering
\[
\pi\colon
\Specan R\times \C^l
\longrightarrow
\Specan R\times(\C^*)^l,
\quad
(b,a)\mapsto (b,\textup{Exp}(a)),
\]
by deck transformations. Consequently,
\begin{equation}\label{eq:actionBsupp}
t_i\cdot B_\cQ=B_{\cQ(-D_i)}
\quad \textup{and} \quad
t_i\cdot\supp(\cQ)=\supp\big(\cQ(-D_i)\big).
\end{equation}

For a nonempty subset $I\subseteq\{1,\dots,l\}$, we write
\[
t_I=\prod_{i\in I}t_i;
\]
for example,
$t_i=t_{\{i\}}$ and 
$t=t_{\{1,\dots,l\}}.$
We denote by $D_I$ the effective divisor defined by $(t_I=0)$, and by
\[
\sG_I\simeq \sG_{|I|}\simeq \Z^{|I|}
\]
the free abelian group generated by the $t_i$ for $i\in I$. More intrinsically, $\sG_k$ is naturally isomorphic to the fundamental group of the fibers of the Kato--Nakayama space $Y_{\log}$ over the log strata $E^k$; see \cite[V.1]{Ogusbook} for the construction of $Y_{\log}$.

\begin{theorem}\label{thm:mainsuppfiberation}
Let $\cN$ be a $j$-pure $\sA^R_{Y,D}$-lattice of GE-type with respect to
$G=(g_1,\dots,g_l)$
such that $\cN(*D)$ is a relative regular holonomic $\sA^R_Y$-module. Then $\cN$ is a relative holonomic $\sA_X^{R\otimes_\C R_l}$-module, and $\cN(*D)$ is a relative regular holonomic $\sA_X^{R\otimes_\C R_l}$-module.

Moreover, for nonzero $\bm\in\N^l$, the projection
\[
p_{\log}\colon \Spec(R\otimes_\C R_l)\to \Spec R
\]
induces a surjective morphism $\bar p_{\log}
=
p_{\log}|_{\supp^h_X(\cN/\alpha(\bm)\cdot\cN)}$,
\[
\bar p_{\log}
\colon
\supp^h_X(\cN/\alpha(\bm)\cdot\cN)
\longrightarrow
\supp_Y(\cN(*D))
=
\supp^R_{Y\setminus D}(\cN|_{Y\setminus D})
\]
such that, for every closed point
$a\in \supp_Y(\cN(*D)),$
the fiber
\[
\bar p_{\log}^{-1}(a)\subseteq \Spec R_l
\]
is a finite union of codimension-one translated linear subvarieties of the form
\[
\left(\sum_{i=1}^l a_i s_i+b=0\right),
\]
where $(a_1,\dots,a_l)\in \Z_{\ge 0}^l$ is a nonzero primitive vector satisfying
$\bm\cdot(a_1,\dots,a_l)\neq 0.$
Here
$\supp^h_X(\cN/\alpha(\bm)\cdot\cN)$
denotes the top-dimensional part of
\[
\supp_X^{R\otimes_\C R_l}(\cN/\alpha(\bm)\cdot\cN).
\]
\end{theorem}

\begin{proof}
Since $\cN(*D)$ need not be finitely generated over $\sA_Y^R$, 
a priori, we do not know if $\supp_Y(\cN(*D))$ is algebraic. However, as $R$-modules, $\cN|_{Y\setminus D}$ and $\cN(*D)$ have the same support. Hence
\begin{equation}\label{eq:extensionsupplongD}
\supp^R_Y(\cN(*D))
=
\supp^R_{Y\setminus D}(\cN|_{Y\setminus D}).
\end{equation}
In particular,
\[
\supp_Y(\cN(*D))\subseteq \Spec R
\]
is an algebraic subvariety by \eqref{eq:suppbideal=}. Notice also that $B_{\cN|_{Y\setminus D}}$ annihilates both $\cN$ and $\cN(*D)$, since $\cN$ is a lattice.

Since $\cN$ is of GE-type, so is $\cN(*D)$. By \eqref{eq:getypeD}, we have
\[
\cN(*D)
\simeq
\cS\left[
1/\prod_{i=1}^l g_i,
s_1,\dots,s_l
\right]\cdot
\prod_{i=1}^l g_i^{s_i}.
\]
Since $\cN(*D)$ is a relative regular holonomic $\sA_Y^R$-module, $\cS$ is a relative regular holonomic $\sA_X^R$-module by \cite[Thm.1]{FFS23}. Thus, by Theorem \ref{thm:mainFFS} and \cite[Prop.4.34]{FFS23}, $\cN(*D)$ is relative regular holonomic over $\sA_X^{R\otimes_\C R_l}$. Together with \eqref{eq:extensionsupplongD}, this gives
\begin{equation}\label{eq:pullbacksupp*D}
\supp_X(\cN(*D))
=
p_{\log}^{-1}\big(\supp_Y(\cN(*D))\big)
\subseteq
\Spec(R\otimes_\C R_l).
\end{equation}

Moreover, since $\cN$ is a finitely generated $\sA_X^{R\otimes_\C R_l}$-submodule of $\cN(*D)$ by Lemma \ref{lm:GE-type}, it is relative regular holonomic over $\sA_X^{R\otimes_\C R_l}$, because the category of relative regular holonomic modules is abelian. Similarly, the maximal tame pure extension $\cL$ of $\cN$ inside $\cN(*D)$ is also relative regular holonomic over $\sA_X^{R\otimes_\C R_l}$. Since $\cN$ and $\cL$ are finitely generated over $\sA_X^{R\otimes_\C R_l}$, they are algebraic relative holonomic over $\sA_X^{R\otimes_\C R_l}$. Thus $\supp_X(\cN)$ and $\supp_X(\cN/t_i\cN)$ are algebraic subvarieties of $\Spec(R\otimes_\C R_l)$, and the same holds for $\cL$.

By Corollary \ref{cor:puregradequot}, $\cL/t_i\cL$ is pure over $\sA_{Y,D}^R$. Since $\cL$ and $\cN$ are of GE-type, Lemma \ref{lm:GE-type} and Lemma \ref{lm:tamepurechv} imply that $\cN$, $\cL$, and $\cL/t_i\cL$ are pure over $\sA_X^{R\otimes_\C R_l}$. Moreover, $\Ch_X(\cL/t_i\cL)$ is the top-dimensional part of $\Ch_X(\cN/t_i\cN)$. Hence, by relative holonomicity,
\[
\supp_X(\cL/t_i\cL)
=
\supp^h_X(\cN/t_i\cN).
\]

Now fix an arbitrary point
$a\in \supp^R_Y(\cN(*D))$.
Choose a general smooth hypersurface passing through $a$ and intersecting $\supp_Y(\cN(*D))$ properly in a neighborhood of $a$. Suppose this hypersurface is defined by $b=0$ for some degree-one polynomial $b\in R$, which we also regard as an element of $R\otimes_\C R_l$. For $m\in \Z_{<0}$, consider the commutative diagram
\[
\begin{tikzcd}
0\arrow[r]&
\cL\arrow[r]\arrow[d,"b"]&
t^{m}\cL\arrow[r]\arrow[d,"b"]&
{t^{m}\cL}/{\cL}\arrow[r]\arrow[d,"b"]&
0\\
0\arrow[r]&
\cL\arrow[r]&
t^{m}\cL\arrow[r]&
{t^{m}\cL}/{\cL}\arrow[r]&
0.
\end{tikzcd}
\]
Let $\cK_m$ denote the kernel of the third vertical map. By Corollary \ref{cor:ccresintprop}, the first and second vertical maps are injective, and
\begin{equation}\label{eq:restrictb=0}
\CC^\dagger_{\sA_{Y,D}^R}(\cL/b\cL)
=
\CC^\dagger_{\sA_{Y,D}^R}(\cL)|_{b=0}
=
\overline{\CC(\cN|_{Y\setminus D})|_{b=0}}^{\log}.
\end{equation}
By the snake lemma,
$\cK_m\subseteq \cL/b\cL.$

Letting $m\to -\infty$, we obtain a direct system whose limit is
\[
\begin{tikzcd}
0\arrow[r]&
\cL\arrow[r]\arrow[d,"b"]&
\cN(*D)\arrow[r]\arrow[d,"b"]&
{\cN(*D)}/{\cL}\arrow[r]\arrow[d,"b"]&
0\\
0\arrow[r]&
\cL\arrow[r]&
\cN(*D)\arrow[r]&
{\cN(*D)}/{\cL}\arrow[r]&
0.
\end{tikzcd}
\]
Let $\cK$ denote the kernel of the third vertical map. We obtain an ascending chain inside $\cL/b\cL$:
\[
\cK_{-1}\subseteq \cK_{-2}\subseteq \cdots \subseteq \cK_m\subseteq \cdots.
\]
Since the direct limit functor is exact, the limit of this chain is $\cK$. Since $\cL/b\cL$ is finitely generated over the noetherian ring $\sA^R_{Y,D}$, the ascending chain stabilizes; hence
$\cK=\cK_m$
for $m\ll 0$. Thus we have an exact sequence
\[
0\to
\cK
\to
\cL/b\cL
\to
\cN(*D)/b\cN(*D).
\]
Let $\cN'_1$ denote the image of
\[
\cL/b\cL\to \cN(*D)/b\cN(*D).
\]
Since $\cN(*D)/\cL$ is supported on $D$, the module $\cN'_1$ is a lattice of $\cN(*D)/b\cN(*D)$, and
\[
(\cL/b\cL)|_{Y\setminus D}
=
\cN'_1|_{Y\setminus D}.
\]
By Theorem \ref{thm:relGins-Sab} and \eqref{eq:restrictb=0}, we have
\[
\CC^\dagger(\cL/b\cL)=\CC^\dagger(\cN'_1).
\]
Now consider the short exact sequence
\[
0\to
\cK
\to
\cL/b\cL
\to
\cN'_1
\to0.
\]
Counting multiplicities gives
\[
\dim\Ch^\dagger_{\sA_{Y,D}^R}(\cK)
<
\dim\Ch^\dagger_{\sA_{Y,D}^R}(\cL/b\cL).
\]
By Lemma \ref{lm:ausdual}, we obtain
$j(\cK)>j+1.$

On the other hand, applying Corollary \ref{cor:puregradequot}, Lemma \ref{lm:tamepurechv}, and Corollary \ref{cor:sesk-pure} repeatedly, we see that $t^m\cL/\cL$ is $(j+1)$-pure. Since $\cK=\cK_m$ is a submodule of a $(j+1)$-pure module and satisfies $j(\cK)>j+1$, we conclude that
\[
\cK=0
\quad \text{and} \quad
\cL/b\cL=\cN'_1.
\]
Similarly,
\[
t_i\cL/b(t_i\cL)=t_i\cN'_1.
\]
Consider the diagram
\[
\begin{tikzcd}
0\arrow[r]&
t_i\cL\arrow[r]\arrow[d,"b"]&
\cL\arrow[r]\arrow[d,"b"]&
{\cL}/{t_i\cL}\arrow[r]\arrow[d,"b"]&
0\\
0\arrow[r]&
t_i\cL\arrow[r]&
\cL\arrow[r]&
{\cL}/{t_i\cL}\arrow[r]&
0.
\end{tikzcd}
\]
It yields a short exact sequence of $\sA_X^{R\otimes_\C R_l}$-modules
\[
0\to
{\cL}/{t_i\cL}
\xrightarrow{b}
{\cL}/{t_i\cL}
\to
\cN'_1/t_i\cN'_1
\to0.
\]
By Lemma \ref{lm:injimplyproint}, the hypersurface
\[
(b=0)\subseteq \Spec(R\otimes_\C R_l)
\]
intersects $\supp_X(\cL/t_i\cL)$ properly. Applying Corollary \ref{cor:ccresintprop} with $Y=X$ and $D=\emptyset$, we obtain
\[
\Ch_{\sA_X^{R\otimes_\C R_l}}(\cN'_1/t_i\cN'_1)
=
\Ch_{\sA_X^{R\otimes_\C R_l}}(\cL/t_i\cL)|_{b=0},
\]
which is equidimensional.

Since $\cN(*D)$ is relative regular holonomic over $\sA_Y^R$, so is
\[
\cN'_1(*D)=\cN(*D)/b\cN(*D).
\]
By Corollary \ref{cor:purefillattice}, the modules $T_m(\cN'_1)$ and
$T_m(\cN'_1)/T_{m+1}(\cN'_1)$
are lattices for all $m$. Write
$\cN_1=T_{j+1}(\cN'_1)/T_{j+2}(\cN'_1).$
Since $j(\cN'_1)=j+1$, we obtain a short exact sequence
\[
0\to
\frac{T_{j+2}(\cN'_1)}{t_iT_{j+2}(\cN'_1)}
\longrightarrow
\frac{\cN'_1}{t_i\cN'_1}
\longrightarrow
\frac{\cN_1}{t_i\cN_1}
\to0.
\]
Since $T_{j+2}(\cN'_1)$ and $\cN'_1$ are lattices with
\[
j(T_{j+2}(\cN'_1))\ge j+2=j(\cN'_1)+1,
\]
Corollary \ref{cor:puregradequot} gives
\[
j\left(
\frac{T_{j+2}(\cN'_1)}{t_iT_{j+2}(\cN'_1)}
\right)
\ge
j\left(
\frac{\cN'_1}{t_i\cN'_1}
\right)+1.
\]
By Lemma \ref{lm:ausdual}, applied with $Y=X$ and $D=\emptyset$, we get
\[
\dim\Ch_{\sA_X^{R\otimes_\C R_l}}
\left(
\frac{T_{j+2}(\cN'_1)}{t_iT_{j+2}(\cN'_1)}
\right)
<
\dim\Ch_{\sA_X^{R\otimes_\C R_l}}
\left(
\frac{\cN'_1}{t_i\cN'_1}
\right).
\]
Therefore,
\[
\Ch_{\sA_X^{R\otimes_\C R_l}}(\cN_1/t_i\cN_1)
=
\Ch_{\sA_X^{R\otimes_\C R_l}}(\cN'_1/t_i\cN'_1)
=
\Ch_{\sA_X^{R\otimes_\C R_l}}(\cL/t_i\cL)|_{b=0}.
\]

In summary, $\cN_1$ is a $(j+1)$-pure $\sA^R_{Y,D}$-lattice of GE-type such that:
\begin{enumerate}[label=$(\arabic*)$]
    \item $\cN_1(*D)$ is relative regular holonomic over $\sA^R_Y$;
    \item
    $\supp_{Y\setminus D}(\cN_1|_{Y\setminus D})
    =
    \supp_{Y\setminus D}(\cN|_{Y\setminus D})|_{b=0};$
    \item $\cN_1$ is relative holonomic over $\sA_X^{R\otimes_\C R_l}$;
    \item
    $a\in \supp_{Y\setminus D}(\cN_1|_{Y\setminus D})$;

    \item locally over an affine open neighborhood of $a$,
    \[
    \supp_X^{R\otimes_\C R_l}(\cN_1/t_i\cN_1)
    =
    \supp^h_X(\cN/t_i\cN)|_{b=0}.
    \]
\end{enumerate}

Now replace $\cN$ by $\cN_1$, and choose another general hyperplane $(b_1=0)$ passing through $a$ and intersecting
\[
\supp_Y(\cN_1(*D))
=
\supp_{Y\setminus D}(\cN_1|_{Y\setminus D})
\]
properly. Repeating the above argument inductively, we obtain modules $\cN_p$ for each $p\ge 1$. We stop the induction when
\[
\supp_Y(\cN_m(*D))=\{a\}.
\]
During the induction, we may replace $\Spec R$ by a sufficiently small Zariski open neighborhood of $a$ if necessary. At the end, we obtain an $\sA^R_{Y,D}$-lattice $\cN_m$ of GE-type such that:
\begin{enumerate}[label=$(\arabic*)$]
    \item $\cN_m(*D)$ is relative regular holonomic over $\sA^R_Y$;
    \item
    $
    \supp_{Y\setminus D}(\cN_m|_{Y\setminus D})
    =
    \supp_{Y\setminus D}(\cN|_{Y\setminus D})|_{b=0,b_1=0,\dots,b_{m-1}=0}
    =
    \{a\}$;
    \item $\cN_m$ is relative holonomic over $\sA_X^{R\otimes_\C R_l}$;
    \item
    $
    \supp_X^{R\otimes_\C R_l}(\cN_m/t_i\cN_m)
    =
    \supp^h_X(\cN/t_i\cN)|_{b=0,b_1=0,\dots,b_{m-1}=0}.
    $
\end{enumerate}

Properties $(1)$ and $(2)$ imply that $\cN_m(*D)$ is holonomic over $\shD_Y$ after forgetting the $R$-module structure. Thus $\cN_m$ is a $\shD_{Y,D}$-lattice of GE-type, and hence also a relative holonomic $\sA_X^l$-module. Applying Theorem \ref{thm:latticemtqsupp}, we conclude that
\[
\supp_X(\cN_m/t_i\cN_m)
\subseteq
\Spec R_l\simeq p_{\log}^{-1}(a)
\]
is a finite union of codimension-one translated linear subvarieties, since
\[
\Ch_{\sA_X^{R\otimes_\C R_l}}(\cN_m/t_i\cN_m)
\simeq
\Ch_{\sA_X^l}(\cN_m/t_i\cN_m),
\]
where the isomorphism is induced by
\[
\Spec R_l\simeq p_{\log}^{-1}(a)\hookrightarrow \Spec(R\otimes_\C R_l).
\]
Meanwhile, by \eqref{eq:extensionsupplongD}, we have
\[
p_{\log}\big(\supp^h_X(\cN/t_i\cN)\big)
\subseteq
\supp_{Y\setminus D}(\cN|_{Y\setminus D}).
\]
Therefore, the restriction $\bar p_{\log}$ is surjective, and its fibers are finite unions of codimension-one translated linear subvarieties. This proves the case $\bm=e_i$. Replacing $t_i$ by $\alpha(\bm)$ gives the general case.
\end{proof}

\begin{coro}\label{cor:anlocmaxminex}
Let $\cN$ be an $\sA^R_{Y,D}$-lattice of GE-type with respect to
$G=(g_1,\dots,g_l)$
such that $\cN(*D)$ is a relative regular holonomic $\sA^R_Y$-module, and let $\bm\in \N^l$ be nonzero. Then, for every
\[
\beta\in \Specan(R\otimes_\C R_l),
\]
there exist a small analytic open neighborhood $U_\beta$ of $\beta$ and an integer $m_0\in \Z_{>0}$ such that
\[
\widetilde{\alpha(\bm)^{-m}\cN}|_{X\times U_\beta}
=
\widetilde{\cN(*D_\bm)}|_{X\times U_\beta}
\]
and
\[
\widetilde{\alpha(\bm)^m\cN}|_{X\times U_\beta}
=
\widetilde{\alpha(\bm)^{m+1}\cN}|_{X\times U_\beta}
\]
as
$
\widetilde{\sA_X^{R\otimes_\C R_l}}|_{X\times U_\beta}
=
\shD_{X\times U_\beta/U_\beta}
$
modules for all $m\ge m_0$.
\end{coro}

\begin{proof}
Since the analytification functor is flat, by Corollary \ref{cor:purefillattice} and induction on the pure filtration, it is enough to assume that $\cN$ is $j$-pure.

Choose $I\subseteq\{1,\dots,l\}$ such that $D_I$ and $D_\bm$ have the same support. It is enough to replace $\alpha(\bm)$ by $t_I$. Since $\cN$ and its maximal tame pure extension $\cL_I$ are both lattices and
$\cN\subseteq\cL_I\subseteq\cN(*D_I)$,
we have
\[
\cN\subseteq\cL_I\subseteq t_I^{-m}\cN
\]
for all $m\gg 0$. Since analytification is exact, it is enough to replace $\cN$ by $\cL_I$.

By \eqref{eq:actionBsupp}, we have
\[
t_I^m\cdot
\supp_X(\cL_I/t_I\cL_I)
=
\supp_X(t_I^m\cL_I/t_I^{m+1}\cL_I).
\]
By Theorem \ref{thm:mainsuppfiberation},
\[
\supp^h_X(\cN/t_I\cN)
=
\supp_X(\cL_I/t_I\cL_I),
\]
and $\supp_X(\cL_I/t_I\cL_I)$ is a fibration over
\[
\supp_Y(\cN(*D))\subseteq \Spec R
\]
whose fibers are unions of codimension-one translated linear subvarieties in $\Spec R_l$. By construction, the $t_I$-action on $\supp_X(\cL_I/t_I\cL_I)$ fixes the base and translates each fiber integrally inside
$\Spec R_l\simeq \C^l$.
Thus the $t_I$-action preserves the fibration structure.

For
$\beta\in \Specan(R\otimes_\C R_l)$,
write
$a=p_{\log}(\beta)$.
If
$a\in \supp_Y(\cN(*D))$,
then there exists $m_0\gg0$ such that
$\beta\notin t_I^{\pm m}\cdot\bar p_{\log}^{-1}(a)$
for all $m\ge m_0$. Thus
\[
\beta\notin
t_I^{\pm m}\cdot
\supp_X(\cL_I/t_I\cL_I)
=
\supp_X(t_I^{\pm m}\cL_I/t_I^{\pm m+1}\cL_I).
\]
Since $\supp_X(\cL_I/t_I\cL_I)$ is closed, there exists a small analytic open neighborhood $U^m_\beta$ of $\beta$ such that
\[
U^m_\beta
\cap
\supp_X(t_I^{\pm m}\cL_I/t_I^{\pm m+1}\cL_I)
=
\emptyset.
\]
Equivalently,
\[
V^m_\beta
\cap
\supp_X(\cL_I/t_I\cL_I)
=
\emptyset,
\]
where
\[
V^m_\beta=t_I^{\pm m}\cdot U^m_\beta
\]
is a small neighborhood of $t_I^{\pm m}\cdot\beta$.

Compactify $\Spec R_l\simeq \A_\C^l$ to $\P^l$, and consider the open embedding
\[
\Specan(R\otimes_\C R_l)
\hookrightarrow
\Specan R\times \P^l.
\]
The Zariski closure of
$\{t_I^{\pm m}\cdot\beta\}_{m\ge m_0}$
inside
\[
\supp_Y(\cN(*D))\times \P^l
\subseteq
\Specan R\times \P^l
\]
is a complete complex line $l_I$ contained in $\{a\}\times \P^l$. The point at infinity of this line is the limit of
$\{t_I^{\pm m}\cdot\beta\}_{m\ge m_0}$.

Let $F$ denote the fiber over $a$ of the closure of
$\supp_X(\cL_I/t_I\cL_I)$
inside
$$\supp_Y(\cN(*D))\times \P^l.$$
Choose a good filtration
$F_\bullet(\cL_I/t_I\cL_I)$
over $F_\bullet\sA_X^{R\otimes_\C R_l}$ such that
$
\cG
\coloneqq
\gr^F_\bullet(\cL_I/t_I\cL_I)
$
is pure; see \cite[A.IV 4.11]{BJ}. Since $\cG$ is algebraic, it extends to a pure coherent module $\bar\cG$ on
$T^*X\times \Spec R\times \P^l$
whose support is the closure of
\[
\Ch_{\sA_X^{R\otimes_\C R_l}}(\cL_I/t_I\cL_I).
\]
Using an inductive process similar to the one in the proof of Theorem \ref{thm:mainsuppfiberation}, we obtain
\[
F=\overline{\bar p_{\log}^{-1}(a)},
\]
where the latter denotes the closure of $\bar p_{\log}^{-1}(a)$ in $\{a\}\times \P^l$.

The components of $\bar p_{\log}^{-1}(a)$ are of the form
\[
\left(\sum_{i=1}^l a_is_i+b=0\right),
\]
with
$(a_1,\dots,a_l)\in \Z_{\ge 0}^l$
satisfying
$\bm\cdot(a_1,\dots,a_l)\neq 0.$
Hence $l_I$ and $F$ do not meet at infinity. Since the point at infinity of $l_I$ is the limit of
\[
\{t_I^{\pm m}\cdot\beta\}_{m\ge m_0}
\]
and is not contained in $F$, the neighborhoods $V^{\pm m}_\beta$ do not shrink to $t_I^{\pm m}\cdot\beta$ as they approach infinity. Therefore, we can choose an analytic open neighborhood $U_\beta$ independent of $m$ such that, for all $m\ge m_0$,
\[
U_\beta
\cap
\supp_X(t_I^{\pm m}\cL_I/t_I^{\pm m+1}\cL_I)
=
\emptyset.
\]

If
$a\notin \supp_Y(\cN(*D)),$
then we choose $U_\beta$ to be a small neighborhood disjoint from
\[
p_{\log}^{-1}\big(\supp^R_{Y\setminus D}(\cN|_{Y\setminus D})\big).
\]
In either case,
\[
U_\beta
\cap
\supp_X(t_I^{\pm m}\cL_I/t_I^{\pm m+1}\cL_I)
=
\emptyset
\]
for all $m\ge m_0$. Therefore, by \cite[Lem.5.6]{WuRHA},
\[
\left.
\frac{\widetilde{t_I^{\pm m}\cL_I}}
{\widetilde{t_I^{\pm m+1}\cL_I}}
\right|_{X\times U_\beta}
=
0
\]
for all $m\ge m_0$. Consequently,
\[
\widetilde{t_I^{-m}\cL_I}|_{X\times U_\beta}
=
\widetilde{\cN(*D_I)}|_{X\times U_\beta}
\]
and
\[
\widetilde{t_I^m\cL_I}|_{X\times U_\beta}
=
\widetilde{t_I^{m+1}\cL_I}|_{X\times U_\beta}
\]
for all $m\ge m_0$. Replacing $t_I$ by $\alpha(\bm)$ gives the desired statement.
\end{proof}

\subsection{Relative minimal extensions of lattices}\label{subsect:rmel}

We continue to assume that $\cN$ is an $\sA^R_{Y,D}$-lattice of GE-type with
respect to
$G=(g_1,\dots,g_l)$
such that $\cN(*D)$ is a relative regular holonomic $\sA^R_Y$-module. For
$\bm\in\N^l$, let $I\subseteq \{1,\dots,l\}$ be the subset such that the support
of $D_\bm$ is $D_I$. We denote by $Y_I$ the subvariety of $Y$ defined by
$(t_i=0)_{i\notin I}$,
and by $\bar D_I\subseteq Y_I$ the divisor
$\bar D_I=Y_I\cap D_I$.
Then we have the closed embedding
\[
i_{Y_I}\colon Y_I\simeq X\times \A_\C^{|I|}\hookrightarrow Y
\]
and the graph embedding
\[
\iota_{G_I}\colon X\hookrightarrow Y_I,
\]
where
$G_I=(g_i)_{i\in I}.$

By Theorem \ref{thm:mainsuppfiberation}, both $\cN$ and $\cN(*D)$ are relative
regular holonomic $\sA_X^{R\otimes_\C R_l}$-modules; hence so is
$\cN(*D_I)$. Using the factorization through the partial graph
$\iota_{G_I}$, Corollary \ref{cor:factorthroughgraph} implies that
$\cN(*D)$ is relative regular holonomic over
$\sA_{Y_I}^{R\otimes_\C R_{l-|I|}}$,
where we write
\[
R_{l-|I|}=\C[s_i]_{i\notin I}.
\]
As a submodule, $\cN(*D_I)$ is relative regular holonomic over
$\sA_{Y_I}^{R\otimes_\C R_{l-|I|}}$ and is supported on the graph of $G_I$.

By \eqref{eq:getypeD}, we have
\begin{equation}\label{eq:partiallocKE}
\cN(*D_I)
\simeq
\cS_I[s_i]_{i\in I}\cdot \prod_{i\in I}g_i^{s_i},
\end{equation}
where
\[
\cS_I=\cS_I\left[1/\prod_{i\in I}g_i\right]
\]
is a relative regular holonomic
$\sA_X^{R\otimes_\C R_{l-|I|}}$-module. For simplicity, set
\[
\cS_{G_I}^*
\coloneqq
\cO_X\left[1/\prod_{i\in I}g_i\right][s_i]_{i\in I}
\cdot
\prod_{i\in I}g_i^{-s_i}.
\]
We define
\[
\wt{\cN(!D_\bm)}
=
\wt{\cN(!D_I)}
=
\wt{\cN(!D_I)}^X
\]
to be
\[
\cRhom_{\widetilde{\sA_X^{R\otimes_\C R_l}}}
\left(
\widetilde{\Rhom}_{\sA_X^{R\otimes_\C R_{l-|I|}}
[1/\prod_{i\in I}g_i]}
\left(
\cS_I,
\sA_X^{R\otimes_\C R_{l-|I|}}\otimes_{\cO_X}\cS_{G_I}^*
\right),
\widetilde{\sA_X^{R\otimes_\C R_l}}
\right).
\]
We call this the \emph{relative minimal extension} of $\cN$ along $D_\bm$.
It depends only on $D_I$, the support of $D_\bm$. Here the
$\widetilde{\phantom{M}}$-functor over the inner $\Rhom$ converts the complex
of $\sA_X^{R\otimes_\C R_l}$-modules into a complex of
$\widetilde{\sA_X^{R\otimes_\C R_l}}$-modules. The superscript $X$ in
$\wt{\cN(!D_I)}^X$ is used, when necessary, to indicate that the construction
is carried out relatively over $X$.

Let us prove that $\wt{\cN(!D_I)}$ is well-defined. First, observe that
\[
\operatorname{Ext}^i_{\sA_X^{R\otimes_\C R_{l-|I|}}
[1/\prod_{i\in I}g_i]}
\left(
\cS_I,
\sA_X^{R\otimes_\C R_{l-|I|}}\otimes_{\cO_X}\cS_{G_I}^*
\right)
\]
is isomorphic to
\[
\operatorname{Ext}^i_{\sA_X^{R\otimes_\C R_{l-|I|}}
[1/\prod_{i\in I}g_i]}
\left(
\cS_I,
\sA_X^{R\otimes_\C R_{l-|I|}}
\left[1/\prod_{i\in I}g_i\right]
\right)
\otimes_{\cO_X}\cS_{G_I}^*.
\]
This module is relative regular holonomic over
$\sA_X^{R\otimes_\C R_l}$ by \cite[Cor. 3.9 and Prop. 4.34]{FFS23} and
Theorem \ref{thm:mainFFS}. Thus
\[
\widetilde{\Rhom}_{\sA_X^{R\otimes_\C R_{l-|I|}}
[1/\prod_{i\in I}g_i]}
\left(
\cS_I,
\sA_X^{R\otimes_\C R_{l-|I|}}\otimes_{\cO_X}\cS_{G_I}^*
\right)
\]
is a complex of $\widetilde{\sA_X^{R\otimes_\C R_l}}$-modules whose cohomology
sheaves are regular holonomic over
$\widetilde{\sA_X^{R\otimes_\C R_l}}$. Therefore
$\wt{\cN(!D_I)}$ is a complex of relative regular holonomic
$\widetilde{\sA_X^{R\otimes_\C R_l}}$-modules, and
$\wt{\cN(!D_I)}$ depends only on $\cN|_{Y\setminus D_I}$. In particular,
$\wt{\cN(!D_I)}$ is well-defined.

By definition, when $R=\C$ and $I=\{1,2,\dots,l\}$, the object
$\wt{\cN(!D_I)}$ is the minimal extension constructed in
\cite[\S 3.3]{WuRHA}; see also \cite[\S 5.1]{BVWZ2} for the algebraic local
version.

\begin{remark}\label{rmk:gnr!ex}
It is straightforward to see that the complex
\[
\widetilde{\Rhom}_{\sA_X^{R\otimes_\C R_{l-|I|}}
[1/\prod_{i\in I}g_i]}
\left(
\cS_I,
\sA_X^{R\otimes_\C R_{l-|I|}}\otimes_{\cO_X}\cS_{G_I}^*
\right)
\]
is quasi-isomorphic to
\[
\widetilde{\cRhom}_{\widetilde{\sA_X^{R\otimes_\C R_{l-|I|}}}}
\left(
\wt{\cS_I},
\widetilde{\sA_X^{R\otimes_\C R_{l-|I|}}}
\otimes_{\cO_X}\cS_{G_I}^*
\right),
\]
where the $\widetilde{\phantom{M}}$-functor over $\cRhom$ converts the complex
of right
\[
\widetilde{\sA_X^{R\otimes_\C R_{l-|I|}}}[s_i]_{i\in I}
\]
modules into a complex of right
$\widetilde{\sA_X^{R\otimes_\C R_l}}$
modules. Therefore, to define $\wt{\cN(!D_I)}$, it suffices to assume that
$\cN$ is an $\sA^R_{Y,D}$-module of GE-type such that $\cS_I$ is relative
regular holonomic over $\sA_X^{R\otimes_\C R_{l-|I|}}$.
\end{remark}

\begin{lemma}
There is a natural morphism
\[
\wt{\cN(!D_\bm)}\longrightarrow \wt{\cN(*D_\bm)}.
\]
\end{lemma}

\begin{proof}
It is enough to prove that
\[
\Rhom_{\sA_X^{R\otimes_\C R_l}(*\bar D_I)}
\left(
\Rhom_{\sA_X^{R\otimes_\C R_{l-|I|}}
[1/\prod_{i\in I}g_i]}
\left(
\cS_I,
\sA_X^{R\otimes_\C R_{l-|I|}}\otimes_{\cO_X}\cS_{G_I}^*
\right),
\sA_X^{R\otimes_\C R_l}(*\bar D_I)
\right)
\]
is isomorphic to $\cN(*D_I)$, where
\[
\sA_X^{R\otimes_\C R_l}(*\bar D_I)
=
\sA_X^{R\otimes_\C R_l}
\left[1/\prod_{i\in I}g_i\right].
\]
Indeed, replacing $\shD_U$ by
\[
\sA_X^{R\otimes_\C R_{l-|I|}}
\left[1/\prod_{i\in I}g_i\right]
\]
in the proof of \cite[Lem.5.3.1]{BVWZ2}, the same argument gives the desired
claim. Moreover, by \eqref{eq:partiallocKE}, the same argument also gives
\begin{equation}\label{eq:isolocdualKE}
\begin{split}
&
\Rhom_{\sA_X^{R\otimes_\C R_{l-|I|}}
[1/\prod_{i\in I}g_i]}
\left(
\cS_I,
\sA_X^{R\otimes_\C R_{l-|I|}}\otimes_{\cO_X}\cS_{G_I}^*
\right)
\\
&\simeq
\Rhom_{\sA_X^{R\otimes_\C R_l}
[1/\prod_{i\in I}g_i]}
\left(
\cN(*D_I),
\sA_X^{R\otimes_\C R_l}
\left[1/\prod_{i\in I}g_i\right]
\right).
\end{split}
\end{equation}
\end{proof}

\begin{theorem}\label{thm:anarelbeilinson}
With notation and assumptions as above, we have:
\begin{enumerate}[label=$(\arabic*)$]
    \item
    $\wt{\cN(!D_\bm)}
    \stackrel{\mathrm{q.i.}}{\simeq}
    \cH^0\big(\wt{\cN(!D_\bm)}\big);$
    \item
    $\wt{\cN(!D_\bm)}
    \hookrightarrow
    \wt{\cN(*D_\bm)}.$
\end{enumerate}
\end{theorem}

The preceding theorem is a natural generalization of \cite[Lem.2.7]{WuRHA},
\cite[Thm.5.3]{WZ}, and the $j_!$-extension of Beilinson--Bernstein
\cite{BBJats}; see also \cite[Prop.3.8.3]{Gil}.

\begin{theorem}\label{thm:relminext}
The module $\wt{\cN(!D_\bm)}$ has no nonzero
$\wt{\sA_X^{R\otimes_\C R_l}}$-sub or quotient modules supported on the
divisor
\[
\left(\prod_{i\in I}g_i=0\right)
\subseteq
X\times \Specan(R\otimes_\C R_l).
\]
\end{theorem}

By the two preceding theorems, $\wt{\cN(!D_\bm)}$ gives the analytic relative
$\shD$-module analogue of the Deligne--Goresky--MacPherson minimal extension
for perverse sheaves; cf. \cite[\S 8.2.1]{HTT}. Moreover, by
Corollary \ref{cor:anlocmaxminex}, for every
$\beta\in \Specan(R\otimes_\C R_l)$
there exists a sufficiently small analytic open neighborhood $U_\beta$ of
$\beta$ such that
\begin{equation}\label{eq:relminexge}
\wt{\cN(!D_\bm)}|_{X\times U_\beta}
=
\widetilde{\alpha(\bm)^m\cN}|_{X\times U_\beta}
\end{equation}
for all $m\gg 0$.

Let
\[
R_I=\C[s_i]_{i\in I}
\quad \text{and} \quad
T_I=\C[\chi_i^{\pm1}]_{i\in I}.
\]
The covering
\[
\pi\colon \Specan R_I\to \Specan T_I
\]
induces a covering map
\[
\pi_I\colon
X\times\Specan(R\otimes_\C R_l)
\longrightarrow
X\times\Specan(R\otimes_\C R_{l-|I|}\otimes_\C T_I),
\]
whose fibers are $\sG_I$-torsors. Similarly to \cite[Thm.2.8]{WuRHA}, we
have the following consequence.

\begin{coro}
Both $\wt{\cN(!D_\bm)}$ and $\wt{\cN(*D_\bm)}$ are $\sG_I$-equivariant. Moreover,
the morphism
\[
\wt{\cN(!D_\bm)}
\hookrightarrow
\wt{\cN(*D_\bm)}
\]
is $\sG_I$-equivariant.
\end{coro}

\begin{proof}
By Corollary \ref{cor:anlocmaxminex} and \eqref{eq:relminexge}, the local data
can be glued into equivariant objects, as in \cite[\S 2.4]{WuRHA}.
\end{proof}

\begin{proof}[Proof of Theorems \ref{thm:anarelbeilinson} and \ref{thm:relminext}]
We argue by induction on the dimension of
\[
\supp_{Y_I\setminus \bar D_I}^{R\otimes_\C R_{l-|I|}}
(\cN|_{Y\setminus D}).
\]
Take the pure filtration of $\cN|_{Y\setminus D_I}$. By the inductive
hypothesis, it is enough to assume that $\cN|_{Y\setminus D_I}$ is $j$-pure.

Choose
$b\in R\otimes_\C R_{l-|I|}$
such that:
\begin{enumerate}[label=$(\arabic*)$]
    \item the hypersurface
    $(b=0)\subseteq \Spec(R\otimes_\C R_{l-|I|})$
    is smooth;
    \item $(b=0)$ intersects
    \[
    \supp_{Y_I\setminus \bar D_I}^{R\otimes_\C R_{l-|I|}}
    (\cN|_{Y\setminus D})
    \]
    properly.
\end{enumerate}
Set
$\cN_1\coloneqq \cN/b\cN.$
By Corollary \ref{cor:ccresintprop} and relative holonomicity, we have
\[
\supp_{Y_I\setminus \bar D_I}^{R\otimes_\C R_{l-|I|}}
(\cN_1|_{Y\setminus D})
=
\supp_{Y_I\setminus \bar D_I}^{R\otimes_\C R_{l-|I|}}
(\cN|_{Y\setminus D})|_{b=0}.
\]
By the inductive hypothesis, we may assume that
\[
\wt{\cN_1(!D_I)}
\stackrel{\mathrm{q.i.}}{\simeq}
\cH^0\big(\wt{\cN_1(!D_I)}\big),
\]
and that $\wt{\cN_1(!D_I)}$ has no nonzero quotient module supported on the
boundary divisor
\[
\left(\prod_{i\in I}g_i=0\right)
\subseteq
X\times\Specan(R\otimes_\C R_l).
\]

By construction, the $\cO$-module pullback
$Li_b^*\wt{\cN(!D_I)}$
is isomorphic to
\[
\cRhom_{\widetilde{\sA_X^{R\otimes_\C R_l}}/(b)}
\left(
\widetilde{\Rhom}_{\sB}
\left(
\cS_I/b\cS_I,
\sB\otimes_{\cO_X}\cS_{G_I}^*
\right),
\widetilde{\sA_X^{R\otimes_\C R_l}}/(b)
\right),
\]
where, for simplicity, we write
\[
\sB
=
\sA_X^{(R\otimes_\C R_{l-|I|})/(b)}
\left[1/\prod_{i\in I}g_i\right],
\]
and where
$i_b\colon
X\times (b=0)
\hookrightarrow
X\times \Specan(R\otimes_\C R_l)$
is the analytic closed embedding. By Rees theorem \cite[Thm. 8.34]{Rtha}, we
obtain
\[
Li_b^*\wt{\cN(!D_I)}
\simeq
\wt{\cN_1(!D_I)}.
\]

Consider the hyperhomology spectral sequence for $Li_b^*\wt{\cN(!D_I)}$:
\[
E^2_{p,q}
=
L^p i_b^*
\left(
\cH^{-q}\big(\wt{\cN(!D_I)}\big)
\right)
\Longrightarrow
\cH^{-p-q}
\left(
Li_b^*\wt{\cN(!D_I)}
\right).
\]
We also consider the complex of right
$\sA^{R\otimes_\C R_{l-|I|}}(*\bar D_I)$-modules
\[
\iota_{G_I,+}
\left(
\Rhom_{\sA_X^{R\otimes_\C R_{l-|I|}}
[1/\prod_{i\in I}g_i]}
\left(
\cS_I,
\sA_X^{R\otimes_\C R_{l-|I|}}
\left[1/\prod_{i\in I}g_i\right]
\right)
\right),
\]
which is isomorphic to
\[
\Rhom_{\sA_X^{R\otimes_\C R_{l-|I|}}
[1/\prod_{i\in I}g_i]}
\left(
\cS_I,
\sA_X^{R\otimes_\C R_{l-|I|}}\otimes_{\cO_X}\cS_{G_I}^*
\right).
\]
For $k\ge j$, set
\[
\cE^k
\coloneqq
\operatorname{Ext}^k_{\sA_X^{R\otimes_\C R_{l-|I|}}
[1/\prod_{i\in I}g_i]}
\left(
\cS_I,
\sA_X^{R\otimes_\C R_{l-|I|}}\otimes_{\cO_X}\cS_{G_I}^*
\right).
\]
Then $\cE^k$ is an
$\sA_{Y_I}^{R\otimes_\C R_{l-|I|}}(*\bar D_I)$-module, and it is also
relative regular holonomic over
$\sA_{Y_I}^{R\otimes_\C R_{l-|I|}}$. Choose an
$\sA_{Y_I,D_I}^{R\otimes_\C R_{l-|I|}}$-lattice
$\overline{\cE^k}$ of $\cE^k$. By Theorem \ref{thm:relGins-Sab},
Lemma \ref{lm:ausdual}, and Lemma \ref{lm:GE-type}, we have
\[
j_{\sA_X^{R\otimes_\C R_l}}(\overline{\cE^k})
=
j_{\sA_X^{R\otimes_\C R_{l-|I|}}
[1/\prod_{i\in I}g_i]}(\cE^k).
\]
By Corollary \ref{cor:anlocmaxminex}, and since analytification is faithfully
flat, we conclude that
\begin{equation}\label{eq:extendinggraden}
j_{\wt{\sA_X^{R\otimes_\C R_l}}}(\wt{\cE^k})
=
j_{\sA_X^{R\otimes_\C R_{l-|I|}}
[1/\prod_{i\in I}g_i]}(\cE^k).
\end{equation}
By Auslander regularity, it follows that
$\cH^i\big(\wt{\cN(!D_I)}\big)=0$
for $i<0$.

Define
\[
m=\max\left\{
i\mid
\cH^i\big(\wt{\cN(!D_I)}\big)\neq 0
\right\}.
\]
If $m\neq 0$, then the boundary term
\[
E^2_{0,-m}
=
i_b^*
\left(
\cH^m\big(\wt{\cN(!D_I)}\big)
\right)
\]
is a subquotient of
$\cH^m\big(\wt{\cN_1(!D_I)}\big)$
by convergence of the spectral sequence. Since $m\neq 0$, Nakayama's lemma
for relative holonomic modules \cite[Prop. 1.9(2)]{FS17} gives
\[
i_b^*
\left(
\cH^m\big(\wt{\cN(!D_I)}\big)
\right)
\neq 0.
\]
Thus, 
$\cH^m\big(\wt{\cN_1(!D_I)}\big)\neq 0,$
contradicting the inductive hypothesis. Therefore,
\[
\wt{\cN(!D_I)}
\stackrel{\mathrm{q.i.}}{\simeq}
\cH^0\big(\wt{\cN(!D_I)}\big).
\]

Moreover, $\wt{\cN(!D_I)}$ is pure by \eqref{eq:extendinggraden} and
Lemma \ref{lm:dualalPF}. At the same time, $\wt{\cN(!D_I)}$ has no nonzero
quotient supported on the boundary divisor; otherwise, pulling back this
quotient by $i_b^*$ would produce a nonzero quotient of
$\wt{\cN_1(!D_I)}$, again contradicting the inductive hypothesis.

To prove injectivity, we apply the same hyperhomology spectral sequence
argument to
\[
Li_b^*
\left(
\wt{\cN(!D_I)}
\longrightarrow
\wt{\cN(*D_I)}
\right).
\]
This shows that
$\wt{\cN(!D_I)}
\longrightarrow
\wt{\cN(*D_I)}$
is injective.

It remains to prove the base case, where
\[
\dim
\supp_{Y_I\setminus \bar D_I}^{R\otimes_\C R_{l-|I|}}
(\cN|_{Y\setminus D})
=
0.
\]
In this case, after forgetting the $R\otimes_\C R_{l-|I|}$-module structure,
$\cN(*D_I)$ becomes a holonomic $\shD_{Y_I}$-module supported on the graph of
$G_I$. By Rees theorem again, the base case follows from
\cite[Lem.3.7]{WuRHA}.
\end{proof}

\subsection{The Artin--Rees Theorem for $\sA^R_{Y,D}$-modules}
\label{subsect:logArtinRees}

For $\bm\in \N^l$, we define
\[
\sB^R_\bm
\coloneqq
\bigoplus_{i\in \Z_{\ge 0}}
\alpha(\bm)^i\sA^R_{Y,D}\cdot u^i
\subseteq
\sA^R_{Y,D}[u].
\]

\begin{lemma}
The ring $\sB^R_\bm$ is graded noetherian.
\end{lemma}

\begin{proof}
Since
\[
\alpha(\bm)^i\cdot\sA^R_{Y,D}
=
\sA^R_{Y,D}\cdot\alpha(\bm)^i
\]
for each $i$, the algebra $\sB^R_\bm$ is a graded ring.
The filtration $F^\dagger_\bullet\sA^R_{Y,D}$ induces filtrations on
$\alpha(\bm)^i\sA^R_{Y,D}$, and hence an induced filtration
$F^\dagger_\bullet\sB^R_\bm$ such that
\[
\gr^\dagger_\bullet\sB^R_\bm
\simeq
\gr^\dagger_\bullet\sA^R_{Y,D}[\alpha(\bm)u].
\]
Thus $\gr^\dagger_\bullet\sB^R_\bm$ is noetherian. By
\cite[A.III 1.27]{BJ}, $\sB^R_\bm$ is also noetherian.
\end{proof}

\begin{theorem}[Logarithmic Artin--Rees]\label{thm:artin-reeslog}
Let $\cM$ be a finitely generated $\sA^R_{Y,D}$-module, and let
$\cN\subseteq \cM$ be an $\sA^R_{Y,D}$-submodule. Then there exists
$n_0\in \Z_{\ge 0}$ such that
\[
\alpha(\bm)^n\cdot\cM\cap \cN
\subseteq
\alpha(\bm)^{n-n_0}\cN
\quad
\textup{for all } n\ge n_0.
\]
\end{theorem}

\begin{proof}
Define
\[
R_\bm\cM
=
\bigoplus_{i\in \Z_{\ge 0}}
\alpha(\bm)^i\cM\cdot u^i.
\]
This is a finitely generated graded $\sB^R_\bm$-module. We also have a
graded $\sB^R_\bm$-submodule
\[
R'_\bm\cN
=
\bigoplus_{i\in \Z_{\ge 0}}
(\cN\cap \alpha(\bm)^i\cM)\cdot u^i
\subseteq
R_\bm\cM.
\]
Since $\sB_\bm^R$ is noetherian, $R'_\bm\cN$ is finitely generated over
$\sB_\bm^R$. Choose homogeneous generators
$a_1u^{i_1},\dots,a_Nu^{i_N}$
of $R'_\bm\cN$, and set
\[
n_0=\max\{i_1,\dots,i_N\}.
\]
Then, for $n\ge n_0$ and every
$m\in \alpha(\bm)^n\cM\cap \cN,$
we have
\[
m\cdot u^n
=
\sum_{j=1}^N
\big(\alpha(\bm)^{n-i_j}P_j\cdot a_j\big)\cdot u^n
\]
for some $P_j\in \sA^R_{Y,D}$. Hence
$m\in \alpha(\bm)^{n-n_0}\cdot\cN.$
\end{proof}

\begin{remark}
Since
$\alpha(K)\cdot\sA^R_{Y,D}
=
\sA^R_{Y,D}\cdot\alpha(K)$
for every monoid ideal $K\subseteq \N^l$, by Lemma \ref{lm:twosidedideal},
Theorem \ref{thm:artin-reeslog} remains valid if one replaces
$\alpha(\bm)$ by $\alpha(K)$. Notice that every monoid ideal in $\N^l$ is
finitely generated.
\end{remark}

\subsection{Relative $!$-pullback along log strata and duality}
\label{subsect:r!pullback}

Let $\cN$ be a finitely generated $\sA^R_{Y,D}$-module of GE-type with
respect to
$G=(g_1,\dots,g_l)$
such that $\cN$ is relative regular holonomic over
$\sA^{R\otimes_\C R_l}_X$, with the induced
$\sA^{R\otimes_\C R_l}_X$-module structure given by
\eqref{eq:uncaningen}. Let
$I\subseteq \{1,2,\dots,l\}$
be a nonempty subset. Let $\cN_I$ denote the image of
\[
\cN\longrightarrow \cN(*D_I).
\]
Then $\cN(*D_I)$ is an
$\sA^{R\otimes_\C R_{l-|I|}}_{Y_I}$-module supported on the graph of
$(g_i)_{i\in I}$. By \eqref{eq:getypeD}, we have
\[
\cN(*D_I)
=
\cN_I\left[1/\prod_{i\in I}g_i\right]
\simeq
\cS_I[s_i]_{i\in I}\prod_{i\in I}g_i^{s_i}.
\]
By \cite[Cor.4.31]{FFS23}, $\cN(*D_I)$ is also relative regular holonomic
over $\sA^{R\otimes_\C R_l}_X$. Since $\cN(*D_I)$ is flat over
$\C[s_i]_{i\in I}$, the module $\cS_I$ is relative regular holonomic over
$\sA^{R\otimes_\C R_{l-|I|}}_X$. Consequently, by Remark
\ref{rmk:gnr!ex} and Theorem \ref{thm:anarelbeilinson}, we have
\[
\wt{\cN(!D_I)}
\hookrightarrow
\wt{\cN(*D_I)}
\]
for each nonempty subset $I$.

\begin{lemma}\label{lm:!exinclude}
With the notation and assumptions above, if
\[
I\subseteq J\subseteq \{1,2,\dots,l\}
\]
are nonempty subsets, then
\[
\wt{\cN(!D_J)}
\hookrightarrow
\wt{\cN(!D_I)}
\hookrightarrow
\wt{\cN}.
\]
\end{lemma}

\begin{proof}
By construction, we have a short exact sequence of $\sA^R_{Y,D}$-modules
\[
0\to \cK\to \cN\to \cN_I\to 0,
\]
where
$\cK=H^0_{D_I}(\cN)$
is the 0-th local cohomology module. Since $\sA^R_{Y,D}$ is noetherian,
$\cK$ is finitely generated; say
\[
\cK=\sum_{q=1}^p\sA^R_{Y,D}\cdot m_q.
\]
Since $\cN_I$ has no $t_I$-torsion, we obtain another short exact sequence
\[
0\to t_I^j\cK\to t_I^j\cN\to t_I^j\cN_I\to 0
\]
for every $j\ge 0$. By Lemma \ref{lm:twosidedideal},
\[
t_I\cdot\sA^R_{Y,D}
=
\sA^R_{Y,D}\cdot t_I.
\]
Thus
\[
t_I^j\cK
=
t_I^j\cdot\sum_{q=1}^p\sA^R_{Y,D}\cdot m_q
=
\sum_{q=1}^p\sA^R_{Y,D}\cdot(t_I^j\cdot m_q).
\]
Since $\cK$ is supported on $D_I$, the Nullstellensatz implies that
$t_I^j\cK=0$
for $j\gg 0$. Hence
\begin{equation}\label{eq:naturalkilltorsion}
t_I^j\cdot \cN=t_I^j\cdot \cN_I
\end{equation}
for $j\gg 0$. Finally, since $\cN_I$ is an
$\sA^{R\otimes_\C R_{l-|I|}}_{Y_I,\bar D_I}$-lattice of GE-type with
respect to $(g_i)_{i\in I}$, the required inclusions follow from
\eqref{eq:relminexge}.
\end{proof}

For $I\subseteq \{1,\dots,l\}$, let $K_I\subseteq \N^l$ be the monoid ideal
generated by $\{e_i\}_{i\in I}$. Then we have the closed embedding
\[
i_{K_I}=i_{D^I}\colon
Z_{K_I}=D^I\coloneqq \bigcap_{i\in I}D_i
\hookrightarrow Y.
\]
For
\[
J=\{i_1,\dots,i_{|J|}\}\subseteq I,
\qquad
i_1<i_2<\cdots<i_{|J|},
\]
we denote the corresponding standard vector by
\[
e_J=e_{i_1}\wedge e_{i_2}\wedge \cdots \wedge e_{i_{|J|}},
\]
and its dual by
\[
e^*_J=e^*_{i_1}\wedge e^*_{i_2}\wedge \cdots \wedge e^*_{i_{|J|}}.
\]

We define
\[
\wt{i^!_{K_I}}\cN
=
\wt{i^!_{K_I}}^X\cN
\]
to be the following Koszul complex of relative regular holonomic
$\wt{\sA^{R\otimes_\C R_l}_X}$-modules:
\[
0\to
\wt{\cN(!D_I)}\cdot e_I
\to
\bigoplus_{\substack{J\subseteq I\\ |J|=|I|-1}}
\wt{\cN(!D_J)}\cdot e_J
\to \cdots \to
\bigoplus_{\substack{J\subseteq I\\ |J|=1}}
\wt{\cN(!D_J)}\cdot e_J
\to
\wt{\cN}
\to 0,
\]
where the differential is given by
\[
m\cdot e_J
=
m\cdot e_{i_1}\wedge\cdots\wedge e_{i_{|J|}}
\longmapsto
\sum_{q=1}^{|J|}
(-1)^{q+1}
m\cdot
e_{i_1}\wedge\cdots\wedge
\widehat{e_{i_q}}
\wedge\cdots\wedge e_{i_{|J|}},
\]
for $m\in \wt{\cN(!D_J)}$. The last term $\wt\cN$ is placed in
cohomological degree $0$. By Lemma \ref{lm:!exinclude},
$\wt{i^!_{K_I}}\cN$ is well-defined.

More generally, let $K\subseteq \N^l$ be a monoid ideal, and assume that
$K$ is generated by an ordered set
\[
I_K=\{v_1,\dots,v_p\},
\qquad
v_i\in K.
\]
We define
\[
\wt{i^!_{K}}\cN
=
\wt{i^!_{K}}^X\cN
\]
to be the analogous Koszul complex of relative regular holonomic
$\wt{\sA^{R\otimes_\C R_l}_X}$-modules:
\[
0\to
\wt{\cN(!D_{I_K})}\cdot e_{I_K}
\to
\bigoplus_{\substack{J_K\subseteq I_K\\ |J_K|=|I_K|-1}}
\wt{\cN(!D_{J_K})}\cdot e_{J_K}
\to \cdots \to
\bigoplus_{\substack{J_K\subseteq I_K\\ |J_K|=1}}
\wt{\cN(!D_{J_K})}\cdot e_{J_K}
\to
\wt{\cN}
\to 0,
\]
where
$i_K\colon Z_K\hookrightarrow Y$
is the closed embedding and
$D_{J_K}
=
D_{\alpha(\sum_{v_i\in J_K}v_i)}.$

\begin{theorem}\label{thm:upper!DI}
We have
\[
\wt{i^!_{K}}\cN
\stackrel{\mathrm{q.i.}}{\simeq}
\cH^0(\wt{i^!_{K}}\cN)
\simeq
\frac{\wt\cN}{\sum_{v_i\in I_K}\wt{\cN(!D_{v_i})}}.
\]
In particular, for $K=K_I$,
\[
\wt{i^!_{K_I}}\cN
\stackrel{\mathrm{q.i.}}{\simeq}
\frac{\wt\cN}{\sum_{i\in I}\wt{\cN(!D_i)}}.
\]
\end{theorem}

\begin{proof}
We prove the assertion for $K=K_I$, where
$I\subseteq \{1,\dots,l\}$. The general case is similar.

Choose $i_1\in I$ and set
\[
I_1=I\setminus\{i_1\}.
\]
Let $C_1^\bullet$ be the complex
\[
0\to
\wt{\cN(!D_{I_1})}\cdot e_{I_1}
\to
\bigoplus_{\substack{J\subseteq I_1\\ |J|=|I_1|-1}}
\wt{\cN(!D_J)}\cdot e_J
\to \cdots \to
\bigoplus_{\substack{J\subseteq I_1\\ |J|=1}}
\wt{\cN(!D_J)}\cdot e_J
\to
\wt{\cN}
\to 0,
\]
and let $C_2^\bullet$ be the subcomplex
\[
0\to
\wt{\cN(!D_I)}\cdot e_I
\to
\bigoplus_{\substack{J\subseteq I\\ |J|=|I|-1\\ i_1\in J}}
\wt{\cN(!D_J)}\cdot e_J
\to \cdots \to
\bigoplus_{\substack{J\subseteq I\\ |J|=2\\ i_1\in J}}
\wt{\cN(!D_J)}\cdot e_J
\to
\wt{\cN(!D_{\{i_1\}})}
\to 0.
\]
By construction, $\wt{i^!_{D^I}}\cN$ is the cone of
\[
C_2^\bullet\longrightarrow C_1^\bullet.
\]
We argue by induction on $|I|$. By the inductive hypothesis, we may assume
that
\[
C_1^\bullet
\stackrel{\mathrm{q.i.}}{\simeq}
\cH^0(C_1^\bullet)
\simeq
\frac{\wt\cN}{\sum_{i\in I_1}\wt{\cN(!D_i)}}
\]
and
\[
C_2^\bullet
\stackrel{\mathrm{q.i.}}{\simeq}
\cH^0(C_2^\bullet)
\simeq
\frac{\wt{\cN(!D_{\{i_1\}})}}
{\sum_{i\in I_1}\wt{\cN(!D_{\{i,i_1\}})}}.
\]
Thus $\wt{i^!_{D^I}}\cN$ is quasi-isomorphic to the two-term complex
\[
\left[
\frac{\wt{\cN(!D_{\{i_1\}})}}
{\sum_{i\in I_1}\wt{\cN(!D_{\{i,i_1\}})}}
\longrightarrow
\frac{\wt\cN}{\sum_{i\in I_1}\wt{\cN(!D_i)}}
\right].
\]

By \eqref{eq:relminexge} and \eqref{eq:naturalkilltorsion}, locally over
$\Specan(R\otimes_\C R_l)$ we may identify
\[
\frac{\wt\cN}{\sum_{i\in I_1}\wt{\cN(!D_i)}}
\simeq
\widetilde{
\left(
\frac{\cN}{\sum_{i\in I_1}t_i^{k_1}\cN}
\right)}
\]
and
\[
\frac{\wt{\cN(!D_{\{i_1\}})}}
{\sum_{i\in I_1}\wt{\cN(!D_{\{i,i_1\}})}}
\simeq
\widetilde{
\left(
\frac{t_{i_1}^{k}\cN}
{\sum_{i\in I_1}t_i^{k_1}t_{i_1}^{k}\cN}
\right)}
\]
for sufficiently large integers $k_1$ and $k$. We then consider the natural
map
\[
\frac{t_{i_1}^{k}\cN}
{\sum_{i\in I_1}t_i^{k_1}t_{i_1}^{k}\cN}
\longrightarrow
\frac{\cN}{\sum_{i\in I_1}t_i^{k_1}\cN}.
\]
Its kernel is
\[
\frac{
t_{i_1}^{k}\cN\cap \sum_{i\in I_1}t_i^{k_1}\cN
}{
\sum_{i\in I_1}t_i^{k_1}t_{i_1}^{k}\cN
}.
\]
By Theorem \ref{thm:artin-reeslog}, this kernel vanishes for $k\gg 0$,
after increasing the exponents if necessary. Since analytification is exact,
the morphism
\[
\frac{\wt{\cN(!D_{\{i_1\}})}}
{\sum_{i\in I_1}\wt{\cN(!D_{\{i,i_1\}})}}
\longrightarrow
\frac{\wt\cN}{\sum_{i\in I_1}\wt{\cN(!D_i)}}
\]
is injective. Hence
\[
\wt{i^!_{D^I}}\cN
\stackrel{\mathrm{q.i.}}{\simeq}
\cH^0
\left(
\frac{\wt{\cN(!D_{\{i_1\}})}}
{\sum_{i\in I_1}\wt{\cN(!D_{\{i,i_1\}})}}
\longrightarrow
\frac{\wt\cN}{\sum_{i\in I_1}\wt{\cN(!D_i)}}
\right)
\simeq
\frac{\wt\cN}{\sum_{i\in I}\wt{\cN(!D_i)}}.
\]
When $|I|=1$, the base case follows from Lemma \ref{lm:!exinclude}; the
induction is complete.
\end{proof}

We denote by $\wt\D$ the analytic duality functor for
$\wt{\sA^{R\otimes_\C R_l}_X}$-modules:
\[
\wt\D(\bullet)
\coloneqq
\cRhom_{\wt{\sA^{R\otimes_\C R_l}_X}}
\big(
\bullet,
\wt{\sA^{R\otimes_\C R_l}_X}
\big).
\]
For $I\subseteq\{1,\dots,l\}$, we write
\[
\wt{D}^I
=
(g_i=0)_{i\in I}
\subseteq
X\times \Specan(R\otimes_\C R_l),
\]
and
\[
\wt{D}_I
=
\left(\prod_{i\in I}g_i=0\right)
\subseteq
X\times \Specan(R\otimes_\C R_l).
\]
We define $\wt{D}_\bm$ similarly for $\bm\in \N^l$.

Let $I^c$ denote the complement of $I$ in $\{1,2,\dots,l\}$, and set
\[
D_o^I=D^I\setminus D_{I^c},
\qquad
\wt D_o^I=\wt D^I\setminus \wt D_{I^c}.
\]
Notice that $D^I\subseteq Y$ is smooth, whereas
$\wt D^I\subseteq X\times \Specan(R\otimes_\C R_l)$ may be singular.

More generally, for a monoid ideal $K\subseteq \N^l$, set
\[
\wt Z_K=\wt Z_K^X
\coloneqq
\iota_G^*Z_K\times \Specan(R\otimes_\C R_l)
\subseteq
X\times \Specan(R\otimes_\C R_l),
\]
where
\[
D_X^K\coloneqq \iota_G^*Z_K
\]
is the subscheme of $X$ obtained by pulling back $Z_K$. For
$\bm\in \N^l$, let $K_\bm$ denote the localization of $K$ at $\bm$, that is,
the monoid ideal obtained by adjoining $-\bm$ to $K$; see
\cite[Prop.1.4.4]{Ogusbook}. Thus $K_\bm$ is a monoid ideal of
$\N_\bm$, and $\N_\bm$ corresponds to the localization of the log structure
on $Y$ at $\alpha(\bm)$. When $\bm=0$, we have $K_\bm=K$.
Then we have the locally closed embedding
\[
i_{K_\bm}\colon
Z_{K_\bm}
\hookrightarrow
Y\setminus D_\bm
\hookrightarrow
Y
\]
and the locally closed subscheme
$D_X^{K_\bm}
\coloneqq
\iota_G^*Z_{K_\bm}
\subseteq X.$

\begin{prop}\label{prop:!dual}
For a monoid ideal $K\subseteq \N^l$, we have
\[
\wt{i^!_K}\cN
\simeq
\wt\D\big(
R\Gamma_{[\wt Z_K]}(\wt\D\wt\cN)
\big),
\]
where $R\Gamma_{[\wt Z_K]}$ is the derived algebraic local cohomology
functor for $\wt{\sA^{R\otimes_\C R_l}_X}$-modules. In particular,
$\wt{i^!_{K}}\cN$ is independent of the choice of the generating set
$I_K$ of $K$.
\end{prop}

From the preceding proposition, $\wt{i^!_{K}}\cN$ is compatible, under the
relative regular Riemann--Hilbert correspondence, with the
$i_{\wt Z_K}^{-1}$ functor for relative constructible sheaves defined in
\cite[\S 2]{FS13}, where
\[
i_{\wt Z_K}\colon
\wt Z_K\hookrightarrow X\times \Specan(R\otimes_\C R_l)
\]
is the analytic closed embedding. The proof of Proposition
\ref{prop:!dual} will be given later.

\begin{remark}\label{rmk:localizingupper!}
Let $I'\subseteq\{1,\dots,l\}$ be another subset.

\noindent
$(1)$ The minimal extension
\[
\wt{\cN(*D_{I'})(!D_I)}
\]
of $\cN(*D_{I'})$ along $D_I$ is defined by
\[
\wt \D
\left(
\cRhom_{\wt{\sA_{X}^{R\otimes_\C R_{l-|I|}}}(*\wt D_I)}
\left(
\wt{\cS^{I'}_I},
\wt{\sA_{X}^{R\otimes_\C R_{l-|I|}}}
\otimes_{\cO_X}
\cS_{G_I}^*
\right)
\right),
\]
where $\cS^{I'}_I$ is the relative regular holonomic
$\sA_{X}^{R\otimes_\C R_{l-|I|}}$-module satisfying
\[
\cN(*D_{I\cup I'})
\simeq
\iota_{G_I,+}\cS^{I'}_I
\]
by Kashiwara's equivalence. By Corollary \ref{cor:anlocmaxminex},
$\wt{t_{I\cup I'}^{-m}\cN_{I\cup I'}}$ and
$\wt{\cN(*D_{I\cup I'})}$ agree locally over
$\Specan(R\otimes_\C R_l)$ for $m\gg 0$. Hence
$\wt{t_{I\cup I'}^{-m}\cN_{I\cup I'}(!D_I)}$ and
$\wt{\cN(*D_{I'})(!D_I)}$ also agree locally over
$\Specan(R\otimes_\C R_l)$. Gluing the local data, Theorem
\ref{thm:upper!DI} and Proposition \ref{prop:!dual} remain valid after
replacing $\cN$ by $\cN(*D_{\bm})$ for every $\bm\in \N^l$, where $I'$ is
chosen so that the support of $D_\bm$ is $D_{I'}$. In particular,
$\wt{i^!_K}\cN(*D_\bm)$ is also well-defined for every $\bm\in \N^l$, and
it is still a relative regular holonomic
$\wt{\sA_X^{R\otimes_\C R_l}}$-module.

\noindent
$(2)$ The algebraic localization of $\wt{\cN(!D_I)}$ along $D_{I'}$ is
\[
\wt{\cN(!D_I)(*D_{I'})}
\coloneqq
\wt{\cN(!D_I)}(*\wt D_{I'}).
\]
On the other hand, one can directly take the minimal extension of
$\cN(*D_{I'})$ on $Y\setminus D_{I'}$ with respect to the log structure
given by
\[
\N^{l-|I'|}
\hookrightarrow
S[t_i^{\pm1}]_{i\in I'}[\N^{l-|I'|}].
\]
More precisely, $\cN(*D_{I'})$ is a finitely generated
$\sA^{R}_{Y,D}(*D_{I'})$-module and a relative regular holonomic
$\sA^{R\otimes_\C R_{I'}}_{Y_{I'}}(*\bar D_{I'})$-module. This gives the
minimal extension
\[
\wt{\cN(*D_{I'})(!D_I)}^{Y_{I'}\setminus \bar D_{I'}},
\]
which is a relative regular holonomic
$\wt{\sA^{R\otimes_\C R_{I'}}_{Y_{I'}}(*\bar D_{I'})}$-module. Since
$\cN$ is relative regular holonomic over $\sA^{R\otimes_\C R_l}_X$, the
module
\[
\wt{\cN(*D_{I'})(!D_I)}^{Y_{I'}\setminus \bar D_{I'}}
\]
is relative regular holonomic over
$\wt{\sA^{R\otimes_\C R_{I'}}_{Y_{I'}}}$ and supported on the graph of
$G_{I'}$. By \eqref{eq:relgetype}, it is a
\[
\wt{\sA^{R\otimes_\C R_{I'}}_X}[s_i]_{i\notin I'}
=
\wt{\sA^{R\otimes_\C R_{I'}}_X}
\otimes_\C
\C[s_i]_{i\notin I'}
\]
module. By \cite[Thm. 5.15]{SS94}; see also \cite[Prop. 1.7]{FS17}, and
by \eqref{eq:relgetype}, we have a canonical isomorphism
\[
\wt{\cN(!D_I)(*D_{I'})}
\simeq
\widetilde{
\wt{\cN(*D_{I'})(!D_I)}^{Y_{I'}\setminus \bar D_{I'}}
},
\]
where the outer $\widetilde{\phantom M}$-functor turns the
$\wt{\sA^{R\otimes_\C R_{I'}}_X}[s_i]_{i\notin I'}$-module
\[
\wt{\cN(*D_{I'})(!D_I)}^{Y_{I'}\setminus \bar D_{I'}}
\]
into a
\[
\wt{\sA^{R\otimes_\C R_l}_X}
\]
module.

\noindent
$(3)$ Similarly, there are two ways to localize $\wt{i^!_{D^I}}\cN$:
\[
(\wt{i^!_{D^I}}\cN)(*\wt D_{I'})
\quad
\textup{and}
\quad
\wt{i^!_{D^I}}^{Y_{I'}\setminus \bar D_{I'}}\cN.
\]
There is a canonical isomorphism
\[
(\wt{i^!_{D^I}}\cN)(*\wt D_{I'})
\simeq
\widetilde{
\wt{i^!_{D^I}}^{Y_{I'}\setminus \bar D_{I'}}(\cN)
}.
\]
We then define
\[
\wt{i^!_{D_o^I}}\cN
\coloneqq
(\wt{i^!_{D^I}}\cN)(*\wt D_{I^c}),
\]
which is still a relative regular holonomic
$\wt{\sA^{R\otimes_\C R_l}_{X}}$-module.

More generally, for $\bm,\bn\in \N^l$ and for a monoid ideal
$K\subseteq \N^l$, we define
\[
\wt{i^!_{K_\bm}}\cN(*D_\bn)
\coloneqq
\big(\wt{i^!_{K}}\cN(*D_\bn)\big)(*\wt D_{\bm}).
\]
\end{remark}

By Theorem \ref{thm:upper!DI} and Proposition \ref{prop:!dual}, we
immediately obtain the following.

\begin{coro}\label{cor:exactupper!}
Let $\cN$ be a finitely generated $\sA^R_{Y,D}$-module of GE-type with
respect to
$G=(g_1,\dots,g_l)$
such that $\cN$ is relative regular holonomic over
$\sA^{R\otimes_\C R_l}_X$. Let
\[
0\to \cN_1\to \cN\to \cN_2\to 0
\]
be a short exact sequence of $\sA^R_{Y,D}$-modules. Then we obtain a short
exact sequence of relative regular holonomic
$\wt{\sA^{R\otimes_\C R_l}_X}$-modules:
\[
0\to
\wt{i^!_{K}}\cN_1
\to
\wt{i^!_{K}}\cN
\to
\wt{i^!_{K}}\cN_2
\to 0.
\]
\end{coro}

\begin{coro}\label{cor:latticesinc}
Let $\cN$ be a finitely generated $\sA^R_{Y,D}$-module of GE-type with
respect to
$G=(g_1,\dots,g_l)$
such that $\cN(*D)$ is relative regular holonomic over $\sA^R_Y$. For
$\bm,\bn\in \N^l$, if
$k\bn-\bm\in \N^l$
for some $k\in \N$, then we have an inclusion
\[
\wt{i^!_{K}}\cN(*D_\bm)
\hookrightarrow
\wt{i^!_{K}}\cN(*D_\bn).
\]
\end{coro}

\begin{proof}
For every $q\in \N$, we have a short exact sequence
\[
0\to
\alpha(\bm)^{-q}\cN
\to
\alpha(k\bn)^{-q}\cN
\to
\frac{\alpha(k\bn)^{-q}\cN}{\alpha(\bm)^{-q}\cN}
\to 0.
\]
By Corollary \ref{cor:exactupper!}, we have
\[
\wt{i^!_{K}}\big(\alpha(\bm)^{-q}\cN\big)
\hookrightarrow
\wt{i^!_{K}}\big(\alpha(k\bn)^{-q}\cN\big).
\]
By Corollary \ref{cor:anlocmaxminex}, the natural morphism
\[
\wt{i^!_{K}}\cN(*D_\bm)
\longrightarrow
\wt{i^!_{K}}\cN(*D_\bn)
\]
is injective locally on $X\times U$ for every sufficiently small analytic
open subset
\[
U\subseteq \Specan(R\otimes_\C R_l).
\]
Therefore it is injective globally.
\end{proof}

\subsection{Algebraic local cohomology of lattices}

Let $\cN$ be an $\sA^R_{Y,D}$-lattice of GE-type with respect to
$G=(g_1,\dots,g_l)$
such that $\cN(*D)$ is a relative regular holonomic $\sA^R_Y$-module. For a
monoid ideal $K\subseteq \N^l$ and for $\bm\in \N^l$, we consider the derived
algebraic local cohomology complexes
\[
R\Gamma_{[Z_K]}\cN
\quad \textup{and} \quad
R\Gamma_{[Z_{K_\bm}]}\cN .
\]
By construction (cf. \cite[\S 3.4]{Kasbook}), if
$I_K=\{v_1,\dots,v_p\}$
is a set of generators of $K$, then $R\Gamma_{[Z_K]}\cN$ is the Koszul-type
complex
\[
0\to
\cN
\to
\bigoplus_{\substack{J\subseteq I_K\\ |J|=1}}
\cN(*D_J)\cdot e^*_J
\to \cdots \to
\bigoplus_{\substack{J\subseteq I_K\\ |J|=p-1}}
\cN(*D_J)\cdot e^*_J
\to
\cN(*D_{I_K})\cdot e^*_{I_K}
\to 0,
\]
with differential
\[
m\cdot e^*_J
\longmapsto
\sum_{v_i\in I_K\setminus J} m\cdot e_i^*\wedge e_J^* .
\]
Moreover,
\[
R\Gamma_{[Z_{K_\bm}]}\cN
=
R\Gamma_{[Z_K]}\cN\otimes_{\cO_Y}\cO_Y(*D_\bm).
\]

By Theorem \ref{thm:mainsuppfiberation}, both
\[
R\Gamma_{[Z_K]}\cN
\quad \textup{and} \quad
R\Gamma_{[Z_{K_\bm}]}\cN
\]
are relative regular holonomic over $\sA_X^{R\otimes_\C R_l}$. Similarly, when
$K=K_I$, since $\cN(*D_{I^c})$ is of GE-type with respect to
$G_I=(g_i)_{i\in I}$, the complex $R\Gamma_{[D^I_o]}\cN$ is relative regular
holonomic over $\sA_{Y_I}^{R\otimes_\C R_I}$.

By construction, we have
\begin{equation}\label{eq:bfafloc}
\cN(*D_J)
=
\cN\left[1/\prod_{v_i\in J}\alpha(v_i)\right]
\end{equation}
for every $J\subseteq I_K$. Therefore,
\[
\widetilde{R\Gamma_{[Z_K]}\cN}
\simeq
R\Gamma_{[\wt Z_K]}\wt\cN .
\]

\begin{proof}[Proof of Proposition \ref{prop:!dual}]
Let $I\subseteq \{1,\dots,l\}$. Since $\wt D^I$ is globally defined by
$(g_i=0)_{i\in I}$, the complex $R\Gamma_{[\wt D^I]}(\wt\cO)$ is
\[
0\to
\wt\cO
\to
\bigoplus_{\substack{J\subseteq I\\ |J|=1}}
\wt\cO\left[1/\prod_{i\in J}g_i\right]\cdot e^*_J
\to \cdots \to
\bigoplus_{\substack{J\subseteq I\\ |J|=|I|-1}}
\wt\cO\left[1/\prod_{i\in J}g_i\right]\cdot e^*_J
\to
\wt\cO\left[1/\prod_{i\in I}g_i\right]\cdot e^*_I
\to 0,
\]
with differential
\[
m\cdot e^*_J
=
m\cdot e^*_{i_1}\wedge\cdots\wedge e^*_{i_{|J|}}
\longmapsto
\sum_{i\in I\setminus J}
m\cdot e_i^*\wedge e^*_{i_1}\wedge\cdots\wedge e^*_{i_{|J|}}.
\]
Here $\wt\cO$ is the sheaf of holomorphic functions on
$X\times \Specan(R\otimes_\C R_l).$
Moreover, by \cite[Prop.3.35]{Kasbook}, we have
\[
R\Gamma_{[\wt D^I]}(\wt\D\wt\cN)
\simeq
R\Gamma_{[\wt D^I]}(\wt\cO)
\otimes^L_{\wt\cO}
\wt\D\wt\cN .
\]
Consequently, by the definition of $\wt{i^!_{D^I}}\cN$, together with
\eqref{eq:bfafloc} and \eqref{eq:isolocdualKE}, we obtain the required
statement for $K=K_I$. The general case is similar.
\end{proof}
Let $K\subseteq \N^l$ be a monoid ideal generated by
$I_K=\{v_1,\dots,v_p\}$.
For $q\in \Z_{\ge 0}$, we have intermediate Koszul-type complexes
\[
0\to
\cN
\to
\bigoplus_{\substack{J\subseteq I_K\\ |J|=1}}
t_J^{-q}\cN\cdot e^*_J
\to \cdots \to
\bigoplus_{\substack{J\subseteq I_K\\ |J|=p-1}}
t_J^{-q}\cN\cdot e^*_J
\to
t_{I_K}^{-q}\cN\cdot e^*_{I_K}
\to 0,
\]
which we denote by $\cK^{I_K}(\cN,q)$, where
\[
t_J=\alpha\left(\sum_{v_i\in J}v_i\right)
\]
for $J\subseteq I_K$. As $q\to+\infty$, we have direct systems
\[
\{\cK^{I_K}(\cN,q)\to \cK^{I_K}(\cN,q+1)\}
\]
and
\[
\{\wt{\cK^{I_K}}(\cN,q)\to \wt{\cK^{I_K}}(\cN,q+1)\},
\]
where the second direct system is the analytic sheafification of the first.

By Corollary \ref{cor:anlocmaxminex}, for every sufficiently small analytic
open neighborhood
\[
U\subseteq \Specan(R\otimes_\C R_l),
\]
we have
\begin{equation}\label{eq:loccohloc}
R\Gamma_{[\wt Z_K]}\wt\cN|_{X\times U}
\simeq
\wt{\cK^{I_K}}(\cN,q)|_{X\times U}
\end{equation}
for $q\gg 0$. Globally,
\begin{equation}\label{eq:loccohest}
R\Gamma_{[\wt Z_K]}\wt\cN
\simeq
\varinjlim_{q\to\infty}
\wt{\cK^{I_K}}(\cN,q).
\end{equation}

For $\bm\in \N^l$, define the double limit
\[
\lim_{\longleftrightarrow}
\wt{\cK^{I_K}}(\alpha(\bm)^j\cN,q)
\coloneqq
\varprojlim_j
\varinjlim_q
\wt{\cK^{I_K}}(\alpha(\bm)^j\cN,q),
\]
which naturally generalizes the double-limit construction in \cite{Beiglue}.

\begin{prop}\label{prop:doublelimitrgamma}
With the notation and assumptions above, we have natural quasi-isomorphisms
\[
R\Gamma_{[\wt Z_K]}\wt{\cN(!D_\bm)}
\stackrel{\mathrm{q.i.}}{\simeq}
\lim_{\longleftrightarrow}
\wt{\cK^{I_K}}(\alpha(\bm)^j\cN,q)
\simeq
\varinjlim_q
\varprojlim_j
\wt{\cK^{I_K}}(\alpha(\bm)^j\cN,q),
\]
and
\[
R^i\Gamma_{[\wt Z_K]}\wt{\cN(!D_\bm)}
\simeq
\lim_{\longleftrightarrow}
\cH^i\wt{\cK^{I_K}}(\alpha(\bm)^j\cN,q).
\]
\end{prop}

\begin{proof}
By \eqref{eq:loccohest} and \eqref{eq:relminexge}, we obtain
\[
R\Gamma_{[\wt Z_K]}\wt{\cN(!D_\bm)}
\stackrel{\mathrm{q.i.}}{\simeq}
\lim_{\longleftrightarrow}
\wt{\cK^{I_K}}(\alpha(\bm)^j\cN,q),
\]
and the resulting double limit is independent of the order of the two limits.

For every sufficiently small analytic open neighborhood
\[
U\subseteq \Specan(R\otimes_\C R_l),
\]
we have
\begin{equation}\label{eq:!Rgamma}
R\Gamma_{[\wt Z_K]}\wt{\cN(!D_\bm)}|_{X\times U}
\simeq
\wt{\cK^{I_K}}(\alpha(\bm)^j\cN,q)|_{X\times U}
\end{equation}
for all $q\gg 0$ and $j\gg 0$. Thus the Mittag--Leffler conditions are
satisfied locally on $X\times U$, and we can apply
\cite[Prop.1.12.4]{KSbook}. Hence
\[
R^i\Gamma_{[\wt Z_K]}\wt{\cN(!D_\bm)}
\simeq
\varprojlim_j
\cH^i
\varinjlim_q
\wt{\cK^{I_K}}(\alpha(\bm)^j\cN,q).
\]
Since the direct limit functor is exact, we conclude that
\[
R^i\Gamma_{[\wt Z_K]}\wt{\cN(!D_\bm)}
\simeq
\lim_{\longleftrightarrow}
\cH^i\wt{\cK^{I_K}}(\alpha(\bm)^j\cN,q).
\]
\end{proof}

Since $\wt{\cK^{I_K}}(\alpha(\bm)^j\cN,q)$ is algebraic over
$\wt{\sA_X^{R\otimes_\C R_l}}$, a continuation
argument and \eqref{eq:!Rgamma} immediately imply the following.

\begin{coro}\label{cor:algsupploccoh}
For every $i$, each irreducible component of
\[
\supp_{\Specan(R\otimes_\C R_l)}
\left(
R^i\Gamma_{[\wt Z_K]}\wt{\cN(!D_\bm)}
\right)
\]
is an algebraic subvariety of $\Spec(R\otimes_\C R_l)$.
\end{coro}

For finitely generated $\sA^R_{Y,D}$-modules supported on $D^I$, we have the
following logarithmic version of Kashiwara's equivalence.

\begin{lemma}\label{lm:logKashequiv}
Let $\cM$ be a finitely generated $\sA^R_{Y,D}$-module supported on
$D^I=Y_{I^c}$
for some $I\subseteq \{1,\dots,l\}$. Then $\cM$ is a finitely generated
$\sA_{Y_{I^c},\bar D_{I^c}}^{R\otimes_\C R_I}$-module via the inclusion
\[
\sA_{Y_{I^c},\bar D_{I^c}}^{R\otimes_\C R_I}
\hookrightarrow
\sA^R_{Y,D}
\]
induced by \eqref{eq:uncanincl}.
\end{lemma}

\begin{proof}
Let $e_1,\dots,e_m$ be generators of $\cM$ over $\sA^R_{Y,D}$, and let
$\cI\subseteq \cO_Y$ be the ideal of $D^I$. Since $\cM$ is also an
$\cO_Y$-module, the algebraic Nullstellensatz implies that there exists
$q>0$ such that
$\cI^q\cdot e_i=0$
for $i=1,\dots,m$. By Lemma \ref{lm:twosidedideal}, we obtain
\[
\cI^q\cdot \cM
=
\sum_{i=1}^m(\cI^q\cdot \sA^R_{Y,D})\cdot e_i
=
\sum_{i=1}^m(\sA^R_{Y,D}\cdot \cI^q)\cdot e_i
=
0.
\]
Thus we have a finite $\cI$-adic filtration
\[
\cM
\supseteq
\cI\cdot\cM
\supseteq
\cI^2\cdot\cM
\supseteq
\cdots
\supseteq
\cI^q\cdot\cM
=
0.
\]
For each $j$, the quotient
\[
\cI^j\cdot\cM/\cI^{j+1}\cdot\cM
\]
is a finitely generated $i^*\sA^R_{Y,D}$-module, where
$i\colon D^I=Y_{I^c}\hookrightarrow Y$
is the closed embedding. By Lemma \ref{lm:logsp}, the inclusion
\[
\sA_{Y_{I^c},\bar D_{I^c}}^{R\otimes_\C R_I}
\hookrightarrow
\sA^R_{Y,D}
\]
gives a section of
$\sA^R_{Y,D}\to i^*\sA^R_{Y,D}.$
Therefore,
\[
\cI^j\cdot\cM/\cI^{j+1}\cdot\cM
\]
is a finitely generated
$\sA_{Y_{I^c},\bar D_{I^c}}^{R\otimes_\C R_I}$-module for each $j$. Since
the $\cI$-adic filtration is finite, induction gives the claim.
\end{proof}

\subsection{Sabbah's specialization complexes and the comparison}
\label{subsect:SabSp}

Let
\[
F=(f_1,\dots,f_r)\colon X\to \A^r_\C
\]
be a regular map as in Example \ref{ex:maingrebex}, and let
$\cF^\bullet\in D^b_{\C-c}(X^\an,\C)$
be a constructible complex. For $\bm\in \N^r$ and for a monoid ideal
$K\subseteq \N^r$, the \emph{Sabbah specialization complex} of
$\cF^\bullet$ along $K_\bm$ is
\[
\psi_{K_\bm}(\cF^\bullet)
\coloneqq
Rj_*\big(\cF^\bullet|_U\otimes_\C F^{-1}\psi^{\mathrm{univ}}\big)
\big|_{D_X^{K_\bm}}
\in
D^b_{\C-c}(X^\an,\C[\cG_r]),
\]
where
$j\colon U=X\setminus(f=0)\hookrightarrow X$
is the open embedding and $\psi^{\mathrm{univ}}$ is the universal local system
on
$(\C^*)^r\hookrightarrow \A^r_\C$;
cf. \cite[\S 3.1]{WuRHA}. By definition,
$\psi_{K_\bm}(\cF^\bullet)$ depends only on the support of $Z_{K_\bm}$.

Let $\cM$ be a regular holonomic $\shD_X$-module, and let $\cN_0$ be a
$\shD_{Z,E}$-lattice of $\iota_{F,+}\cM[1/f]$. Then $\cN_0$ is of GE-type
with respect to $F$. We write
\[
\pi=(\id,\operatorname{Exp})
\colon
X^\an\times \C^r\to X^\an\times(\C^*)^r
\]
for the covering map. With the help of Proposition \ref{prop:!dual}, the
following theorem is a natural generalization of \cite[Thm.1.2]{WuRHA}. It
can be proved by a similar argument; we omit the details.

\begin{theorem}\label{thm:sabbspcmp}
With the notation and assumptions above, for a monoid ideal
$K\subseteq \N^r$ and for $\bm\in \N^r$, we have natural quasi-isomorphisms
\[
\begin{aligned}
&
\pi_*^{\sG_r}
\DR_{X^\an\times\C^r/\C^r}
\big(\wt{i^!_{K_\bm}}\cN_0(*E)\big)
\\
&\simeq
\DR_{X^\an\times(\C^*)^r/(\C^*)^r}
\left(
\pi_*^{\sG_r}
\big(\wt{i^!_{K_\bm}}\cN_0(*E)\big)
\right)
\\
&\simeq
\wt\psi_{K_\bm}\big(\DR(\cM)\big).
\end{aligned}
\]
\end{theorem}

Sabbah specialization complexes generalize Deligne's classical nearby cycles
of constructible sheaves \cite{Sab90}. Motivated by
Theorem \ref{thm:sabbspcmp}, we make the following definition.

\begin{definition}\label{def:grnearby}
Let $\cN$ be a finitely generated $\sA^R_{Y,D}$-module of GE-type with respect
to
\[
G=(g_1,\dots,g_l)
\]
such that $\cN$ is relative regular holonomic over
$\sA_X^{R\otimes_\C R_l}$, with the induced
$\sA_X^{R\otimes_\C R_l}$-module structure given by
\eqref{eq:uncaningen}. For a monoid ideal $K\subseteq \N^l$ and for
$\bm,\bn\in \N^l$, the relative regular holonomic
$\wt{\sA_X^{R\otimes_\C R_l}}$-module
\[
\wt{i^!_{K_\bm}}\cN(*D_\bn)
\]
is called the \emph{generalized nearby cycles} of $\cN(*D_\bn)$ along
$K_\bm$.
\end{definition}

By Theorem \ref{thm:upper!DI} and \cite[Thm. 4.1]{FTM18},
\[
\DR_{X^\an\times\C^l/\C^l}
\big(\wt{i^!_{K_\bm}}\cN(*D_\bn)\big)
\]
is relative perverse over $\Specan(R\otimes_\C R_l)$ in the sense of
\cite[Notation 3.3]{FTM18}. Together with Theorem \ref{thm:sabbspcmp}, and
using the faithful flatness of the analytic sheafification functor, we obtain
in particular that
$\psi_{K_\bm}\big(\DR(\cM[1/f])\big)$
is perverse (with respect to the middle perversity) in the sense of \cite[\S 10.3]{KSbook}; see also \cite[\S 3]{LMW}
for a global algebraic approach.

We define
\[
\lim_{\longleftrightarrow}
\frac{\wt{\alpha(\bm)^{-q}\cN}}
{\sum_{i=1}^p\wt{\alpha(v_i)^j\cN}}
\coloneqq
\varprojlim_j
\varinjlim_q
\frac{\wt{\alpha(\bm)^{-q}\cN}}
{\sum_{i=1}^p\wt{\alpha(v_i)^j\cN}},
\]
as $q,j\to+\infty$, where $v_1,\dots,v_p$ generate the monoid ideal $K$.
Similarly to Proposition \ref{prop:doublelimitrgamma}, Theorem
\ref{thm:upper!DI} immediately gives the following.

\begin{prop}\label{prop:doublelimitnearby}
With the notation and assumptions of Definition \ref{def:grnearby}, we have
\[
\wt{i^!_K}\cN
\simeq
\varprojlim_j
\frac{\wt\cN}
{\sum_{i=1}^p\wt{\alpha(v_i)^j\cN}},
\]
and, for $\bm\in \N^l$,
\[
\wt{i^!_K}\cN(*D_\bm)
\simeq
\lim_{\longleftrightarrow}
\frac{\wt{\alpha(\bm)^{-q}\cN}}
{\sum_{i=1}^p\wt{\alpha(v_i)^j\cN}}
\simeq
\varinjlim_q
\varprojlim_j
\frac{\wt{\alpha(\bm)^{-q}\cN}}
{\sum_{i=1}^p\wt{\alpha(v_i)^j\cN}},
\]
where $v_1,\dots,v_p\in \N^l$ generate $K$.
\end{prop}

\begin{remark}\label{rmk:analyticlattices}
By the same methods, one can also construct relative minimal extensions,
$!$-pullbacks, and algebraic local cohomology complexes for coherent
$\wt{\sA^R_{Y,D}}$-modules $\sN$ of GE-type satisfying that
$\sN(*\wt D)$ is a relative regular holonomic $\wt{\sA^R_Y}$-module. These
objects have the same properties as in the algebraic case.
\end{remark}

Taking $K\subseteq \N^r$ to be the monoid ideal generated by
$\sum_{i=1}^r e_i$, we obtain the special generalized nearby cycles
\[
\Psi_F(\cM)
\coloneqq
\frac{\wt{\cN_0(*E)}}{\wt{\cN_0(!E)}}
\]
studied in \cite{WuRHA}. By \cite[Cor.1.3 and Thm.1.4]{WuRHA},
\[
\supp_X
\left(
\pi^{\sG_r}_*
\frac{\wt{\cN_0(*E)}}{\wt{\cN_0(!E)}}
\right)
\subseteq
(\C^*)^r
\]
is a finite union of translated subtori. If $\cM$ underlies a $\Q$-mixed Hodge
module, then, with the help of Theorem \ref{thm:latticemtqsupp}, this support
is a finite union of torsion-translated subtori.

Since
\[
\frac{\wt{\cN_0(!E)(*E_\bm)}}{\wt{\cN_0(!E)}}
\]
is a submodule of
$\frac{\wt{\cN_0(*E)}}{\wt{\cN_0(!E)}}$,
the same containment holds for the support
\[
\supp_X
\left(
\pi^{\sG_r}_*
\frac{\wt{\cN_0(!E)(*E_\bm)}}{\wt{\cN_0(!E)}}
\right)
\]
for every $\bm\in \N^r$.

Assume now that $\cM$ underlies a $\Q$-mixed Hodge module. For a general
torsion point
\[
\lambda=(\lambda_1,\dots,\lambda_r)
\in
\supp_X
\left(
\pi^{\sG_r}_*
\frac{\wt{\cN_0(!E)(*E_\bm)}}{\wt{\cN_0(!E)}}
\right),
\]
more precisely, for every torsion point
\[
\lambda\in
\supp
\left(
\pi^{\sG_r}_*
\frac{\wt{\cN_0(!E)(*E_\bm)}}{\wt{\cN_0(!E)}}
\right)
\]
that is contained in exactly one irreducible component of
\[
\supp
\left(
\pi^{\sG_r}_*
\frac{\wt{\cN_0(*E)}}{\wt{\cN_0(!E)}}
\right),
\]
there exists $i_0\in\{1,\dots,r\}$ such that the one-dimensional subtorus
\[
W=
\{\beta\in(\C^*)^r\mid \beta_i=\lambda_i
\textup{ for } i\neq i_0\}
\subseteq
(\C^*)^r
\]
intersects
\[
\supp
\left(
\pi^{\sG_r}_*
\frac{\wt{\cN_0(!E)(*E_\bm)}}{\wt{\cN_0(!E)}}
\right)
\]
properly around $\lambda$.

\begin{prop}\label{prop:essfibMHM}
With the notation above, assume that $\cM$ underlies a $\Q$-mixed Hodge module.
Then, for every $\bm\in \N^r$ and for a general torsion point
\[
\lambda\in
\supp
\left(
\pi^{\sG_r}_*
\frac{\wt{\cN_0(!E)(*E_\bm)}}{\wt{\cN_0(!E)}}
\right),
\]
the essential fiber of
\[
\pi^{\sG_r}_*
\frac{\wt{\cN_0(!E)(*E_\bm)}}{\wt{\cN_0(!E)}}
\]
at $\lambda$ along $W$ is a coherent submodule of a regular holonomic
$\shD_{X^\an}$-module underlying a $\Q$-mixed Hodge module.
\end{prop}

\begin{proof}
Since $\lambda$ is a general torsion point, we can lift $\lambda$ to a very
general $\Q$-point
\[
\alpha\in
\supp
\left(
\frac{\wt{\cN_0(*E)}}{\wt{\cN_0(!E)}}
\right).
\]
By \cite[Cor.2.18]{WuRHA}, and since nearby cycles of $\Q$-mixed Hodge
modules along one function remain $\Q$-mixed Hodge modules, the essential
fiber of
\[
\pi^{\sG_r}_*
\frac{\wt{\cN_0(*E)}}{\wt{\cN_0(!E)}}
\]
at $\lambda$ along $W$, which is the same as the essential fiber of
\[
\frac{\wt{\cN_0(*E)}}{\wt{\cN_0(!E)}}
\]
at $\alpha$ along a lift of $W$, underlies a direct summand of a
$\Q$-mixed Hodge module. Since $\lambda$ is general, the essential fiber of
\[
\pi^{\sG_r}_*
\frac{\wt{\cN_0(!E)(*E_\bm)}}{\wt{\cN_0(!E)}}
\]
at $\lambda$ along $W$ is a submodule of that of
\[
\pi^{\sG_r}_*
\frac{\wt{\cN_0(*E)}}{\wt{\cN_0(!E)}}.
\]
This proves the proposition.
\end{proof}

\begin{remark}\label{rmk:locessfiMHM}
Proposition \ref{prop:essfibMHM} is compatible with localizations as in
Remark \ref{rmk:localizingupper!}(2). More precisely, since direct images of
$\Q$-mixed Hodge modules remain $\Q$-mixed Hodge modules, for
$J\subseteq\{1,\dots,r\}$ and for a general torsion point
\[
\lambda\in
\supp_{Y_J\setminus E_J}
\left(
\pi^{\sG_{r-|J|}}_*
\frac{
\wt{\cN_0(*E_J)(!E)(*E_\bm)}^{Y_J\setminus E_J}
}{
\wt{\cN_0(*E_J)(!E)}^{Y_J\setminus E_J}
}
\right)
\subseteq
(\C^*)^{r-|J|},
\]
the essential fiber of
\[
\pi^{\sG_{r-|J|}}_*
\frac{
\wt{\cN_0(*E_J)(!E)(*E_\bm)}^{Y_J\setminus E_J}
}{
\wt{\cN_0(*E_J)(!E)}^{Y_J\setminus E_J}
}
\]
at $\lambda$ along the corresponding one-dimensional subtorus
$W\subseteq(\C^*)^{r-|J|}$
is also a coherent submodule of a regular holonomic $\shD_{Y_J^\an}$-module
underlying a $\Q$-mixed Hodge module. Here, by abuse of notation,
\[
\pi\colon Y_J\times \C^{r-|J|}
\to
Y_J\times(\C^*)^{r-|J|}
\]
still denotes the covering map.
\end{remark}

We now know the geometry of the relative support of
${\wt{\cN_0(*E)}}/{\wt{\cN_0(!E)}}$.
We also want to understand the geometry of
$\supp\big(\wt{i^!_{K_\bm}}\cN_0(*E)\big)$
in general.

\begin{theorem}\label{thm:linearsupp}
In the situation of Theorem \ref{thm:sabbspcmp}, the following statements hold:
\begin{enumerate}[label=$(\arabic*)$]
    \item The $\sG_r$-invariant subset
    \[
    \supp\big(\wt{i^!_{K_\bm}}\cN_0(*E)\big)\subseteq \C^r
    \]
    is an infinite union of translated linear subvarieties such that
    \[
    \supp\big(\psi_{K_\bm}(\DR(\cM))\big)
    =
    \operatorname{Exp}
    \big(
    \supp(\wt{i^!_{K_\bm}}\cN_0(*E))
    \big)
    \subseteq
    (\C^*)^r
    \]
    is a finite union of translated subtori.

    \item For every $\bn\in \N^r$,
    \[
    \supp\big(\wt{i^!_{K_\bm}}\cN_0(*E_\bn)\big)\subseteq \C^r
    \]
    is an infinite union of translated linear subvarieties.

    \item If $\sN$ is a coherent $\wt{\sA_X^{R_r}}$-submodule of
    $\wt{i^!_{K_\bm}}\cN_0(*E_\bn)$ for some $\bn\in \N^r$, then
    \[
    \supp(\sN)\subseteq \C^r
    \]
    is a union of translated linear subvarieties.

    \item If, moreover, $\cM$ underlies a $\Q$-mixed Hodge module, then
    \[
    \supp\big(\wt{i^!_{K_\bm}}\cN_0(*E)\big)
    \]
    is an infinite union of $\Q$-translated linear subvarieties such that
    \[
    \supp\big(\psi_{K_\bm}(\DR(\cM))\big)
    =
    \operatorname{Exp}
    \big(
    \supp(\wt{i^!_{K_\bm}}\cN_0(*E))
    \big)
    \subseteq
    (\C^*)^r
    \]
    is a finite union of torsion-translated subtori.

    \item If $\cM$ underlies a $\Q$-mixed Hodge module and if $\sN$ is a
    coherent $\wt{\sA_X^{R_r}}$-submodule of
    $\wt{i^!_{K_\bm}}\cN_0(*E_\bn)$ for some $\bn\in \N^r$, then
    \[
    \supp(\sN)\subseteq \C^r
    \]
    is a union of $\Q$-translated linear subvarieties.
\end{enumerate}
\end{theorem}

To prove the theorem, we need the following Ax--Lindemann type criterion of
Budur and Wang.

\begin{lemma}[Budur--Wang]\label{lm:BudurWang}
Let $V\subseteq \C^r$ and $W\subseteq(\C^*)^r$ be irreducible algebraic
subvarieties such that
$\operatorname{Exp}(V)=W.$
Then:
\begin{enumerate}[label=$(\arabic*)$]
    \item $V$ is a translated linear subvariety of $\C^r$, and $W$ is a
    translated subtorus;
    \item if $V$ has a $\Q$-point, then $V$ is a $\Q$-translated linear
    subvariety and $W$ is a torsion-translated subtorus.
\end{enumerate}
\end{lemma}

\begin{proof}
Part $(1)$ is \cite[Prop.1.4 and Lem.2.1]{BWjlk}. Part $(2)$ follows
directly from the definition.
\end{proof}

We also need the following equivariant version of
Theorem \ref{thm:mainsuppfiberation}.

\begin{theorem}\label{thm:equivmainsuppfiberation}
Let $\cN$ be an $\sA^R_{Y,D}$-lattice, respectively a
$\wt{\sA^R_{Y,D}}$-lattice, of GE-type with respect to
$G=(g_1,\dots,g_l)$,
such that $\cN(*D)$ is a relative regular holonomic $\sA^R_Y$-module,
respectively a relative regular holonomic $\wt{\sA^R_Y}$-module. Let
$V\subseteq \supp_Y(\cN(*D))$
be the closed subset such that $\cN(*D)$ is Cohen--Macaulay away from $V$;
see Lemma \ref{lm:genericCMness}. Then:
\begin{enumerate}[label=$(\arabic*)$]
    \item The projection
    \[
    p\colon \Specan R\times(\C^*)^l\to \Specan R
    \]
    induces a surjective morphism
    \[
    \bar p\colon
    \supp_X^h
    \left(
    \pi_{R*}^{\sG_l}
    \frac{\wt{\cN(*D)}}{\wt{\cN(!D)}}
    \right)
    \to
    \supp_Y(\cN(*D)),
    \]
    such that for every
    $a\in \supp_Y(\cN(*D))\subseteq \Specan R$,
    the fiber
    $\bar p^{-1}(a)\subseteq(\C^*)^l$
    is a finite union of translated subtori of codimension one. Here
    \[
    \supp_X^h
    \left(
    \pi_{R*}^{\sG_l}
    \frac{\wt{\cN(*D)}}{\wt{\cN(!D)}}
    \right)
    \]
    denotes the top-dimensional part of the support, and
    \[
    \pi_R\colon
    X^\an\times \Specan R\times \C^l
    \to
    X^\an\times \Specan R\times(\C^*)^l
    \]
    is the covering map.

    \item For every closed point
    $a\in \supp(\cN(*D))\setminus V,$
    we have
    \[
    \bar p^{-1}(a)
    =
    \supp
    \left(
    \pi_{R*}^{\sG_l}
    \frac{\wt{\cN(*D)_a}}{\wt{\cN(!D)_a}}
    \right)
    \]
    and
    \[
    \pi_{R*}^{\sG_l}
    \frac{\wt{\cN(*D)_a}}{\wt{\cN(!D)_a}}
    \simeq
    \pi_{R*}^{\sG_l}
    \frac{\wt{\cN(*D)}}{\wt{\cN(!D)}}
    \otimes_R
    \frac{R_\Omega}{I_W}.
    \]
    Here $\cN(*D)_a$ is the essential fiber of $\cN(*D)$ at $a$ along a
    smooth subvariety
   $W\subseteq \supp(\cN(*D))$
    containing $a$ and intersecting $\supp(\cN(*D))$ properly around $a$, with
    \[
    \dim W+\dim\supp(\cN(*D))=\dim\Spec R.
    \]
    Moreover, $I_W\subset R_\Omega$ is the ideal of $W$, where
    $\Spec R_\Omega\subseteq\Spec R$ is a Zariski open neighborhood of $a$,
    and $\wt{\cN(!D)_a}$ is the minimal extension of $\cN(*D)_a$ along $D$.

    \item If $\cN(*D)_a$ is $\Q$-specializable for
    $a\in \supp(\cN(*D))\setminus V,$
    then $\bar p^{-1}(a)$ is a finite union of torsion-translated subtori of
    codimension one.
\end{enumerate}
\end{theorem}

\begin{proof}
By Corollary \ref{cor:anlocmaxminex} and \eqref{eq:relminexge}, and since the
analytic sheafification functor is exact, we can use the same inductive
argument as before, cutting by general smooth hypersurfaces. We reduce to the
case where
$\supp(\cN_m(*D))=\{a\}.$
After forgetting the $R$-module structure, $\cN_m(*D)$ is a regular holonomic
$\shD_Y$-module supported on the graph of $G$. Part $(1)$ then follows from
\cite[Cor.1.3 and Thm.1.4]{WuRHA}.

Since $\cN(*D)$ is Cohen--Macaulay around $a$, we have
$\cN_m(*D)=\cN(*D)_a$
and
\[
\pi_{R*}^{\sG_l}
\frac{\wt{\cN(*D)_a}}{\wt{\cN(!D)_a}}
\simeq
\pi_{R*}^{\sG_l}
\frac{\wt{\cN(*D)}}{\wt{\cN(!D)}}
\otimes_R
\frac{R_\Omega}{I_W}.
\]
Here one can take $W$ to be the intersection of the hypersurfaces appearing
in the induction. This proves part $(2)$.

By construction, every component of
\[
\supp
\left(
\pi_{R*}^{\sG_l}
\frac{\wt{\cN(*D)_a}}{\wt{\cN(!D)_a}}
\right)
\]
can be lifted to a component of
$\supp\big(\cN_1/\cN_1(-D)\big)$
for some $\shD_{Y,D}$-lattice $\cN_1$ of $\cN(*D)_a$. Part $(3)$ follows
from part $(2)$, Lemma \ref{lm:sqstQsp}, and
Theorem \ref{thm:latticemtqsupp}.
\end{proof}

\begin{proof}[Proof of Theorem \ref{thm:linearsupp}]
By Theorem \ref{thm:sabbspcmp}, the module
$\wt{i^!_{K_\bm}}\cN_0(*E)$
is $\sG_r$-equivariant. Hence its support is $\sG_r$-invariant. By
Remark \ref{rmk:localizingupper!}(2) and $(3)$, after localizing the log
structure of $Z$ at $\alpha(\bm)$, it is enough to assume that
$K_\bm=K.$

By construction; see also \cite[(2.11)]{WuRHA}, we have
\[
\wt\D\big(\wt{\cN_0(*E)}\big)
\simeq
\wt{\cN_0'(!E)},
\]
where $\cN_0'$ is a $\shD_{Z,E}$-lattice of
\[
\operatorname{Ext}^{n+r}_{\shD_Z}
\big(\iota_{F,+}\cM[1/f],\shD_Z\big).
\]
By Remark \ref{rmk:analyticpurefil} and Lemma \ref{lm:dualalPF}, and since
$\wt\D$ is $\sG_r$-equivariant; cf. \cite[Remark 2.10(2)]{WuRHA}, the pure
filtration of $\wt{i^!_K}\cN_0(*E)$ is $\sG_r$-equivariant. In particular,
\[
\pi_*^{\sG_r}
\big(T_p(\wt{i^!_K}\cN_0(*E))\big)
=
T_p
\big(
\pi_*^{\sG_r}(\wt{i^!_K}\cN_0(*E))
\big)
\]
for every $p$.
Moreover, by Lemma \ref{lm:twistdual}, Theorem \ref{thm:sabbspcmp}, and
\cite[Thm.3.11]{FS13}, both
\[
\pi_*^{\sG_r}\big(\wt{i^!_K}\cN_0(*E)\big)
\quad \textup{and} \quad
\wt\D\left(
\pi_*^{\sG_r}\big(\wt{i^!_K}\cN_0(*E)\big)
\right)
\]
are dR-algebraic. Then \cite[Thm.4.1]{FTM18}, together with faithful
flatness of the analytic sheafification functor, implies that
\[
T_p
\big(
\pi_*^{\sG_r}(\wt{i^!_K}\cN_0(*E))
\big)
\]
is dR-algebraic for every $p$. Hence
\[
\frac{
T_p
\big(
\pi_*^{\sG_r}(\wt{i^!_K}\cN_0(*E))
\big)
}{
T_{p+1}
\big(
\pi_*^{\sG_r}(\wt{i^!_K}\cN_0(*E))
\big)
}
=
\pi_*^{\sG_r}
\left(
\frac{
T_p\big(\wt{i^!_K}\cN_0(*E)\big)
}{
T_{p+1}\big(\wt{i^!_K}\cN_0(*E)\big)
}
\right)
\]
is also dR-algebraic.

By Lemma \ref{lm:dualalPF}, Corollary \ref{cor:algsupploccoh}, and purity,
every irreducible component $V$ of
\[
\supp
\left(
\frac{
T_p\big(\wt{i^!_K}\cN_0(*E)\big)
}{
T_{p+1}\big(\wt{i^!_K}\cN_0(*E)\big)
}
\right)
\subseteq
\C^r
\]
is an algebraic subvariety. Since
\[
\pi_*^{\sG_r}
\left(
\frac{
T_p\big(\wt{i^!_K}\cN_0(*E)\big)
}{
T_{p+1}\big(\wt{i^!_K}\cN_0(*E)\big)
}
\right)
\]
is dR-algebraic, Lemma \ref{lm:dralg} implies that
$\operatorname{Exp}(V)\subseteq(\C^*)^r$
is algebraic. By Lemma \ref{lm:BudurWang}(1), $V$ is a translated linear
subvariety of codimension $p-n$. Therefore, every irreducible component of
$\supp\big(\wt{i^!_K}\cN_0(*E)\big)$
is a translated linear subvariety.

By $\sG_r$-invariance and Theorem \ref{thm:sabbspcmp},
$\supp\big(\wt{i^!_K}\cN_0(*E)\big)$
is an infinite union of translated linear subvarieties satisfying
\[
\supp\big(\psi_{K_\bm}(\DR(\cM))\big)
=
\operatorname{Exp}
\big(
\supp(\wt{i^!_{K_\bm}}\cN_0(*E))
\big).
\]
This proves part $(1)$.

By uniqueness of the pure filtration,
$\sN\cap T_\bullet\big(\wt{i^!_K}\cN_0(*E)\big)$
is the pure filtration of $\sN$. Thus part $(3)$ follows. By Corollary
\ref{cor:latticesinc}, every irreducible component of
$\supp\big(\wt{i^!_K}\cN_0(*E_\bn)\big)$
is also a translated linear subvariety. It remains to show that this support
is an infinite union of translated linear subvarieties. Again by Corollary
\ref{cor:latticesinc}, it is enough to prove that
$\supp\big(\wt{i^!_K}\cN_0\big)$
is an infinite union.
Suppose, for contradiction, that
$\supp\big(\wt{i^!_K}\cN_0\big)$
is a finite union. By Corollary \ref{cor:latticesinc}, we have a short exact
sequence
\[
0\to
\wt{i^!_K}\cN_0
\to
\wt{i^!_K}\cN_0(*E)
\to
\wt{i^!_K}
\left(
\frac{\cN_0(*E)}{\cN_0}
\right)
\to 0.
\]
Let $V$ be an irreducible component of
$\supp\big(\wt{i^!_K}\cN_0(*E)\big).$
Then
\[
\prod_{i=1}^r t_i^{v_i}\cdot V
\subseteq
\supp
\left(
\wt{i^!_K}
\left(
\frac{\cN_0(*E)}{\cN_0}
\right)
\right)
\]
for all but finitely many
$(v_1,\dots,v_r)\in \Z^r\simeq \sG_r.$
By the existence of Sabbah's generalized $b$-functions; cf.
\cite[Thm.2.1]{WuRHA}, and by substitution, there exist finitely many
nonzero vectors
$L=(l_1,\dots,l_r)\in \N^r$
and finitely many $\gamma\in\C$ such that
\[
\supp
\left(
\frac{\cN_0(*E)}{\cN_0}
\right)
\subseteq
\bigcup_{L,\gamma}
\bigcup_{p\ge 0}
\left(
\sum_{i=1}^r l_is_i+\gamma-p=0
\right).
\]
By Theorem \ref{thm:upper!DI}, we have
\[
\supp
\left(
\wt{i^!_K}
\left(
\frac{\cN_0(*E)}{\cN_0}
\right)
\right)
\subseteq
\supp
\left(
\frac{\cN_0(*E)}{\cN_0}
\right).
\]
But for all $q\gg 0$,
\[
\prod_{i=1}^r t_i^q\cdot V
\not\subseteq
\bigcup_{L,\gamma}
\bigcup_{p\ge 0}
\left(
\sum_{i=1}^r l_is_i+\gamma-p=0
\right),
\]
which is a contradiction. Therefore, part $(2)$ follows.

We now prove part $(4)$. By part $(3)$ and Lemma \ref{lm:dualalPF}, it is
enough to prove that every codimension-$q$ component of
\[
\supp
\left(
R^q\Gamma_{[\wt Z_K]}\wt{\cN_0'(!E)}
\right)
\]
is a $\Q$-translated linear subvariety for every $q$. Notice that, by
Lemma \ref{lm:ausdual} and relative holonomicity, the dimension of
\[
\supp
\left(
R^q\Gamma_{[\wt Z_K]}\wt{\cN_0'(!E)}
\right)
\]
is at most $q$. By $\sG_r$-equivariance, it is equivalent to prove that every
codimension-$q$ component of
\[
\supp
\left(
\pi_*^{\sG_r}
\big(
R^q\Gamma_{[\wt Z_K]}\wt{\cN_0'(!E)}
\big)
\right)
\]
is a torsion-translated subtorus for every $q$.

We first prove this for $K=K_I$ with $I=\{1,\dots,r\}$. Write
\[
K^i=\langle e_1,\dots,e_i\rangle,
\qquad
\bm_i=\sum_{j=i+1}^r e_j,
\]
\[
Y_{>i}=(t_1=\cdots=t_i=0)\subseteq Z,
\qquad
D_{>i}=\left(\prod_{j>i}t_j=0\right)\subseteq Y_{>i}.
\]
By Remark \ref{rmk:localizingupper!}, taking the base space $Y_{>i}$ instead
of $X$ and the log structure given by
\[
\N^i\hookrightarrow
S[t_{i+1}^{\pm1},\dots,t_r^{\pm1}][\N^i],
\]
we have a relative regular holonomic
$\wt{\sA_{Y_{>i}}^{R_i}}$-module
\[
\wt{\cN_0'(*E_{\bm_i})(!E)}^{Y_{>i}\setminus D_{>i}}.
\]
Since this module is supported on the graph of $(f_j)_{j>i}$, Remark
\ref{rmk:localizingupper!}(2) and relative Kashiwara equivalence; see
\eqref{eq:relgetype}, give
\[
R\Gamma_{[\wt Z_{K^i_{\bm_i}}]}(\wt{\cN_0'(!E)})
\simeq
\widetilde{
R\Gamma_{[\wt Z_{K^i_{\bm_i}}^{Y_{>i}}]}
}
\,
\wt{\cN_0'(*E_{\bm_i})(!E)}^{Y_{>i}\setminus D_{>i}},
\]
where the tilde over
$R\Gamma_{[\wt Z_{K^i_{\bm_i}}^{Y_{>i}}]}$
means that the resulting complex of
\[
\wt{\sA_X^{R_i}}[s_j]_{j>i}
=
\wt{\sA_X^{R_i}}[s_{i+1},\dots,s_r]
\]
modules is regarded as a complex of $\wt{\sA_X^{R_r}}$-modules.

Similarly, for every subset $J\subseteq\{i+1,\dots,r\}$ and every $q$, we
have a $\sG_r$-equivariant, hence $\sG_{|J|}$-equivariant, relative regular
holonomic $\wt{\sA_{Y_J}^{R_{\bar J}}}$-module supported on the graph of
$(f_j)_{j\in J}$:
\[
\sM^q_{i,J}
\coloneqq
R^q\Gamma_{[\wt Z_{K^i_{\bm_J}}^{Y_J}]}
\wt{\cN_0'(*E_{\bm_J})(!E)}^{Y_J\setminus D_J}.
\]
Here $\bar J$ is the complement of $J$ in $\{1,\dots,r\}$,
\[
\bm_J=\sum_{j\in J}e_j,
\]
and $D_J\subseteq Y_J$ is the divisor defined by
$\prod_{j\in J}t_j=0.$
Let
\[
\pi_J\colon
Y_J^\an\times \C^{|\bar J|}
\to
Y_J^\an\times(\C^*)^{|\bar J|}
\]
be the covering map.

We argue by induction on $i$. The inductive hypotheses are:
\begin{enumerate}[label=(H\arabic*)]
    \item every codimension-$q$ component of
    \[
    \supp_{Y_J}
    \big(
    \pi_{J,*}^{\sG_{\bar J}}\sM^q_{i,J}
    \big)
    \subseteq
    (\C^*)^{|\bar J|}
    \]
    is a torsion-translated subtorus for every
    $J\subseteq\{i+1,\dots,r\}$;

    \item for every $J$, the essential fibers of
    $\pi_{J,*}^{\sG_{\bar J}}\sM^q_{i,J}$
    at general torsion points in its support are iterated extensions of
    subquotients of $\Q$-mixed Hodge modules.\footnote{Here an iterated
    extension means a module obtained by finitely many successive extensions
    starting from subquotients of $\shD$-modules underlying $\Q$-mixed Hodge
    modules.}
\end{enumerate}

By construction, for $I\subseteq\{i+2,\dots,r\}$, we have a
$\sG_{\bar I}$-equivariant distinguished triangle
\begin{equation}\label{eq:distria1}
\begin{split}
&
R\Gamma_{[\wt Z^{Y_I}_{K^{i+1}_{\bm_I}}]}
\wt{\cN_0'(!E)}^{Y_I\setminus D_I}
\to
R\Gamma_{[\wt Z^{Y_I}_{K^i_{\bm_I}}]}
\wt{\cN_0'(!E)}^{Y_I\setminus D_I}
\\
&\to
R\Gamma_{[\wt Z^{Y_I}_{K^i_{\bm_{I'}}}]}
\wt{\cN_0'(!E)}^{Y_I\setminus D_I}
\xrightarrow{+1},
\end{split}
\end{equation}
where
$I'=I\cup\{i+1\}.$
Since $\sM^q_{i,I'}$ is also a
$\wt{\sA_{Y_I}^{R_{r-|I'|}}}[s_{i+1}]$-module by relative Kashiwara equivalence, and since we have the decomposition
\[
\begin{tikzcd}
Y_I^\an\times \C^{r-|I|}
\arrow[rr,"\pi_I"]
\arrow[rd,"\varphi"]
&&
Y_I^\an\times(\C^*)^{r-|I|}
\\
&
Y_I^\an\times(\C^*)^{r-|I'|}\times\C
\arrow[ru,"\psi"]
&
\end{tikzcd}
\]
we obtain
\begin{equation}\label{eq:splitequiv}
\psi_*^{\sG_1}
\varphi_*^{\sG_{r-|I'|}}
\wt{\sM^q_{i,I'}}
\simeq
\psi_*^{\sG_1}
\widetilde{
\pi_{I',*}^{\sG_{r-|I'|}}\sM^q_{i,I'}
}
\simeq
\pi_{I,*}^{\sG_{r-|I|}}
R^q\Gamma_{[\wt Z^{Y_I}_{K^i_{\bm_I}}]}
\wt{\cN_0'(!E)}^{Y_I\setminus D_I}.
\end{equation}

Set $Y=Y_{I'}$, let $D\subseteq Y$ be the smooth divisor defined by
$t_{i+1}=0$, and set
\[
\bar E=E_{\bm_I}\cap Y_{I'},
\qquad
\bar E_{I'}=\bar E+D.
\]
Choose a $\wt{\sA_{Y,D}^{T_{r-|I'|}}(*\bar E)}$-lattice $\sN$ of
$\pi_{I',*}^{\sG_{r-|I'|}}\sM^q_{i,I'}.$
Notice that 
$\pi_{I',*}^{\sG_{r-|I'|}}\sM^q_{i,I'}$
is a $\wt{\sA_{Y_{I'}}^{T_{r-|I'|}}(*\bar E_{I'})}$-module, relative regular
holonomic over $\wt{\sA_{Y_{I'}}^{T_{r-|I'|}}}$, and supported on the graph of
$\wt f_{i+1}$, where $\wt f_{i+1}$ is the pullback of $f_{i+1}$ under the
natural projection
\[
Y_I\simeq X\times\C^{|I|}\to X.
\]
Then the minimal extension
\[
\wt{\sN(!D)}
=
\wt{\sN(!D)}^{Y_I\setminus\bar E}
\]
is a $\wt{\sA_{Y_I}^{T_{r-|I'|}[s_{i+1}]}(*\bar E)}$-module.

By Proposition \ref{prop:doublelimitrgamma}, for every sufficiently small
simply connected analytic open neighborhood
$U\subseteq(\C^*)^{r-|I'|}\times\C,$
there exists a finitely generated $\shD_{Z,E}$-module $\cP$ such that
$\cQ\coloneqq\cP(*E_{\bm_I})$
is supported on $Y_{I'}$ and
\[
\wt{\cQ}|_{Y_I\times U}
=
\wt{\sN(*D)}|_{Y_I\times U}.
\]
By Lemma \ref{lm:logKashequiv}, $\cQ$ is finitely generated over
$\sA_{Y,D}^{R_{r-|I'|}}(*\bar E).$
Therefore,
\[
\wt{\sN(!D)}|_{Y_I\times U}
=
\wt{t_{i+1}^k\cQ}|_{Y_I\times U}
\quad \textup{and} \quad
\wt{\sN(*D)}|_{Y_I\times U}
=
\wt{t_{i+1}^{-k}\cQ}|_{Y_I\times U}
\]
for $k\gg 0$.

Since $\sN$ is dR-algebraic, the same argument as in
Theorem \ref{thm:sabbspcmp} shows that
\[
\psi_*^{\sG_1}
\frac{\wt{\sN(*D)}}{\wt{\sN(!D)}}
\]
is also dR-algebraic. Therefore, by Lemma \ref{lm:BudurWang}, every
irreducible codimension-$(q+1)$ component of
\[
\supp
\left(
\psi_*^{\sG_1}
\frac{\wt{\sN(*D)}}{\wt{\sN(!D)}}
\right)
\subseteq
(\C^*)^{r-|I|}
\]
is a translated subtorus.

By the inductive hypothesis (H1), Lemma \ref{lm:sqstQsp}, and
Theorem \ref{thm:equivmainsuppfiberation}(3), the fibers of
\[
\supp
\left(
\psi_*^{\sG_1}
\frac{\wt{\sN(*D)}}{\wt{\sN(!D)}}
\right)
\subseteq
(\C^*)^{r-|I|-1}\times\C^*
\]
over general torsion points in
\[
\supp
\left(
\pi_{I',*}^{\sG_{r-|I'|}}\sM^q_{i,I'}
\right)
\subseteq
(\C^*)^{r-|I|-1}
\]
are finite unions of roots of unity in $\C^*$. In particular, every
codimension-$(q+1)$ component of
\[
\supp
\left(
\psi_*^{\sG_1}
\frac{\wt{\sN(*D)}}{\wt{\sN(!D)}}
\right)
\]
contains a torsion point. By Lemma \ref{lm:BudurWang}(2), every such component
is a torsion-translated subtorus.

Using \eqref{eq:distria1} and \eqref{eq:splitequiv}, we obtain a short exact
sequence
\[
0\to
\cQ_1
\to
\pi_{I,*}^{\sG_{r-|I|}}\sM^{q+1}_{i+1,I}
\to
\cK
\to 0
\]
such that $\cQ_1$ is a quotient module of
\[
\psi_*^{\sG_1}
\frac{\wt{\sN(*D)}}{\wt{\sN(!D)}}
\]
and $\cK$ is a submodule of
$\pi_{I,*}^{\sG_{r-|I|}}\sM^{q+1}_{i,I}.$
By Theorem \ref{thm:equivmainsuppfiberation}(2), Proposition
\ref{prop:essfibMHM}, and the inductive hypothesis (H2), the essential fibers
of
\[
\psi_*^{\sG_1}
\frac{\wt{\sN(*D)}}{\wt{\sN(!D)}}
\]
over general torsion points in its support are iterated extensions of
subquotients of $\Q$-mixed Hodge modules; hence the same holds for $\cQ_1$.

By the inductive hypothesis, the irreducible codimension-$(q+1)$ components of
\[
\supp(\cK)\subseteq(\C^*)^{r-|I|}
\]
are torsion-translated subtori, and the essential fibers of $\cK$ over general
torsion points in its support are also iterated extensions of subquotients of
$\Q$-mixed Hodge modules. Thus every irreducible codimension-$(q+1)$ component
of
\[
\supp
\left(
\pi_{I,*}^{\sG_{r-|I|}}\sM^{q+1}_{i+1,I}
\right)
\]
is a torsion-translated subtorus, and its essential fibers over general
torsion points are iterated extensions of subquotients of $\Q$-mixed Hodge
modules.

When $i=1$, the base case follows from Proposition \ref{prop:essfibMHM}.
Therefore, the case $K=K_I$ is proved. More precisely, we have proved the
cases $K=K_J$ for all subsets $J\subseteq\{1,\dots,r\}$, as well as these
cases after localizing the log structure at $\alpha(\bm)$.

For a general monoid ideal $K$, we repeatedly use the distinguished triangle
\[
R\Gamma_{[\wt Z_{K'}]}\wt{\cN_0'(!E)}
\to
R\Gamma_{[\wt Z_K]}\wt{\cN_0'(!E)}
\to
R\Gamma_{[\wt Z_{K_\bm}]}\wt{\cN_0'(!E)}
\xrightarrow{+1}
\]
for a suitable $\bm\in\N^r$ such that
\[
Z_{K'}\subsetneq Z_K
\quad \textup{and} \quad
K_\bm=(K_J)_\bn
\]
for some $J\subseteq\{1,\dots,r\}$ and some $\bn\in\N^r$. This reduces the
general case to the cases already proved, by induction.
\end{proof}

\section{Logarithmic interpretation of Bernstein--Sato ideals}
\label{subsect:logBSideal}

We keep the notation from \S\ref{subsect:logDmodule}. Let $\cN$ be a
$\shD_{Z,E}$-lattice. For a monoid ideal $K\subseteq \N^r$, we consider the
module
\[
\cN|_{Z_K}
=
\frac{\cN}{\alpha(K)\cdot \cN}.
\]
By Lemma \ref{lm:twosidedideal}, $\cN|_{Z_K}$ is again a
$\shD_{Z,E}$-module, and in particular an $\sA_X^{R_r}$-module and a
$\C[s_1,\dots,s_r]$-module via \eqref{eq:uncanincl}.

\begin{definition}
The \emph{Bernstein--Sato ideal} of $\cN$ along $K$ is the annihilator
\[
B^K(\cN)
\coloneqq
B_{\cN|_{Z_K}}
=
\Ann_{\C[s_1,\dots,s_r]}(\cN|_{Z_K})
\subseteq
\C[s_1,\dots,s_r].
\]
\end{definition}

A priori, it is not clear that
$B^K(\cN)\subseteq \C[s_1,\dots,s_r]$
is nonzero, since in general $\cN$ is not finitely generated over
$\C[s_1,\dots,s_r]$. The following lemma is essentially due to Sabbah
\cite{Sab}.

\begin{lemma}\label{lm:subbahgb}
If $\cN(*E)$ is a holonomic $\shD_Z$-module, then $B^K(\cN)$ contains a
nonzero polynomial $b=b(s_1,\dots,s_r)$ of the form
\[
b
=
\prod_{(L,\alpha)}
\big(L\cdot(s_1,\dots,s_r)+\alpha\big)^{n_{L,\alpha}},
\]
where $L\in \Z_{\ge 0}^r$, $\alpha\in \C$, and only finitely many exponents
$n_{L,\alpha}$ are nonzero.
\end{lemma}

\begin{proof}
Take a nonzero $\bm\in K$. Since $\cN|_{Z_K}$ is a quotient of
$\cN|_{E_\bm}$, it suffices to prove that there exists such a polynomial in
the Bernstein--Sato ideal of $\cN$ along the monoid ideal generated by $\bm$.
This follows from the existence of Sabbah's generalized $b$-functions
\cite[\S 3]{Sab}; see also \cite[Thm.4.23]{Wuch}.
\end{proof}

Let
$F=(f_1,\dots,f_r)\colon X\to \A_\C^r$
be a regular map. As in Example \ref{ex:maingrebex},
\[
\cN_F
=
\shD_X[s_1,\dots,s_r]F^\bs
\]
is a $\shD_{Z,E}$-lattice of the regular holonomic $\shD_Z$-module
$\iota_{F,+}(\cO_X[1/f])$.
We write
\[
B_F^K=B^K(\cN_F)
\]
and call it the \emph{Bernstein--Sato ideal of $F$ along the monoid ideal
$K$}. If $K$ is generated by
$\bm_1,\dots,\bm_p\in \N^r$,
then one checks that $B_F^K$ agrees with the Bernstein--Sato ideal
$B_F^{\bm_1,\dots,\bm_p}$ introduced by Budur \cite[Def.4.2]{Budur}.

More generally, let $\cM$ be a regular holonomic $\shD_X$-module, and let
$\cN_0$ be a $\shD_{Z,E}$-lattice of the regular holonomic $\shD_Z$-module
\[
\cM\otimes_{\cO_X}\cS_F
\simeq
\iota_{F,+}(\cM[1/f]).
\]
For instance, one may take
$\cN_0
=
\shD_X[s_1,\dots,s_r]\cH_0\cdot F^\bs,$
where $\cH_0\subseteq \cM[1/f]$ is a finitely generated $\cO_X$-submodule
such that
$\shD_X[1/f]\cdot \cH_0=\cM[1/f].$

For a monoid ideal $K\subseteq \N^r$, the ideal $B^K(\cN_0)$ is called the
\emph{Bernstein--Sato ideal of $\cN_0$ along $K$ associated with $\cM$}.
More generally, for $\bm\in \N^r$, we define
\[
B^{K_\bm}(\cN_0)
=
B^K(\cN_0(*E_\bm))
\coloneqq
\Ann_{\C[s_1,\dots,s_r]}
\left(
\frac{\cN_0(*E_\bm)}
{\alpha(K)\cdot \cN_0(*E_\bm)}
\right),
\]
and set
\[
B_F^{K_\bm}
\coloneqq
B^{K_\bm}(\cN_F).
\]
Notice that when
$\bm=e=(0,\dots,0),$
the unit in the monoid,
we have $K_e=K$ and $E_e=0$, and hence
\[
B^{K_e}(\cN_0)=B^K(\cN_0).
\]
Thus the definitions of $B^{K_\bm}(\cN_0)$ and $B^K(\cN_0)$ are compatible.

We denote by
\[
Z\big(B^{K_\bm}(\cN_0)\big)
\subseteq
\Spec \C[s_1,\dots,s_r]
\]
the zero locus of this ideal, and call it the \emph{Bernstein--Sato variety
of $\cN_0$ along $K_\bm$ associated with $\cM$}.

\begin{theorem}\label{thm:nozeroproperBSI}
Let
$F=(f_1,\dots,f_r)\colon X\to \A_\C^r$
be a regular map, and let $K\subsetneq \N^r$ be a proper monoid ideal. Suppose
that $\cN(*E)$ is a nonzero regular holonomic $\shD_Z$-module supported on the
graph of $F$. If the support of $\cN$, as an $\cO_Z$-module, intersects
$Z_K$ nontrivially, then $B^K(\cN)$ is a nonzero proper ideal. In particular,
if
$(f_1=f_2=\cdots=f_r=0)\subseteq X$
is nonempty, then
\[
B_F^K\subseteq \C[s_1,\dots,s_r]
\]
is a nonzero proper ideal.
\end{theorem}

\begin{proof}
By Lemma \ref{lm:subbahgb}, it remains to prove that
$\cN|_{Z_K}\neq 0.$
Since the support of $\cN$ is a Zariski closed subset of $Z$, we may assume
that the support of $\cN$ meets $Z_{K_I}$ for some nonempty subset
$I\subseteq \{1,\dots,r\}$
such that
$Z_{K_I}\hookrightarrow Z_K.$
It is therefore enough to prove
$\cN|_{Z_{K_I}}\neq 0$
under the assumption that the support of $\cN$ intersects $Z_{K_I}$
nontrivially.

For simplicity, assume
$I=\{1,\dots,k\}$
for some $k\le r$, and write
\[
I_j=\{j,j+1,\dots,r\}
\]
for $j\le k$. We argue by induction on $|I|$ to prove that
$\cN|_{Z_{K_I}}\neq 0.$
By the inductive hypothesis, we may assume that
\[
\cN_j
\coloneqq
\cN|_{Z_{K_{\bar I_j}}}
\neq 0
\]
for $j<k$, and that the support of $\cN_j$, as an $\cO_Z$-module or as an
$\cO_{Y_j}$-module, intersects $Z_{K_I}$ nontrivially.

By Lemma \ref{lm:logsp}, $\cN_j$ is a finitely generated
$\sA_{Y_j,\bar E_j}^{R_{j-1}}\text{-module},$
where
\[
Y_j=Y_{\bar I_j}
\quad \text{and} \quad
\bar E_j=E_{I_j}\cap Y_j.
\]
By Theorem \ref{thm:mainsuppfiberation}, applied to the log structure
\[
\N^{j-1}
\hookrightarrow
S[t_j^{\pm1},\dots,t_r^{\pm1}][\N^{j-1}],
\]
the module $\cN_j(*\bar E_j)$ is relative regular holonomic over
$\sA_{Y_j}^{R_{j-1}}$.

Consider the short exact sequence
\[
0
\to
\cK_j
\to
\cN_j
\to
\cQ_j
\to
0,
\]
where $\cK_j$ and $\cQ_j$ are, respectively, the kernel and image of the
natural map
\[
\cN_j\to \cN_j(*\bar E_j)=\cN_j(*E_{I_j}).
\]
Then $\cQ_j$ is a lattice, and hence has no $t_i$-torsion for $i\ge j$.
Therefore we have a short exact sequence
\[
0
\to
\frac{\cK_j}{t_j\cK_j}
\to
\cN_{j+1}
\to
\frac{\cQ_j}{t_j\cQ_j}
\to
0.
\]

If the $\cO_{Y_j}$-support of $\cQ_j$ intersects $Z_{K_I}$ nontrivially, then
Corollary \ref{cor:puregradequot} gives
\[
\Ch^\dagger\left(\frac{\cQ_j}{t_j\cQ_j}\right)
=
\Ch^\dagger(\cQ_j)|_{(t_j=0)}.
\]
In particular, $\cN_{j+1}\neq 0$, and the support of $\cN_{j+1}$ intersects
$Z_{K_I}$ nontrivially.

Now suppose that the $\cO_{Y_j}$-support of $\cQ_j$ does not intersect
$Z_{K_I}$. Then the $\cO_{Y_j}$-support of $\cK_j$ must intersect $Z_{K_I}$
nontrivially by the inductive hypothesis. Consider the natural map
$\cK_j\to \cK_j(*D_{\widehat{j+1}})$,
where
\[
D_{\widehat{j+1}}
=
\left(
\prod_{\substack{i\in\{j,j+1,\dots,r\}\\ i\neq j+1}}t_i=0
\right)
\subseteq Y_j.
\]
Let $\cQ_j^1$ denote its image and $\cK_j^1$ its kernel. Then we have a short
exact sequence
\begin{equation}\label{eq:aaabbbses123}
0
\to
\cK_j^1
\to
\cK_j
\to
\cQ_j^1
\to
0.
\end{equation}
By construction, $\cQ_j^1$ is supported on
\[
D_{j+1}=(t_{j+1}=0)\subseteq Y_j,
\]
and $\cK_j^1$ is supported on $D_{\widehat{j+1}}$.

For $l=1,\dots,r-j$, set
\[
Y_j^l\coloneqq D_{j+l}
\quad \text{and} \quad
\bar D_{\widehat l}\coloneqq D_{\widehat{j+l}}\cap Y_j^l.
\]
By Lemma \ref{lm:logKashequiv}, $\cQ_j^1$ is an
$\sA_{Y_j^1,\bar D_{\widehat 1}}^{R_{j-1}[s_{j+1}]}\text{-lattice}.$
Since $\cN$ is relative regular holonomic over $\sA_X^{R_r}$ by
Theorem \ref{thm:mainsuppfiberation}, and since $\cQ_j^1$ is supported on the
graph of
\[
(f_j,f_{j+2},f_{j+3},\dots,f_r)
\]
as an $\sA_{Y_j^1,\bar D_{\widehat 1}}^{R_{j-1}[s_{j+1}]}$-module, Kashiwara
equivalence implies that
$\cQ_j^1(*\bar D_{\widehat 1})$
is relative regular holonomic over
$\sA_{Y_j^1}^{R_{j-1}[s_{j+1}]}.$

If the $\cO$-support of $\cQ_j^1$ intersects $Z_{K_I}$ nontrivially, then the
same argument applied to $\cQ_j$ shows that
${\cK_j}/{t_j\cK_j}\neq 0$,
and that the support of $\cN_{j+1}$ intersects $Z_{K_I}$ nontrivially.

If instead the $\cO$-support of $\cQ_j^1$ does not intersect $Z_{K_I}$, then
the $\cO$-support of $\cK_j^1$ intersects $Z_{K_I}$ nontrivially. Repeating
the same construction with
\[
\cK_j^1\to \cK_j^1(*D_{\widehat{j+2}}),
\]
we obtain another short exact sequence
\[
0
\to
\cK_j^2
\to
\cK_j^1
\to
\cQ_j^2
\to
0,
\]
where $\cQ_j^2$ is an
$\sA_{Y_j^2,\bar D_{\widehat 2}}^{R_{j-1}[s_{j+2}]}\text{-lattice}$
and
$\cQ_j^2(*\bar D_{\widehat 2})$
is relative regular holonomic over
$\sA_{Y_j^2}^{R_{j-1}[s_{j+2}]}$.
If the $\cO$-support of $\cQ_j^2$ intersects $Z_{K_I}$ nontrivially, then
${\cK_j^1}/{t_j\cK_j^1}\neq 0$,
and the support of this quotient intersects $Z_{K_I}$ nontrivially. Hence the
same is true for $\cK_j/t_j\cK_j$, and therefore for $\cN_{j+1}$.

Otherwise, we continue this process with
\[
\cK_j^2,\cK_j^3,\dots,\cK_j^{r-j}
\]
as necessary. In the worst case, the $\cO$-support of $\cQ_j^{r-j}$ does not
intersect $Z_{K_I}$. Then the $\cO$-support of $\cK_j^{r-j}$ intersects
$Z_{K_I}$ nontrivially. By construction, $\cK_j^{r-j}$ is supported on
$D_j$. Therefore, by Lemma \ref{lm:twosidedideal}, there exists $m\gg 0$ such
that
$t_j^m\cdot \cK_j^{r-j}=0$.
Working algebraically locally at a point in the intersection of the
$\cO$-support of $\cK_j^{r-j}$ with $Z_{K_I}$, it follows that
\[
\frac{\cK_j^{r-j}}{t_j\cK_j^{r-j}}\neq 0
\]
and that the support of this quotient intersects $Z_{K_I}$ nontrivially.
Since each $\cQ_j^a$ is a lattice, we have successive injections
\[
\frac{\cK_j^{r-j}}{t_j\cK_j^{r-j}}
\hookrightarrow
\frac{\cK_j^{r-j-1}}{t_j\cK_j^{r-j-1}}
\hookrightarrow
\cdots
\hookrightarrow
\frac{\cK_j}{t_j\cK_j}.
\]
Hence the supports of both
\[
\frac{\cK_j}{t_j\cK_j}
\quad \text{and} \quad
\cN_{j+1}
\]
intersect $Z_{K_I}$ nontrivially in this worst case as well. This completes
the induction and hence the proof.
\end{proof}

Let $I\subseteq\{1,\dots,r\}$ be the subset such that the support of
$E_\bm$ is
\[
E_I
\coloneqq
E_{\alpha(\sum_{i\in I}e_i)}.
\]
Let $\cN_0$ be a $\shD_{Z,E}$-lattice of the regular holonomic $\shD_Z$-module
$\iota_{F,+}(\cM[1/f]).$
Then $\cN_0(*E_\bm)$ becomes a
$\shD_{Z,E_{I^c}}(*E_\bm)$-lattice, where $I^c$ denotes the complement of
$I$. Since $\shD_{Z,E_{I^c}}(*E_\bm)$ corresponds to the log structure
\[
\alpha\colon
\N^r_\bm
\hookrightarrow
S[\N^r_\bm]
\simeq
S[t_1,\dots,t_r]_{\alpha(\bm)},
\]
the Bernstein--Sato ideal of $\cN_0(*E_\bm)$ along $Z_{K_\bm}$ is naturally an
ideal in
$\C[s_i]_{i\notin I}.$

\begin{prop}\label{prop:localizeBF}
With the notation and assumptions above, for $K\subseteq \N^r$ and
$\bm\in \N^r$, we have
\[
B^{K_\bm}(\cN_0)
=
B^{K_\bm}(\cN_0(*E_\bm))\cdot \C[s_1,\dots,s_r].
\]
In particular, the generators of $B^{K_\bm}(\cN_0)$ are independent of
$s_i$ for $i\in I$.
\end{prop}

\begin{proof}
By \eqref{eq:uncanincl}, the quotient
\[
\frac{\cN_0(*E_\bm)}
{\alpha(K)\cdot \cN_0(*E_\bm)}
\]
is an
$\sA_{Y_I}^{R_{I^c}}(*\bar E_I)\text{-module}$
supported on the graph of
$F_I=(f_i)_{i\in I}.$
By Kashiwara equivalence, similarly to \eqref{eq:getypeD}, we have
\[
\frac{\cN_0(*E_\bm)}
{\alpha(K)\cdot \cN_0(*E_\bm)}
\simeq
\iota_{F_I,+}\cS'
\simeq
\cS'[s_i]_{i\in I}\prod_{i\in I}f_i^{s_i},
\]
where $\cS'$ is an $\sA_X^{R_{I^c}}$-module. In particular, this quotient is
free over $\C[s_i]_{i\in I}$. The assertion follows.
\end{proof}

\begin{theorem}\label{thm:expofBF}
With the notation and assumptions of Proposition \ref{prop:localizeBF}, for
$K\subseteq \N^r$ and $\bm\in \N^r$, we have
\[
\Exp\big(Z(B^{K_\bm}(\cN_0))\big)
=
\Exp\big(\supp(\wt{i^!_{K_\bm}}\cN_0(*E))\big)
=
\supp\big(\psi_{K_\bm}(\DR(\cM))\big).
\]
As a consequent,
$\Exp\big(Z(B^{K_\bm}(\cN_0))\big)$
depends only on the support of $Z_{K_\bm}$. In particular,
\[
\Exp\big(Z(B_F^{K_\bm})\big)
=
\supp\big(\psi_{K_\bm}(\C_X)\big).
\]
\end{theorem}

\begin{proof}
By Theorem \ref{thm:sabbspcmp}, together with the faithful flatness of the
analytic sheafification functor, it suffices to prove
\[
\Exp\big(Z(B^{K_\bm}(\cN_0))\big)
=
\supp\left(
\pi_*^{\sG_r}
\big(\wt{i^!_{K_\bm}}\cN_0(*E)\big)
\right).
\]
Furthermore, by relative Kashiwara equivalence and
Proposition \ref{prop:localizeBF}, it is enough to prove
\[
\Exp\big(Z(B^{K}(\cN_0))\big)
=
\supp\left(
\pi_*^{\sG_r}
\big(\wt{i^!_{K}}\cN_0(*E)\big)
\right).
\]

Similarly to Proposition \ref{prop:doublelimitnearby}, we have
\begin{equation}\label{eq:prolimup!lattice}
\wt{i^!_{K}}\cN_0
\simeq
\varprojlim_j
\frac{\wt\cN_0}{\alpha(K)^j\cdot \wt\cN_0}.
\end{equation}
By substitution, if $K$ is generated by
$v_1,\dots,v_p\in \N^r$,
then
\[
\prod_{i=1}^p \alpha(v_i)\cdot B^K(\cN_0)
\subseteq
B^{K^2}(\cN_0).
\]
On the other hand, the natural surjection
\[
\frac{\cN_0}{K^2\cdot \cN_0}
\twoheadrightarrow
\frac{\cN_0}{K\cdot \cN_0}
\]
implies, by \eqref{eq:suppbideal=}, that
\[
Z(B^K(\cN_0))
\subseteq
Z(B^{K^2}(\cN_0)).
\]
Since
\[
\Exp\big(Z(\alpha(v_i)\cdot B^K(\cN_0))\big)
=
\Exp\big(Z(B^K(\cN_0))\big),
\]
we obtain
\[
\Exp\big(Z(B^K(\cN_0))\big)
=
\Exp\big(Z(B^{K^2}(\cN_0))\big).
\]
Inductively,
\[
\Exp\big(Z(B^K(\cN_0))\big)
=
\Exp\big(Z(B^{K^j}(\cN_0))\big)
\]
for all $j\ge 1$.

By \eqref{eq:relminexge} and Theorem \ref{thm:upper!DI}, for every
sufficiently small simply connected open subset
$V\subseteq \C^r$,
there exists $j\gg 0$ such that
\[
\wt{i^!_K}\cN_0|_{X\times V}
=
\frac{\wt\cN_0}{\alpha(K)^j\cdot \wt\cN_0}\bigg|_{X\times V}.
\]
Together with \eqref{eq:prolimup!lattice}, this gives
\[
\Exp\big(Z(B^K(\cN_0))\big)
=
\Exp\big(\supp(\wt{i^!_K}\cN_0)\big).
\]
Now set
$\mathbf{1}=(1,\dots,1)\in \N^r.$
The same argument applied to
$\alpha(\mathbf{1})^j\cdot \cN_0$
shows that
\[
\Exp\big(Z(B^K(\cN_0))\big)
=
\Exp\big(Z(B^K(\alpha(\mathbf{1})^j\cdot \cN_0))\big)
=
\Exp\big(\supp(\wt{i^!_K}(\alpha(\mathbf{1})^j\cdot \cN_0))\big)
\]
for every $j\in \Z$.
Moreover, similarly to \eqref{eq:prolimup!lattice}, for every sufficiently
small simply connected open subset $V\subseteq \C^r$, we have
\[
\wt{i^!_K}\cN_0(*E)|_{X\times V}
=
\frac{\wt{\cN_0^{-k}}}
{\alpha(K)^j\cdot \wt{\cN_0^{-k}}}
\bigg|_{X\times V}
\]
for some $j,k\gg 0$, where
\[
\cN_0^{-k}
=
\alpha(\mathbf{1})^{-k}\cdot \cN_0.
\]
Therefore,
\[
\Exp\big(Z(B^K(\cN_0))\big)
=
\supp\left(
\pi_*^{\sG_r}
\big(\wt{i^!_K}\cN_0(*E)\big)
\right).
\]
This proves the theorem.
\end{proof}

By the preceding theorem and Theorem \ref{thm:linearsupp}, we immediately
obtain the following.

\begin{coro}\label{cor:expofBF}
With the notation and assumptions of Theorem \ref{thm:expofBF},
\[
\Exp\big(Z(B^{K_\bm}(\cN_0))\big)
\subseteq
(\C^*)^r
\]
is a finite union of translated subtori for every
$K\subseteq \N^r$ and $\bm\in \N^r$. Moreover, if $\cM$ underlies a
$\Q$-mixed Hodge module, then
\[
\Exp\big(Z(B^{K_\bm}(\cN_0))\big)
\subseteq
(\C^*)^r
\]
is a finite union of torsion-translated subtori.
\end{coro}

Finally, we explain how the main results in the introduction follow.
Combining Theorem \ref{thm:expofBF} with Corollary \ref{cor:expofBF} gives
Theorem \ref{thm:mainA}. Theorem \ref{thm:linearsupp} yields
Theorem \ref{thm:mainC} and Theorem \ref{thm:mainD}.

   \appendix
\let\Subsection\subsection
\newcommand{\paragraphe}[1]{\refstepcounter{theorem}\subsubsection*{{\upshape\thetheorem.}\enspace#1}}

\newcommand{\bbullet}{{\scriptscriptstyle\bullet}}

\newcommand\shd{\mathcal{D}}\let\cD\shd
\newcommand\shf{\mathcal{F}}\let\cF\shf
\newcommand\shi{\mathcal{I}}\let\cI\shi
\newcommand\shh{\mathcal{H}}\let\cH\shh
\newcommand\shm{\mathcal{M}}\let\cM\shm
\newcommand\shn{\mathcal{N}}\let\cN\shn
\newcommand\sho{\mathcal{O}}\let\cO\sho

\newcommand{\NN}{\mathbb{N}}
\newcommand{\QQ}{\mathbb{Q}}
\newcommand{\ZZ}{\mathbb{Z}}

\newcommand{\bma}{\boldsymbol{a}}
\renewcommand{\bR}{\boldsymbol{R}}

\newcommand{\rd}{\mathrm{d}}
\renewcommand{\gr}{\mathrm{gr}}
\renewcommand{\coh}{\mathrm{coh}}
\renewcommand{\Mod}{\mathrm{Mod}}
\newcommand{\nb}{\mathrm{nb}}

\newcommand{\pOS}{p^{-1}\sho_S}
\newcommand{\XS}{X_S}
\newcommand{\TS}{T_S}
\newcommand{\XsS}{X^*_S}
\newcommand{\XT}{X_T}
\newcommand{\XTa}{X_{T_\alpha}}
\newcommand{\XST}{\XS\times_S\nobreak\TS}
\newcommand{\XpT}{X'_T}
\newcommand{\XTp}{X_{T'}}
\newcommand{\DXS}{\shd_{\XS/S}}
\newcommand{\DXsS}{\shd_{\XsS/S}}
\newcommand{\DXT}{\shd_{\XT/T}}
\newcommand{\DXpT}{\shd_{\XpT/T}}
\newcommand{\DXST}{\shd_{\XST/\TS}}
\newcommand{\DXSp}{\shd_{X_{S'}/S'}}

\newcommand{\Char}{\operatorname{Char}}
\renewcommand{\DR}{\operatorname{DR}}
\newcommand{\pDR}{\operatorname{{}^\mathrm{p}DR}}
\renewcommand{\id}{\operatorname{Id}}\let\Id\id
\newcommand{\rk}{\operatorname{rk}}
\newcommand{\Supp}{\operatorname{Supp}}\let\supp\Supp

\let\tilde\widetilde
\let\wt\widetilde
\let\wh\widehat
\let\ov\overline
\let\epsilon\varepsilon
\let\emptyset\varnothing
\let\setminus\smallsetminus
\let\leq\leqslant
\let\geq\geqslant
\let\moins\smallsetminus

\newcommand\loccit{loc.\kern3pt cit.{}\xspace}
\newcommand\cf{cf.\kern.3em}
\newcommand\eg{e.g.\kern.3em}
\newcommand\ie{i.e.,\ }
\newcommand\resp{\text{resp.}\kern.3em}

\renewcommand\to{\mathchoice{\longrightarrow}{\rightarrow}{\rightarrow}{\rightarrow}}
\newcommand\mto{\mathchoice{\longmapsto}{\mapsto}{\mapsto}{\mapsto}}
\newcommand\hto{\mathrel{\lhook\joinrel\to}}
\newcommand\from{\mathchoice{\longleftarrow}{\leftarrow}{\leftarrow}{\leftarrow}}
\newcommand\isom{\stackrel{\sim}{\longrightarrow}}

\newcommand\To[1]{\mathchoice{\xrightarrow{\textstyle\kern4pt#1\kern3pt}}{\stackrel{#1}{\longrightarrow}}{}{}}
\newcommand\From[1]{\mathchoice{\xleftarrow{\textstyle~#1~}}{\stackrel{#1}{\longleftarrow}}{}{}}
\def\Hto#1{\mathrel{\lhook\joinrel\To{#1}}}

\let\oldbigoplus\bigoplus
\renewcommand{\bigoplus}{\mathop{\textstyle\oldbigoplus}\displaylimits}
\let\oldbigcap\bigcap
\renewcommand{\bigcap}{\mathop{\textstyle\oldbigcap}\displaylimits}
\let\oldbigcup\bigcup
\renewcommand{\bigcup}{\mathop{\textstyle\oldbigcup}\displaylimits}
\let\oldprod\prod
\renewcommand{\prod}{\mathop{\textstyle\oldprod}\displaylimits}

\renewcommand{\D}[1]{{}^{\scriptscriptstyle\mathrm{D}\mkern-2mu}#1}
\newcommand{\Df}{{}^{\scriptscriptstyle\mathrm{D}\mkern-4mu}f}

\let\sectionname\relax

\bgroup
\renewcommand{\thefootnote}{(*)}\addtocounter{footnote}{-1}
\addtocontents{toc}{\protect\let\protect\footnotemark\protect\relax}
\part*{Appendix I. Relative characteristic varieties of relative\nobreakspace regular holonomic $\mathcal D$-modules, \textnormal{by Claude Sabbah \protect\footnotemark}}
\footnotetext{CMLS, CNRS, École polytechnique, Institut Polytechnique de Paris, 91128 Palaiseau cedex, France. 
\texttt{\email{Claude.Sabbah@polytechnique.edu}}}
\egroup

\section{Introduction}
Let $\varphi:X\to T$ be a smooth morphism of complex manifolds. Given a regular holonomic $\cD_X$-module $\cM$, we have given in \cite[Th.\,3.2]{Sab2} an estimation of the characteristic variety of any coherent $\cD_{X/T}$-submodule of $\cM$ in terms of $\Char\cM$ only. In~this note, we answer a question of Lei Wu and extend this estimate to the case of relative regular holonomic $\cD$-modules as defined in \cite{MF-S16}. Our main tools are those developed in \cite[\S\S2--4]{FFS23}.

As \cite{Sab2} is written in French, we review the proof of \cite[Th.\,3.2]{Sab2} in Section \ref{sec:absolute}, and we add some details for making clear the adaptation to the relative setting.

The proof of the main theorem of this note (Theorem \ref{th:bibi86IIrel}) follows the same line as that of \cite[Th.\,3.2]{Sab2}, so the review of Section \ref{sec:absolute} also serves as a guiding line for it.

\section{A review of some results of \cite{Sab2}}\label{sec:absolute}

\Subsection{The relative characteristic variety}\label{subsec:absolute}
\subsubsection*{Preliminary remarks on relative $\cD$-modules}

Let $X$ and $T$ be complex manifolds. We~denote by $\XT=X\times T$ their product and let $q:\XT\to T$ denote the projection. We~will consider objects $\cN$ in the category $\Mod(\DXT)$ of $\DXT$-modules. It~is equipped with the standard functors pullback and pushforward with respect to morphisms of the form $f\times\id_T$, with $f:X\to X'$. We use the notation $\Df^*$ and $\Df_*$ for the pullback and pushforward of $\cD$-modules (often also denoted by $f^+$ and~$f_+$), and $\Df{}^{(j)}_*$ for $\cH^j(\Df_*)$.

Kashiwara's equivalence holds in this category, namely, if $\iota:X\hto X'$ is a closed embedding of complex manifolds, the functor $\D\iota_*$ induces an equivalence between $\Mod_\coh(\DXT$ and $\Mod_{\coh,\iota(X)}(\DXpT)$.

Let $\cN$ be a coherent $\cD_{\XT/T}$-module. Its characteristic variety $\Char\cN$ is a closed conic subset of the relative tangent bundle $T^*(\XT/T)=(T^*X)\times T$ (\cf\cite[\S III.1.3]{Schpbook}). Let $C$ be an irreducible component of $\Char\cN$, with projection $Z$ in $\XT$ (it is an irreducible closed analytic subset of $\XT$). If $\varpi:T^*(\XT/T)\to \XT$ denotes the projection, then at a general point of $C$, the fiber of $q\circ\varpi$ (in $T^*X$) is involutive, hence of dimension $\geq\dim X$.

For an irreducible closed analytic subset $Z$ of $\XT$, there exists an open dense subset~$Z^\circ$ on which $q$ is of maximal rank. We~denote by $\dim_qZ$ the dimension of the fibers of $q|_{Z^\circ}$, \ie the generic relative dimension of $q|_Z$, which is thus $\leq\dim X$. Locally on $Z^\circ$, the image $q(Z^\circ)$ is a submanifold of $T$ of dimension $\dim Z-\dim_qZ$. Similarly, locally on $C^\circ$ as above, the image $q\circ\varpi(C^\circ)$ is a manifold of dimension $\dim C-\dim_{q\circ\varpi}C$, which is thus equal to $\dim Z-\dim_qZ$. We~conclude
\begin{equation}\label{eq:dimestimate}
\dim C=\dim_{q\circ\varpi}C+\dim Z-\dim_qZ\geq\dim X+\dim Z-\dim_qZ.
\end{equation}
Furthermore, equality holds if and only if at a general point of $C$, the fiber of $q\circ\varpi$ is Lagrangian. This property is equivalent to $C^\circ:=\varpi^{-1}Z^\circ\cap C$ being the relative conormal bundle $T^*_{q|_{Z^\circ}}(\XT/T)$.

\begin{definition}\label{def:relconormal}
If $Z$ is a closed analytic subset of $\XT$, the \emph{relative conormal space} $T^*_{q|_Z}(\XT/T)$ is, by definition, the closure in $T^*(\XT/T)$ of the relative conormal bundle $T^*_{q|_{Z^\circ}}(\XT/T)$, where~$Z^\circ$ is any open dense subset of $Z$ which is smooth and on which~$q|_Z$ is of maximal rank. It~is a closed conic analytic subset of $T^*(\XT/T)$.
\end{definition}

We can thus rephrase the above remark as:
\begin{equation}\label{eq:relconormal}
\begin{gathered}
\text{Equality holds in \eqref{eq:dimestimate} if and only if}\\
\text{$C$ is the relative conormal space $T^*_{q|_{Z}}(\XT/T)$ of $q|_Z$.}
\end{gathered}
\end{equation}

On the other hand, there exists an increasing filtration $\cF_\bbullet(\cN)$ by coherent $\DXT$-submodules such that each irreducible component of the characteristic variety of $\cF_j(\cN)/\cF_{j-1}(\cN)$ is of dimension~$j$ (\cf \cite[\S1.7]{BJ}, that can be readily adapted to the case of $\DXT$).

Let $\iota:X\hto X'$ be a closed inclusion. It~induces a pair of morphisms
\[
(T^*X)\times T\From{T^*\iota}(T^*X')|_X\times T\Hto{\iota}(T^*X')\times T
\]
and for any conic closed analytic subset $C$ of $T^*(\XT/T)$, its $\iota$-pushforward is defined as $\iota[(T^*\iota)^{-1}C]$. In a local chart of $X$, we~can write $X'=X\times Y$, $q'=q\circ p$, where $p$ is the projection $X'\to X$, and $\iota$ is the inclusion $X=X\times\{y_o\}\hto X\times Y$. In such a chart, for $C\subset T^*(\XT/T)$, we~have
\[
T^*(\XpT/T)=T^*(\XT/T)\times (T^*Y),\quad\text{and}\quad \iota[(T^*\iota)^{-1}C]=C\times (T^*_{y_o}Y).
\]
The following is then clear.

\begin{lemma}
For an irreducible closed conic analytic subset $C$ of $T^*(\XT/T)$ which is generically relatively involutive, we~have $C=T^*_{q|_Z}(\XT/T)$ with $Z=\varpi(C)$, equivalently $\dim C=\dim X+\dim Z-\dim_qZ$, if and only if $\iota[(T^*\iota)^{-1}C]=T^*_{q'|_{\iota(Z)}}(\XpT/T)$.
\end{lemma}

\subsubsection*{The relative characteristic variety of a $\cD_X$-module}
Let $\varphi:X\to T$ be a holomorphic map and let $\cM$ be a coherent $\cD_X$-module. We~let $\cM_\varphi$ denote its pushforward by the graph embedding $\iota_\varphi:X\hto \XT$.

\begin{lemma}[{\cite[Chap.\,III, Lem.\,1.3.3]{Schpbook}}]\label{lem:independant}
Given a coherent $\cD_X$-module $\cM$, the charac\-teristic variety of a coherent $\DXT$-submodule that generates $\cM_\varphi$ only depends on~$\cM$.\qed
\end{lemma}

We denote this characteristic variety as $\Char_{X_T/T}(\cM)\subset (T^*X)\times T$ (it is denoted by $\Char^1_\varphi(\cM)$ in \cite{Schpbook}). Although a coherent $\DXT$-submodule of $\cM$ may not exist globally, it exists locally and, for most arguments, a local setting is enough to be considered, according to the independence property, so we can make use of such a coherent $\DXT$-submodule without justification.

\begin{theorem}[{\cite[Th.\,3.2]{Sab2}}]\label{th:bibi86II}
Let $\cM$ be a regular holonomic $\cD_X$-module and let \hbox{$\varphi:X\!\to\! T$} be a holomorphic map. Then each irreducible component of the relative characteristic variety $\Char_{X_T/T}\cM$ takes the~form~\hbox{$T^*_{q|_{\iota_\varphi(Z)}}\!(\XT/T)$}, for some irreducible closed analytic subset of $X$. Furthermore, $T^*_ZX$ is some irreducible component of $\Char\cM$.
\end{theorem}

\begin{remark}[T.\,Monteiro Fernandes]\label{rem:Teresa}
For a closed analytic subset $Z\subset X$, one can interpret the relative conormal space $T^*_{q|_{\iota_\varphi(Z)}}\!(\XT/T)$ as the image of the conormal space $T^*_{\iota(\varphi)}\XT$ by the projection $\wh q:T^*\XT\to T^*(\XT/T)$: indeed, this is clear on the open set $Z^\circ$ where $q|_{\iota_\varphi(Z^\circ)}$ is smooth; on the other hand, the image of $T^*_{\iota(\varphi)}\XT$ by $\wh q$ is conic, irreducible and closed analytic, hence equal to the closure of $T^*_{q|_{\iota_\varphi(Z^\circ)}}\!(\XT/T)$, that~is, $T^*_{q|_{\iota_\varphi(Z)}}\!(\XT/T)$ by definition. The theorem states that $\Char_{X_T/T}(\cM)\subset\wh q(\Char\cM_\varphi)$. On the other hand, the reverse inclusion holds for any coherent $\cD_{\XT}$-module (\cf\cite[Chap.\,III, Th.\,1.3.5(d)]{Schpbook}). We can therefore restate the theorem by saying that
\[
\Char_{X_T/T}(\cM)=\wh q(\Char\cM_\varphi).
\]
\end{remark}

\begin{example}
We will show that the statement of the theorem is not sensible to closed inclusions $T'\hto T$.

Let $\cM$ be a coherent $\cD_X$-module. Let $\iota:T'\hto T$ be a closed submanifold of $T$ and assume that $\cM$ is supported on $\varphi^{-1}(T')$. Then $\cM_\varphi$ is supported on $\XTp$. By Kashiwara's equivalence, we can write $\cM_\varphi=\D\iota_*\cM'$ for some coherent $\cD_{\XTp}$\nobreakdash-mod\-ule~$\cM'$. Let~$\cN'$ be a coherent $\cD_{\XTp/T'}$-submodule of $\cM'$ that generates it. Then $\cN:=\iota_*\cN'$ generates $\D\iota_*\cM'=\cM$, as~can be seen by using local coordinates on~$T$ adapted to $T'$. Assume now that $\cM$, hence $\cM'$, is holonomic. Each irreducible component of $\Char\cM$ takes the form $T^*_ZX$ with $Z$ irreducible closed analytic in $\varphi^{-1}(T')$, so that $\iota_\varphi(Z)\subset\XTp$. Then $\Char\cM'$ is the union of the sets $T^*_{\iota_\varphi(Z)}(\XTp)$ for such components~$Z$. If we know that any irreducible component of $\Char\cN'$ takes the form $T^*_{q'|_{\iota_\varphi(Z)}}(\XTp/T')$ for some irreducible component $T^*_{\iota_\varphi(Z)}(\XTp)$ of $\Char\cM'$, where $q':\XTp\to T'$ denotes the projection, then $\Char\cN=\iota(\Char\cN')$ has components of the form
\[
\iota[T^*_{q'|_{\iota_\varphi(Z)}}(\XTp/T')]=T^*_{q|_{\iota_\varphi(Z)}}(\XT/T).
\]
\end{example}

\subsubsection*{The case of smooth a morphism}
If $\varphi:X\to T$ is a smooth morphism, the results above can equivalently been expressed already on $X$. We~now make the translation explicit. We~let $\cD_{X/T}$ denote the sheaf of relative differential operators, which is well-defined in this smooth context. For a coherent $\cD_X$-module $\cM$, the characteristic variety of a coherent $\cD_{X/T}$-submodule that generates it is independent of this submodule (analogous to Lemma \ref{lem:independant}). We~denote it by $\Char_{X/T}\cM$. 

Let us consider the category of smooth complex manifolds over~$T$, that is, of pairs $(X,\varphi)$, where $\varphi:X\to T$ is a \emph{smooth morphism}. A morphism $(X,\varphi)\to(X',\varphi')$ in this category is nothing but a morphism of complex manifolds $X\to X'$ making the following diagram commutative:
\[
\begin{tikzcd}
X \arrow[dr, "\varphi" swap] \arrow[rr] && X'\arrow[dl, "\varphi'"] \\
& T
\end{tikzcd}
\]
It is equipped with the standard functors pullback and pushforward with respect to morphisms over $T$.

Let $\iota:(X,\varphi)\hto(X',\varphi')$ be a morphism which is a closed embedding, \ie \hbox{$\iota:X\hto X'$} is a closed embedding. Kashiwara's equivalence holds in the relative setting, that is, $\D\iota_*$ induces an equivalence $\Mod_\coh(\cD_{X/T})\simeq\Mod_\coh(\cD_{X'/T})$. Furthermore, this equivalence is compatible with that on $\Mod_\coh(\cD_X)$, that is, if $\cN$ is a coherent $\cD_{X/T}$-module contained in a coherent $\cD_X$-module $\cM$, then, with obvious notation, $\D\iota_*\cN$ is a coherent $\cD_{X'/T}$-submodule of the coherent $\cD_{X'}$-module $\D\iota_*\cM$. In~other words, we~can use Kashiwara's equivalence for pairs $(\cM,\cN)$.

The embedding $\iota:(X,\varphi)\hto(X',\varphi)$ induces a pair of morphisms
\[
T^*(X/T)\From{T^*\iota}T^*(X'/T)|_X\Hto{\iota}T^*(X'/T)
\]
and for any conic closed analytic subset $C$ of $T^*(X/T)$, its $\iota$-pushforward is defined as $\iota[(T^*\iota)^{-1}C]$. In a local chart of $X$, we~can write $X'=X\times Y$, $\varphi'=\varphi\circ p$, where $p$ is the projection $X'\to X$, and $\iota$ is the inclusion $X=X\times\{y_o\}\hto X\times Y$. In such a chart, for $C\subset T^*(X/T)$, we~have
\[
T^*(X'/T)=T^*(X/T)\times (T^*Y),\qquad \iota[(T^*\iota)^{-1}C]=C\times (T^*_{y_o}Y).
\]
Applying these equalities to the graph embedding $\iota_\varphi$, we deduce that
\[
\iota[(T^*\iota)^{-1}\Char_{X/T}\cM]=\Char_{X_T/T}\cM.
\]

We can mimic Definition \ref{def:relconormal} when $\varphi$ is a smooth morphism:

\begin{definition}
If $Z$ is a closed analytic subset of $X$, the \emph{relative conormal space} $T^*_{\varphi|_Z}(X/T)$ is, by definition, the closure in $T^*(X/T)$ of the relative conormal bundle $T^*_{\varphi|_{Z^\circ}}(X/T)$, where~$Z^\circ$ is any open dense subset of $Z$ which is smooth and on which~$\varphi|_Z$ is of maximal rank. It~is a closed conic analytic subset of $T^*(X/T)$.
\end{definition}

The following is then equivalent to Theorem \ref{th:bibi86II}, and is stated like this in \cite[Th.\,3.2]{Sab2}.

\begin{theorem}\label{thm:equivformrelcc}
Let $\cM$ be a regular holonomic $\cD_X$-module and let \hbox{$\varphi:X\!\to\! T$} be a smooth morphism. Then, for any coherent $\cD_{X/T}$-submodule $\cN$ of $\cM$, each irreducible component of the characteristic variety $\Char\cN$ takes the~form~\hbox{$T^*_{\varphi|_{Z}}\!(X/T)$}, where $T^*_ZX$ is some irreducible component of $\Char\cM$.
\end{theorem}

\begin{remark}[analogous to Remark \ref{rem:Teresa}]
Let $\wh\varphi:T^*X\to T^*(X/T)$ denote the projection. The characteristic variety $\Char\cN$ does not depend on $\cN$ generating $\cM$ and the theorem asserts then that it is equal to the projection $\wh\varphi(\Char\cM)$.
\end{remark}

\subsection{Review of the proof of Theorem \ref{th:bibi86II}}
We first notice that the question is local on $\XT$. Furthermore, there is an easy special case of Theorem \ref{th:bibi86II}, without the regularity assumption:

\begin{lemma}\label{lem:nonchar}
Let $\cM$ be a holonomic $\cD_X$-module and let $\varphi:X\to T$ be a morphism which is non-characteristic with respect to $\cM$. Then, if $\Char\cM=\bigcup_{i\in I}T^*_{Z_i}X$, we~have $\Char_{X_T/T}\cM=\bigcup_{i\in I}T^*_{q|_{\iota_\varphi(Z_i)}}(\XT/T)$.
\end{lemma}

\begin{proof}
Let us recall the proof for completeness. The assumption implies that, for each irreducible component $Z_i$ of $\Char\cM$, generically on $Z_i$, $\varphi$ is of rank $\dim T$. Similarly, $q$ is non-characteristic with respect to $\cM_\varphi$, whose characteristic variety is $\bigcup_iT^*_{\iota_\varphi(Z_i)}(\XT)$. This means that $\Char\cM_\varphi\cap (X\times T^*T)$ is contained in the zero section or, equivalently, the projection $\Char\cM_\varphi\to T^*(\XT/T)$ is finite. Furthermore, generically on $\iota_\varphi(Z_i)$, $q|_{\iota_\varphi(Z_i)}$ is of rank $\dim T$. Therefore, a local good filtration of the $\cD_{\XT}$-module $\cM_\varphi$ is also good when $\cM_\varphi$ is considered as a $\DXT$-module. In~other words, $\cM_\varphi$ is $\DXT$-coherent, and its characteristic variety as such, that~is, $\Char_{X_T/T}\cM$, is the finite projection of $\Char\cM_\varphi$ to $T^*(\XT/T)$. For each $i\in I$, the projection $C_i$ of $T^*_{\iota_\varphi(Z_i)}(\XT)$ has thus dimension $\dim X+\dim T$, and $\dim q\circ\varpi (C_i)=\dim T$, so by \eqref{eq:dimestimate}, $C_i$ is the relative conormal space of its projection~$Z_i$ to $\XT$.
\end{proof}

\paragraphe{Determination of the components of $\Char_{X_T/T}\cM$}\label{par:detabsolute}
Assume that we know that each component of $\Char_{X_T/T}\cM$ takes the~form~\hbox{$T^*_{q|_{\iota_\varphi(Z)}}\!(\XT/T)$}. We~will show that $T^*_ZX$ is an irreducible component of $\Char\cM$.

\begin{enumerate}
\item\label{par:detabsolute1}
We first assume that $\Char_{X_T/T}\cM$ is equidimensional. By the first part of the theorem, it can be written as $\bigcup_{j\in J}T^*_{q|\iota_\varphi(Z_j)}(\XT/T)$ for some irreducible closed analytic subsets $Z_j\subset\supp\cM$. On the other hand, we~write the irreducible components of $\Char\cM$ as $T^*_{Z_i}X$ ($i\in I$) and we fix an analytic Whitney stratification $(X_\alpha)_{\alpha\in A}$ of $\supp\cM$ compatible with $(Z_k)_{k\in I\cup J}$ (each $X_\alpha$ is locally closed, smooth and connected, and its closure is a closed analytic subset of $X$ which is the union of other strata; each~$Z_k$ is a union of strata) such that, on each $X_\alpha$, the rank of $\varphi|_{X_\alpha}$ is constant.

Let~$X_\alpha$ be a stratum on which this rank is maximal, equal to $r$, say. If $X_\alpha$ is contained in the closure of another stratum $X_\beta$, then $\varphi|_{X_\beta}$ is also of maximal rank~$r$.
Let $x_o\in X_\alpha$. Then, up to replacing $X$ with an open neighborhood $\nb(x_o,X)$, we~can assume that $\varphi(X_\alpha)=T_\alpha$ is a submanifold of $T$ and $\cM_\varphi$ is supported on $\XTa$, hence of the form $\D\iota_{\alpha*}(\cM_\alpha)$ for some coherent $\cD_{\XTa}$-module $\cM_\alpha$.  In particular, near any $x_o\in X_\alpha$, the projection $q_\alpha:\XTa\to T_\alpha$ is non-characteristic with respect to $\cM_\alpha$ near $\iota_\varphi(x_o)$: indeed, if $X_\alpha\subset \ov X_\beta$ then, by Whitney condition (a), $(\ov{T^*_{\iota_\varphi(X_\beta)}(\XT)})_{x_o}\subset (T^*_{\iota_\varphi(X_\alpha)}(\XT))_{x_o}$ and if the intersection with $X\times(T^*T)$ of the latter is zero, the same holds for the former in a neighborhood of $\iota_\varphi(x_o)$. We~conclude with Lemma \ref{lem:nonchar} for $\Char_{q_\alpha}\cM_\alpha$ and then for $\Char_{X_T/T}\cM$ in the neighborhood of $\iota_\varphi(x_o)$.

Letting $A_r\subset A$ being the subset of $\alpha\in A$ such that $\varphi$ is of maximal rank $r$ on~$X_\alpha$, we deduce that the restriction to $\bigcup_{\alpha\in A_r}X_\alpha$ of the families $(Z_j)_{j\in J}$ and $(Z_i)_{i\in I}$ coincide. For $j\in J$ such that $Z_j\cap (\bigcup_{\alpha\in A_r}X_\alpha)\neq\emptyset$, we have $\dim T^*_{q|_{\iota_\varphi(Z_j)}}(\XT/T)=\dim X+r$. Assume now that some $Z_j$ ($j\in J$) does not cut $\bigcup_{\alpha\in A_r}X_\alpha$, so that $\varphi|_{Z_j}$ has generic rank $<r$. It~follows that, for such a $j$, $\dim T^*_{q|_{\iota_\varphi(Z_j)}}(\XT/T)<\dim X+r$. By our assumption of equidimensionality, such a $j$ does not exist, and the assertion is proved in this case.

\item\label{par:detabsolute2}
If $\Char_{X_T/T}\cM$ is not equidimensional, let $d$ be its dimension and let $\cN$ be a generating $\DXT$-submodule of $\cM_\varphi$, with filtration $\cF_\bbullet\cN$. We~argue by induction on~$d$. The $\cD_{\XT}$-submodule of $\cM_\varphi$ generated by $\cF_{d-1}\cN$ is holonomic and supported on $\iota_\varphi(X)$, so that it takes the form $\cM'_\varphi$ for some holonomic $\cD_X$-submodule $\cM'$ of~$\cM$. We~notice that $\cN\cap\cM'_\varphi=\cF_{d-1}\cN$: the inclusion $\supset$ is clear and, by induction, both submodules have the same characteristic variety as they generate $\cM'_\varphi$; it follows that the characteristic variety of $\cN\cap\cM'_\varphi$ has dimension $\leq d-1$, so $\cN\cap\cM'_\varphi$ is contained in $\cF_{d-1}\cN$.

Set $\cM''=\cM/\cM'$. The image $\cN''$ of $\cN$ in $\cM''_\varphi$ is thus equal to $\cF_d\cN/\cF_{d-1}\cN$, and we can apply to it the first part of the proof. We~conclude by induction that the assertion holds for $\Char_{X_T/T}\cM$.\qed
\end{enumerate}

\paragraphe{Induction hypothesis for the first part of Theorem \ref{th:bibi86II}}\label{par:pffirstpart}
For a closed analytic subset \hbox{$Z\!\subset\!X$}, we set $\dim\varphi(Z):=\dim Z-\dim_\varphi Z$. We~also set $\dim\varphi(\supp\cM)=\max\dim\varphi(Z)$, where the max is taken on the irreducible components of $\supp\cM$. The proof is by induction on $(d,\delta)=(\dim\supp\cM,\dim\varphi(\supp\cM))$ lexicographically ordered. We~note that $\delta\leq d$.

Let us fix $(d,\delta)$. If $\delta=0$, $\supp\cM$ is contained in a locally finite union of fibers of $\varphi$. According to Kashiwara's equivalence, we are reduced to the case where $T$ is a point, in which case the result is clear. We~can thus assume that $\delta\geq1$ and that the first part of the theorem holds if $(\dim\supp\cM,\dim\varphi(\supp\cM))<(d,\delta)$. Note that the second part also holds, by \S\ref{par:detabsolute}.

\begin{lemma}\label{lem:push}
Let $\pi:X'\to X$ be a morphism from a complex manifold $X'$ and set $\varphi'=\varphi\circ\pi$. Let $\cM'$ be a regular holonomic $\cD_{X'}$-module such that
\[
(\dim\supp\cM',\dim\varphi'(\supp\cM'))\leq(d,\delta).
\]
Assume that $\pi$ is proper on $\supp\cM'$, that $\cM'$ is good,\footnote{It is well-known (\cf\cite{Malgrange04}) that any regular holonomic $\cD_{X'}$-module is good, and, even more, is generated by a coherent $\cO_{X'}$-submodule; however, in view of the generalization to a relative setting in Section~\ref{sec:relative}, we avoid making use of this property that is not known in the relative setting.} and that the first part of the theorem holds for~$\cM'$. Then it holds for $\D\pi_*^{(j)}\cM'$ for each $j$. 
\end{lemma}

\begin{proof}
The question is local on $X$. Let $x_o\in X$. In the neighborhood of the compact set $\pi^{-1}(x_o)$, by the goodness assumption, $\cM'$ is generated by a coherent $\cO_{X'}$\nobreakdash-submodule, hence $\cM'_{\varphi'}$ is generated by a coherent $\DXpT$-submodule $\cN'$. In the following, we restrict $X$ to a suitable neighborhood of $x_o$. We~can regard $\pi$ as the composition of a closed inclusion and a projection. The first case being easy, we assume that $\pi$ is a projection and we express the direct image both of $\cM'$ and $\cN'$ in terms of the relative de~Rham complex $\pDR_{\XpT/\XT}:=\DR_{\XpT/\XT}[\dim X'-\dim X]$. It~follows that we have a natural morphism
\[
\D\pi_*^{(j)}\cN'=\cH^j\bR\pi_*(\pDR_{\XpT/\XT}\cN')\to\cH^j\bR\pi_*(\pDR_{\XpT/\XT}\cM')=\D\pi_*^{(j)}\cM'
\]
whose image $\cN_j$ is a coherent $\DXT$-submodule that generates $\D\pi_*^{(j)}\cM'$. It~is thus enough to prove that the characteristic variety of $\cN_j$ is a union of relative conormal spaces. On the other hand, we have \emph{Kashiwara's estimate}:
\begin{equation}\label{eq:Kestimate}
\Char\cN_j\subset \Char(\D\pi_*^{(j)}\cN')\subset\pi\bigl(\Char(\cN')\cap\pi^*T^*(\XT/T)\bigr).
\end{equation}
The end of the proof makes use of the next geometric  lemma.

\begin{lemma}[{\cite[Lem.\,3.5]{Sab2}}]\label{lem:relconormal}
Let $Z'$ be an irreducible closed analytic subset of $\XpT$ which is proper over $\XT$. Then the irreducible components of
\[
\pi\bigl[T^*_{q'|_{Z'}}(\XpT/T)\cap\pi^*T^*(\XT/T)\bigr]
\]
have dimension $\leq \dim X+\dim q'(Z')$, and those of that dimension are generically relatively Lagrangian (\ie are relative conormal spaces).\qed
\end{lemma}

In order to end the proof of Lemma \ref{lem:push}, it suffices to show that the components of $\Char\cN_j$ of dimension $\dim X+\delta$ are relative conormal spaces. Indeed, by the inclusion \eqref{eq:Kestimate} and the above lemma, together with the assumption of Lemma \ref{lem:push}, we have $\dim X+\delta\geq\dim\Char\cN_j$. On the other hand, any irreducible component $C$ of $\cF_{\dim X+\delta-1}\cN_j$ satisfies, according to \eqref{eq:dimestimate},
\[
\dim X+\delta-1\geq\dim C\geq\dim X+\dim q\circ\varpi(C),
\]
so that $\dim q\bigl(\supp\cF_{\dim X+\delta-1}\cN_j\bigr)\leq\delta-1$ and, by induction applied to the $\cD_{\XT}$-submodule it generates in $\D\pi^{(j)}_*\cM'$, each such $C$ is a relative conormal space.

Lastly, the components of $\Char\cN_j$ of dimension $\dim X+\delta$ are among those of $\pi\bigl(\Char\cN'\cap\pi^*T^*(\XT/T)\bigr)$, and the assertion for them follows from Lemma \ref{lem:relconormal}. This ends the proof of Lemma \ref{lem:push}.
\end{proof}

We now prove the first part of Theorem \ref{th:bibi86II}.

\paragraphe{Reduction to the normal crossing situation}
\label{par:reductionnc}
We fix $x_o$ in $X$ and work in a sufficiently small neighborhood of $x_o$. We~fix coordinates $(t_1,\dots,t_\kappa)$ of $T$ in the neighborhood of $\varphi(x_o)$, so that $\varphi=(\varphi_1,\dots,\varphi_\kappa)$. Let $\pi:X'\to X$ be a proper morphism which is a proper modification of $\supp\cM$ satisfying the following property: there exists a divisor with normal crossing $D\subset X'$ such that
\begin{itemize}
\item
$\pi|_{X'\moins D}$ is an isomorphism onto its image;
\item
$\D\pi^*\cM$ is smooth (\ie a vector bundle with connection) on $X'\moins D$;
\item
each divisor $(\varphi_i\circ\pi)^{-1}(0)$ is a union of components of $D$.
\end{itemize}
The pullback $\cM':=\cH^0(\D\pi^*\cM)$ is known to be regular holonomic. Furthermore, it satisfies $(\dim\supp\cM',\dim\varphi'(\supp\cM'))=(\dim\supp\cM,\dim\varphi(\supp\cM))=(d,\delta)$ (with $\varphi'=\varphi\circ\pi$). Also, being the pullback of a good (because local) $\cD_X$-module, it is also good. Lastly, these properties also apply to the localized $\cD_{X'}$-module $\cM'(*D)$. Let us set $j=\dim X-\dim X'=\dim X-d$. Then there exists a natural adjunction morphism
\[
\cM\to\D\pi_*^{(j)}(\cM')\to\D\pi_*^{(j)}(\cM'(*D))
\]
whose kernel and cokernel are supported in dimension $<d$. By the induction hypothesis, it is thus enough to prove the theorem for the right-hand side and, according to Lemma \ref{lem:push}, to $\cM'(*D)$ with respect to $\varphi'$.

\paragraphe{The case with normal crossings}\label{par:endabsolute}
The question is now local on $X'$, and we can work with a local coordinate system $x=(x_1,\dots,x_d)$ of $X'$ which is adapted to~$D$. We~assume that $D=\{x_1\cdots x_\ell=0\}$, so that each $\varphi'_k(x_1,\dots,x_d)$ writes $u_k(x)x^{\bma_k}$ with $\bma_k=(a_{k,1},\dots,a_{k,\ell})\in\NN^\ell$ for each $k=1,\dots,\kappa$, and $u_k(x)$ is an invertible holomorphic function. Let us set $\psi(x)=(x^{\bma_1},\dots,x^{\bma_\kappa})$. The embedding~$\iota_{\varphi'}$ decomposes as $\eta\circ\iota_\psi$, where $\eta:\XpT\to \XpT$ is a local isomorphism defined by means of the units $u_k$. For our purpose, we can thus replace~$\varphi'$ with $\psi$ and assume that all components of~$\varphi'$ are monomials.

For $i\in I:=\{1,\dots,\ell\}$, let $A_i$ denote the linear form on $\CC^\kappa$ with coefficients $a_{k,i}$ (\hbox{$k=1,\dots,\kappa$}). The rank of $\varphi'$ on $X'\moins D$ is the rank of the matrix $A=(A_1,\dots,A_\ell)$. Therefore, we have $\delta=\rk(A)$.

From the classification of regular meromorphic connections with poles along a normal crossing divisor, we deduce that $\cM'(*D)$ is an extension of $\cD_{X'}$-modules $\cD_{X'}/\cI$, where the left ideal $\cI$ is generated by the operators
\begin{equation}\label{eq:presentationMprime}
\begin{cases}
x_i\partial_{x_i}-\alpha_i\;\;(\alpha_i\in\CC\moins\ZZ),&\text{if }i\in I,\\
\partial_{x_i}&\text{if }i\notin I.
\end{cases}
\end{equation}
Then $\cM'_{\varphi'}$ is $\cD_{\XpT}$-generated by one section written $e\otimes\delta(t-\varphi')$ satisfying the relations
\[
\begin{cases}
\begin{aligned}
t_k\cdot e\otimes\delta(t-\varphi')&=x^{\bma_k}\cdot e\otimes\delta(t-\varphi')&&\text{for }k=1,\dots,\kappa,\\
x_i\partial_{x_i}\cdot e\otimes\delta(t-\varphi')&=[-A_i(\partial_tt)+\alpha_i]\cdot e\otimes\delta(t-\varphi')&&\text{for }i\in I,\\
\partial_{x_i}\cdot e\otimes\delta(t-\varphi')&=0&&\text{for }i\notin I.
\end{aligned}
\end{cases}
\]
Let $I_\delta\subset I$ be a subset of cardinal $\delta$ such that $(A_i)_{i\in I_\delta}$ is a maximal subfamily of $(A_i)_{i\in I}$ consisting of independent vectors. From the above relations, we deduce that $e\otimes\delta(t-\varphi')$ is annihilated by operators in $\cD_{\XpT/T}$ of the form
\begin{equation}\label{eq:operators}
\begin{cases}
x_j\partial_{x_j}-\Bigl[\sum_{i\in I_\delta}b_{ij}x_i\partial_{x_i}\Bigr]+\beta_j&b_{ij}\in\QQ,\;\beta_j\in\CC,\;j\in I\moins I_\delta,\\
\partial_{x_i}&i\notin I.
\end{cases}
\end{equation}
It follows that $\Char_{X'_T/T}(\cM'(*D))$ is contained in the subvariety $C'$ of $(T^*X')\times T$, equipped with coordinates $\bigl((x_i,\xi_i)_{i=1,\dots,d},(t_1,\dots,t_\kappa)\bigr)$, defined by the equations
\[
\begin{cases}
t_k-x^{\bma_k}&\text{for }k=1,\dots,\kappa,\\
x_j\xi_j-\Bigl[\sum_{i\in I_\delta}b_{ij}x_i\xi_i\Bigr]&\text{for }j\in  I\moins I_\delta,\\
\xi_i&\text{for }i\notin.
\end{cases}
\]

\begin{lemma}\label{lem:irreducible}
The set $C'$ is irreducible of dimension $d+\delta$.\qed
\end{lemma}

The proof is indicated in \cite[p.\,233]{Sab2}. Let $\cN'$ be a generating $\DXpT$-submodule of~$\cM'_{\varphi'}$. We~conclude that $\cN'=\cF_{d+\delta}(\cN')$. On the other hand, we know by \eqref{eq:dimestimate} that the component $C$ of $\Char_{X'_T/T}(\cM'(*D))$ which projects onto $\iota_{\varphi'}(X')$ has dimension $\geq d+\delta$. This component is thus equal to $C'$. We~conclude that $C'=\Char\gr_{d+\delta}^\cF\cN'$, which is thus a relative conormal space according to \eqref{eq:relconormal}. By the induction hypothesis, the theorem holds for $\Char\cF_{d+\delta-1}\cN'$ since the support of the $\cD_{X'}$-submodule of $\cM'(*D)$ generated by $\cF_{d+\delta-1}\cN'$ has dimension $<d$. We~conclude that it holds for $\cM'(*D)$, as~desired.\qed

\section{Generalization to a relative setting}\label{sec:relative}

Let $S$ be another complex manifold. We~consider the category of complex manifolds $\XS:=X\times S$ with morphisms of the form $\varphi\times\id_S=:\varphi_S$, where $\varphi$ is a holomorphic map $X\to X'$. We~also consider the category $\Mod(\DXS)$ with pullback and pushforward morphisms along morphisms $\varphi_S$. The full subcategory of holonomic $\DXS$-module~$\cM$, as~considered in \cite{FFS23}, consists of those coherent $\DXS$-modules whose characteristic variety $\Char\cM$ is contained in $\Lambda\times S\subset (T^*X)\times S=T^*(\XS/S)$ for some conic Lagrangian analytic subset $\Lambda$ of $T^*X$. Each irreducible component of $\Char\cM$ writes $T^*_{Z_i}X\times \Sigma_{i\ell}$ for some irreducible closed analytic subsets $Z_i\subset X$ and $\Sigma_{i\ell}\subset S$ (\cf\cite[Lem.\,2.10]{FTM18}).\footnote{We can regard such characteristic varieties as relative conormal spaces to closed analytic subsets of the form $Z\times \Sigma$ of $X\times S$, while in Section \ref{sec:absolute} we would consider any closed analytic subset of $X\times S$.} The family $(Z_i)_{i\in I}$ is locally finite on $X$, and the union of the sets $Z_i$ ($i\in I$) is, by definition, the $X$-support of $\cM$, denoted by $\supp_X\cM$. For each fixed~$i$, the family $(\Sigma_{i\ell})_\ell$ is locally finite on $S$. The union of the family $(\Sigma_{i\ell})_{i\ell}$ is called the $S$-support of $\cM$. As we will work locally on $\XS$, we can assume that this family is finite.

Let $\varphi:X\!\to\! T$ be a holomorphic map, and let $\varphi_S=\varphi\times\id_S:X_S\!\to\! T_S$. To a coherent $\DXS$-module we attach its relative characteristic variety
\[
\Char_{X_{ST}/T_S}\cM\subset T^*(\XST/\TS),
\]
which is the relative (over~$\TS$) characteristic variety of any $\DXST$-module generating $\cM_{\varphi_S}$ as a $\DXS$-module.

We aim at proving the $S$-analogue of Theorem \ref{th:bibi86II}, with the terminology of \cite{FFS23}.

\begin{theorem}\label{th:bibi86IIrel}
Let $\cM$ be a regular holonomic $\DXS$-module and let \hbox{$\varphi:X\!\to\! T$} be a morphism. Then each irreducible component of the characteristic variety $\Char_{X_{ST}/T_S}\cM$ takes the~form~\hbox{$T^*_{q|_{\iota_\varphi(Z)}}\!(\XT/T)\!\times\! \Sigma$}, where $T^*_{Z}X\times \Sigma=T^*_{Z_i}X\times \Sigma_{i\ell}$ is some irreducible component of $\Char\cM$.
\end{theorem}

\begin{remark}[analogous to Remark \ref{rem:Teresa}]
The theorem asserts that $\Char_{X_{ST}/T_S}\cM$ is equal to the image of $\Char\cM$ by the projection
\[
T^*(\XS\times_S\TS)\to T^*\bigl[(\XS\times_S\TS)/\TS\bigr].
\]
\end{remark}

We first notice, as~for Theorem \ref{th:bibi86II}, that the question is local on $\XS$ and that Lemma \ref{lem:nonchar} holds as well in the $S$-setting, with the same proof. Next, by considering the filtration $\cF_\bbullet\cM$ recalled in Section \ref{subsec:absolute}, we can assume that $\Char\cM$ is pure dimensional. Since $\dim T^*_{Z_i}X=\dim X$ for any $i$, this assumption amounts to the property that the dimension of each $\Sigma_{i\ell}$ is independent of $i,\ell$. We~denote it by $\sigma(\cM)$. It~is the dimension of the $S$-support of $\cM$.

\begin{lemma}\label{lem:basechange}
Let $\pi:S'\to S$ be a morphism and let $\cM$ be a holonomic $\DXS$-module. If the statement of the theorem holds for $\pi^*\cM$, then it holds for $\cM$.
\end{lemma}

\begin{proof}
One checks that
\[
\Char\pi^*\cM=\pi^{-1}\Char\cM\quad\text{and}\quad \Char_{X_{S'T}/T_{S'}}\pi^*\cM=\pi^{-1}\Char_{X_{ST}/T_S}\cM.
\]
Then the conclusion follows.
\end{proof}

\paragraphe{Determination of the components of $\Char_{X_{ST}/T_S}\cM$}\label{par:detabsoluterel}
We work in a neighborhood of $(x_o,s_o)\in\XS$. For irreducible closed analytic subsets $Z\subset X$ and $\Sigma\subset S$, we set $Z_\Sigma:=Z\times\Sigma$. We~denote by $q_S$ the projection $\XS\times T=\XST\to \TS$. Assume that we know that each irreducible component of $\Char_{X_{ST}/T_S}\cM$ takes the~form
\[
T^*_{q_S|_{\iota_{\varphi_S}(Z_\Sigma)}}\!(\XST/\TS)=T^*_{q|_{\iota_\varphi(Z)}}\!(\XT/T)\times\Sigma.
\]
We~will show that $(T^*_ZX)\times\Sigma$ is an irreducible component of $\Char\cM$, \ie of the form $(T^*_{Z_i}X)\times \Sigma_{i\ell}$. We~argue as in \S\ref{par:detabsolute}, which is possible due to the assumption of equidimensionality of $\Char\cM$. It~implies in particular that for each such $\Sigma$, we have $\dim\Sigma\leq\sigma(\cM)$ because in any case, $\Sigma$ is contained in the $S$-support of $\cM$.

We first assume that $\Char_{X_{ST}/T_S}\cM$ is also equidimensional. By the first part of the theorem, it can be written as $\bigcup_{j\in J}\bigcup_m[T^*_{q|\iota_\varphi(Z_j)}(\XT/T)\times\Sigma_{jm}]$ for some irreducible closed analytic subsets $Z_j\subset\supp_X\cM$ and irreducible closed analytic subsets \hbox{$\Sigma_{jm}\subset S$}. We~argue as in \S\ref{par:detabsolute}\eqref{par:detabsolute1} and consider a Whitney stratification $(X_\alpha)_{\alpha\in A}$ of $X$ compatible with the family $(Z_k)_{k\in I\cup J}$. We~also consider the set $A_r$, where $r$ it the maximal rank of $\varphi$ on $\supp_X\cM$. We~apply Lemma \ref{lem:nonchar} to each stratum $X_\alpha$ with $\alpha\in A_r$: if $Z_j$ cuts a stratum $X_\alpha$ with $\alpha\in A_r$, then $Z_j$ is equal to some $Z_i$ with $i\in I$ and any~$\Sigma_{jm}$ is equal to some $\Sigma_{i\ell}$. Moreover, as~$\Char_{X_{ST}/T_S}\cM$ is assumed to be equidimensional, this also shows that the dimension of each of its irreducible components is equal to $\dim X+r+\sigma(\cM)$.

If~some~$Z_j$ does not cut $\bigcup_{\alpha\in A_r}X_\alpha$, then $\dim T^*_{q|_{\iota_\varphi(Z_j)}}<\dim X+r$, and so, since $\dim\Sigma_{jm}\leq\sigma(\cM)$, we have $\dim T^*_{q|_{\iota_\varphi(Z_j)}}(\XT/T)\times\Sigma_{jm}<\dim\Char_{X_{ST}/T_S}\cM$, in contradiction with the equidimensionality assumption. Therefore, such a $Z_j$ does not exist

If $\Char_{X_{ST}/T_S}\cM$ is not equidimensional, we argue as in \S\ref{par:detabsolute}\eqref{par:detabsolute2} by considering locally a generating coherent $\DXST$-submodule $\cN$ of $\cM_{\varphi_S}$ and its filtration $\cF_\bbullet(\cN)$.\qed

\paragraphe{Induction hypothesis for the first part of Theorem \ref{th:bibi86IIrel}}\label{par:pffirstpartrel}

We argue as for Theorem \ref{th:bibi86II}, by~induction on $(d,\delta)=(\dim\supp\cM,\dim\varphi_S(\supp\cM))$ lexicographically ordered, and we can assume $\delta\geq1$. We~first check that Lemma \ref{lem:push} applies in the relative setting.

\begin{lemma}\label{lem:pushrel}
Let $\pi:X'\to X$ be a morphism from a complex manifold $X'$ and set $\varphi'=\varphi\circ\pi$. Let $\cM'$ be a regular holonomic $\cD_{X'_S/S}$-module such that
\[
(\dim\supp\cM',\dim\varphi_S(\supp\cM'))\leq(d,\delta).
\]
Assume that $\pi$ is proper on $\supp\cM'$, that $\cM'$ is good, and that the first part of the theorem (hence the second part, by \textup\S\ref{par:detabsoluterel}) holds for~$\cM'$. Then it holds for $\D\pi_*^{(j)}\cM'$ for each $j$. 
\end{lemma}

\begin{proof}
By considering the filtration $\cF_\bbullet(\cM')$, we can assume that $\Char\cM'$ is equidimensional, that is, its components take the form $T^*_{Z'_i}X'\times\Sigma_{i\ell}$, with $\dim\Sigma_{i\ell}=\sigma(\cM')$ independent of $i,\ell$. We~obtain in particular the equality
\[
\delta:=\dim\varphi_S(\supp\cM')=\dim\varphi(\supp_X\cM')+\sigma(\cM').
\]

We~define $\cN_j$ in a way similar to that of Lemma \ref{lem:push}, by replacing $T$ with $T_S$, and we get the inclusion corresponding to \eqref{eq:Kestimate}:
\[
\Char\cN_j\subset\pi\bigl(\Char\cN'\cap\pi^*(T^*X\times \TS/\TS\bigr).
\]
Since the theorem holds for $\cM'$, the irreducible components of $\Char\cN'$ take the form $T^*_{q'|_{\iota_{\varphi'(Z'_i)}}}(\XpT/T)\times\Sigma_{i\ell}$ and the right-hand side above for such a component reads
\[
\Bigl(\pi\bigl[T^*_{q'|_{\iota_{\varphi'(Z'_i)}}}(\XpT/T)\cap\pi^*(T^*(\XT/T)\bigr]\Bigr)\times\Sigma_{i\ell}.
\]
By Lemma \ref{lem:relconormal}, each irreducible component of that set is a relative conormal space of dimension $\dim X+\dim\varphi'(Z'_i)+\sigma(\cM')\leq\dim X+\delta$, and therefore $\dim\Char\cN_j\leq \dim X+\delta$.

On the other hand, any irreducible component $C$ of $\cF_{\dim X+\delta-1}\cN_j$ satisfies, according to \eqref{eq:dimestimate},
\[
\dim X+\delta-1\geq\dim C\geq\dim X+\dim q_S\circ\varpi_S(C),
\]
where the latter inequality follows from \eqref{eq:dimestimate} applied to $q_S\circ\varpi_S$. We~deduce the inequality $\dim q_S\bigl(\supp\cF_{\dim X+\delta-1}\cN_j\bigr)\leq\delta-1$ and, by induction applied to the $\cD_{\XST}$-submodule it generates in $\D\pi^{(j)}_*\cM'$, each such $C$ is a relative conormal space.

Lastly, the components of $\Char\cN_j$ of dimension $\dim X+\delta$ are among those of $\pi\bigl(\Char\cN'\cap\pi^*T^*(X_{\TS}/\TS)\bigr)$, and the assertion for them follows from Lemma \ref{lem:relconormal}. This ends the proof of Lemma \ref{lem:pushrel}.
\end{proof}

Let $Y$ be a hypersurface of $X$ and set $X^*=X\moins Y$. Recall (\cf \cite[Def.\,4.23]{FFS23}) that a regular holonomic $\DXS$-module $\cM$ is said to be of D-type along $Y$ if it is also a $\DXS(*D)$-module and $\cM_{|\XsS}$ is $\DXsS$-holonomic with characteristic variety contained in the zero section, equivalently, there exists an $S$-locally constant sheaf~$L$ of coherent $\pOS$-modules on $\XsS$ such that
\[
\cM_{|\XsS}\simeq (\sho_{\XsS}\otimes_{\pOS}L,\rd_{\XsS/S}\otimes\id).
\]
There exists a coherent $\cO_S$\nobreakdash-module $G$ such that $L$ is locally isomorphic to $p^{-1}G$. Such a $G$ is unique up to non canonical isomorphisms. If $G$ is locally $\cO_S$-free, we say that $\cM_{|\XsS}$ is \emph{strict}. Note that the $S$-support of $\cM$ is the support of $G$, since for $s_o\in S$, we have
\[
\cM|_{X\times\nb(s_o)}=0\iff\cM|_{X^*\times\nb(s_o)}\iff G|_{\nb(s_o)}=0;
\]
the first equivalence is caused by $\cM$ being a coherent $\DXS(*D)$-module, while the second one is due to faithful flatness of $\cO_{\XS}$ over $\pOS$. Recall that, by our induction hypothesis, we have $\dim\Char\cM=d$, so that $d=\dim X+\dim\supp G$, and the theorem holds if $\dim\Char\cM<d$.

\begin{coro}
It is enough to prove the theorem when $\cM$ is of D-type along a normal crossing divisor $D$ in $X$.
\end{coro}

\begin{proof}
The proof is then identical to that of \S\ref{par:reductionnc}. Here, since the question is local on $\XS$, the existence of $\pi$ is given by \cite[Prop.\,4.27]{FFS23}.
\end{proof}

\paragraphe{Reduction to the case of strict D-type}

Assume that $\cM$ is regular holonomic of D-type on $(X,D)$ with $\Char\cM$ equidimensional of dimension $d$, so that $\supp G$ is equidimensional of dimension $d-\dim X$. Let $\cI$ be the reduced ideal of $\supp G$. The $\cI$-adic filtration of $\cM$ has graded quotients annihilated by $\cI$. As the statement of Theorem \ref{th:bibi86IIrel} is stable by extension in $\Mod(\DXS)$, we can assume that $\cM$ is annihilated by $\cI$, so that, denoting by $\iota$ the inclusion $\supp G\hto S$, we have $\cM=\iota_*\iota^*\cM$, $L=\iota_*\iota^*L$ and $G=\iota_*\iota^*G$. In particular, the theorem holds if $d=\dim X$, according to \S\ref{par:endabsolute}. We~will assume that $d>\dim X$, equivalently $\dim\supp G>0$.

Let $\pi:S'\to \supp G$ be a proper modification, which is an isomorphism of Zariski dense open subset $S^{\prime o}\isom (\supp G)^o$, such that
\begin{itemize}
\item
$S'$ is a complex manifold of dimension $d-\dim X$, and
\item
$\pi^*(\iota^*G)$ modulo its $\cO_{S'}$-torsion is a locally free $\cO_{S'}$-module of finite rank.
\end{itemize}
Such a proper (in fact projective) modification exists, according to \cite[Th.\,6]{A-H-V18} and \hbox{\cite[Th.\,3.5]{Rossi68}}. Since $\dim S'=\dim \supp G=d-\dim X$, the support of the torsion of $\pi^*(\iota^*G)$ has dimension $<d-\dim X$. We~conclude that $\pi^*(\iota^*\cM)$ is the extension of a strict regular holonomic $\DXSp$-module of D-type by a regular holonomic $\DXSp$-module of D-type for which the support of $G$ has dimension~$<d-\dim X$, hence which satisfies the theorem by the induction hypothesis.

Assume that we have proved the theorem for \emph{strict} regular holonomic relative $\cD$-modules of D-type. Then the theorem holds for $\pi^*(\iota^*\cM)$. We~conclude with Lemma~\ref{lem:basechange} that, since the theorem holds for $\pi^*(\iota^*\cM)$, it holds for $\cM$.\qed

\paragraphe{The case of strict D-type: reduction to rank one}
In this setting, we have $\dim\supp G=\dim S$, so that $d=\dim X+\dim S$ and $\Char\cM=\bigcup_i(T^*_{Z_i}X\times S)$. We~aim at proving that the highest dimensional components of $\Char_{X_{ST}/T_S}\cM$ take the form $T^*_{q|_{\iota_\varphi(Z)}}X\times S$. We~now argue by induction on the rank $r$ of $G$, assuming it is~\hbox{$\geq2$}. Furthermore, we know by \cite[Prop.\,4.25(i)]{FFS23} that $\cM$ is a Deligne extension~$\wt E_L$. We~also recall that the question is local, so that we can choose local coordinates adapted to $D$ as in \S\ref{par:endabsolute} and, according to \cite[Proof of Th.\,4.15(i)]{FFS23}, we can assume that
\[
\cM\simeq(\cO_{\XS}(*D)\otimes_{\pOS}p^{-1}G,\nabla),
\]
with $G=\cO_S^r$ and $\nabla=\rd\otimes\id+\sum_{i=1}^\ell\rd x_i/x_i\otimes A_i(s)$, where $A_i(s)$ is an $r\times r$ matrix with entries in $\cO_S$. We~now argue as in \cite[Step 4 of the proof of Th\,4.15(iii)]{FFS23}: there exists a base change $\pi:S'\to S$ such that the matrices $A_i\circ\pi$ have a common eigenvector, and we can apply the induction hypothesis, thanks to Lemma \ref{lem:basechange}. In such a way, we~can assume that $r=1$.\qed

\paragraphe{The case of strict D-type in rank one}
We can argue as in \cite[Step~5 of the proof of Th\,4.15(iii)]{FFS23} to obtain a presentation of $\cM$ similar to \eqref{eq:presentationMprime} in \S\ref{par:endabsolute}, with $\alpha_i\in\Gamma(\nb(s_o),\cO_S)$. We~similarly obtain the operators \eqref{eq:operators}, with $\beta_j$ being holomorphic functions of $s$. As the computation of the relative characteristic variety in \S\ref{par:endabsolute} does not rely on $\beta_j$, the argument given there applies to the present setting. This yields the conclusion, on noting that the set $\Sigma$ is thus equal to $S$, as desired in the setting where $G$ is locally free.\qed

   \part*{Appendix II. Proof of the relative log characteristic cycle formula}
\setcounter{section}{4}   
\setcounter{theorem}{0}   
In this appendix, we prove the log characteristic cycle formula in
Theorem \ref{thm:relGins-Sab} by using Theorem \ref{th:bibi86IIrel} from
Appendix I.

Let $X$ and $T$ be complex manifolds, and let $R$ be a regular commutative
finitely generated $\C$-algebra. We write
\[
X_R=X\times \Specan R
\quad \text{and} \quad
T_R=T\times \Specan R.
\]
Let
$q\colon X\times T\to T$
be the projection, and let
$T^*(X\times T/T)$
denote the relative cotangent bundle. If $Z$ is a closed analytic subvariety
of $X\times T$, we consider the \emph{relative conormal space}
$T^*_{q|_Z}(X\times T/T)
\subseteq
T^*(X\times T/T);$
see Definition \ref{def:relconormal}. Then
\[
\dim T^*_{q|_Z}(X\times T/T)\le \dim X+\dim T.
\]
Moreover, by Sard's theorem,
$\dim T^*_{q|_Z}(X\times T/T)=\dim X+\dim T$
if and only if $q(Z)$ contains an open subset of $T$; see \cite[\S 2]{BMM}.

Let $\sM$ be a coherent
$\shD_{X_R\times T/\Specan R}\text{-module}$,
and let $\sN$ be a coherent
$\shD_{X_R\times T/T_R}$-submodule
of $\sM$ which generates $\sM$ over
$\shD_{X_R\times T/\Specan R}$.
Similarly to Lemma \ref{lem:independant}, the relative characteristic variety
of $\sN$ inside
$T^*(X_R\times T/T_R)$
depends only on $\sM$. We denote it by
$\Ch_{q_R}(\sM)$,
where
$q_R\colon X_R\times T\to T_R$
is the natural projection.

\begin{theorem}[Sabbah]\label{thmApp:Sabrelch}
Let $\sM$ be a relative regular holonomic
$\shD_{X_R\times T/\Specan R}$-module. Then every irreducible component of
$\Ch_{q_R}(\sM)$ is of the form
\[
T^*_{q|_Z}(X\times T/T)\times \Sigma,
\]
where
$T^*_Z(X\times T)\times \Sigma$
is an irreducible component of the relative characteristic variety of $\sM$
inside
$T^*(X_R\times T/\Specan R).$
\end{theorem}

Similarly to the equivalence between Theorem \ref{thm:equivformrelcc} and Theorem \ref{th:bibi86II}, the above theorem is
equivalent to a special case of Theorem \ref{th:bibi86IIrel} by taking $S=\Specan R$. Motivated by the ideas of
\cite{BMM}, we now prove Theorem \ref{thm:relGins-Sab}, following the same
strategy as in \cite[\S 3.4]{Wuch}.

\begin{lemma}\label{lemApp:pureloc}
Let $\sM$ be a relative regular holonomic $\shD_{X\times T/T}$-module, and let
$D\subseteq X$ be an effective divisor. If
$\sM|_{(X\setminus D)\times T}$
is $j$-pure, then
$\sM(*(D\times T))$
is also $j$-pure.
\end{lemma}

\begin{proof}
By \cite[Cor.4.31]{FFS23},
$\sM(*(D\times T))$
is also relative regular holonomic. By purity, we may write
\[
\Ch(\sM)=\sum_i T^*_{Z_i}X\times \Sigma_i
\subseteq
T^*X\times T,
\]
where, for those $i$ with
$Z_i\cap (X\setminus D)\neq \emptyset$,
the corresponding components have the expected dimension. Suppose that an
irreducible component $\Sigma_{i\lambda}$ of some $\Sigma_i$ gives a component
of dimension larger than the expected one. Then necessarily
$Z_i\subseteq D.$
Since the relative characteristic variety is locally finite, we may take a
small neighborhood intersecting
$T^*_{Z_i}X\times \Sigma_{i\lambda}.$
Localizing $\sM(*(D\times T))$ at a general point of this component, the
relative characteristic variety of $\sM(*(D\times T))$ would contain this
component. This is impossible, since $\sM(*(D\times T))$ has no nonzero
submodules supported on $D\times T$. Therefore the grade of
$\sM(*(D\times T))$ is still $j$.

Since $\sM|_{(X\setminus D)\times T}$ is $j$-pure, any submodule of
$\sM(*(D\times T))$ of grade greater than $j$ must be supported on
$D\times T$. Such a submodule must be zero. Hence $\sM(*(D\times T))$ is
$j$-pure.
\end{proof}

\begin{proof}[Proof of Theorem \ref{thm:relGins-Sab}]
For simplicity, we also use $Y$ to denote its analytification $Y^\an$, and
similarly for $X$. Set
\[
\widetilde Y
\coloneqq
Y\times \C_y^l
=
X\times \C_t^l\times \C_y^l,
\]
where $y=(y_1,\dots,y_l)$ are the coordinates on $\C_y^l$. Consider the map
\[
\phi\colon \widetilde Y\to Y,
\quad
(x,t_1,\dots,t_l,y_1,\dots,y_l)
\mapsto
(x,e^{y_1}t_1,\dots,e^{y_l}t_l),
\]
and the induced map
\[
\widetilde\phi=(\phi,\id)\colon \widetilde Y_R\to Y_R.
\]
Both $\phi$ and $\widetilde\phi$ are smooth. Set
\[
\sM\coloneqq \widetilde\phi^*\big(\widetilde{\cN(*D)}\big)
\quad \text{and} \quad
\sN\coloneqq \widetilde\phi^*(\widetilde\cN),
\]
where $\widetilde\cN$ denotes the analytic sheaf associated with $\cN$. Let
$\widetilde D=\phi^*D$
and let
$\widetilde D_R\subseteq \widetilde Y_R$
be the associated relative divisor.

Since $\phi$ is smooth, $\sN$ is a coherent
$\widetilde{\sA^R_{\widetilde Y,\widetilde D}}\text{-module}$.
We have the correspondence
\[
T^*(Y,D)_R
\xleftarrow{\ \omega\ }
T^*(Y,D)\times_Y \widetilde Y_R
\xrightarrow{\ \rho\ }
T^*(\widetilde Y,\widetilde D)_R,
\]
where $\omega$ is smooth and $\rho$ is a closed embedding. By smoothness,
\[
\Ch_{\widetilde{\sA^R_{\widetilde Y,\widetilde D}}}(\sN)
=
\rho\big(\omega^{-1}(\Ch^\dagger(\cN))\big).
\]

On the other hand, one observes that
$t_i\partial_{t_i}-\partial_{y_i}$
annihilates
$1\otimes \widetilde\phi^{-1}m\in \sN$.
Thus $\sN$ is a coherent
$\shD_{\widetilde Y_R/\C^l_{t,R}}\text{-submodule}$
of $\sM$ which generates $\sM$. Let
\[
q\colon
\widetilde Y=X\times \C_t^l\times \C_y^l
\to
\C_t^l
\]
be the projection. Then
\[
\iota\big(\Ch_{q_R}(\sM)\big)
=
\iota\big(\Ch_{\shD_{\widetilde Y_R/\C^l_{t,R}}}(\sN)\big)
=
\Ch_{\widetilde{\sA^R_{\widetilde Y,\widetilde D}}}(\sN),
\]
where
$\iota\colon
T^*(\widetilde Y_R/\C^l_{t,R})
\hookrightarrow
T^*(\widetilde Y,\widetilde D)_R$
is the closed embedding corresponding to the relations
$t_i\partial_{t_i}-\partial_{y_i}=0.$
In coordinates, this embedding is given by
\[
(x,t,y,\alpha,\xi_x,\xi_y)
\mapsto
(x,t,y,\alpha,\xi_x,\xi_t^{\log}=\xi_y,\xi_y).
\]

Since $\iota$ and $\rho$ are closed embeddings and $\omega$ is smooth, it is
enough to prove that $\Ch_{q_R}(\sM)$ has no irreducible components supported
over $\widetilde D_R$.

Now consider the usual relative Lagrangian correspondence
\[
T^*Y_R
\xleftarrow{\ \omega'\ }
T^*Y\times_Y \widetilde Y_R
\xrightarrow{\ \rho'\ }
T^*\widetilde Y_R.
\]
We have
\[
\Ch(\sM)
=
\rho'\big((\omega')^{-1}(\Ch(\cN(*D)))\big).
\]
Since $\cN(*D)$ is relative holonomic, we may write
\[
\Ch(\cN(*D))
=
\sum_{i\in L} T^*_{Z_i}Y\times \Sigma_i,
\]
where the $\Sigma_i\subseteq \Spec R$ are algebraic subvarieties, possibly
reducible. Let
$L'\subseteq L$
be the subset consisting of those indices $i$ such that
\[
Z_i\cap U\neq \emptyset,
\qquad
U=Y\setminus D.
\]
We write
$\widetilde Z_i=Z_i\times_Y \widetilde Y.$
Then
\[
\Ch(\sM)
=
\sum_{i\in L}T^*_{\widetilde Z_i}\widetilde Y\times \Sigma_i.
\]

By \cite[Thm.1]{FFS23}, $\sM$ is a relative regular holonomic
$\shD_{\widetilde Y_R/\Specan R}\text{-module}.$
By Theorem \ref{thmApp:Sabrelch}, every irreducible component of
$\Ch_{q_R}(\sM)$ is of the form
\[
T^*_{q|_{\widetilde Z_i}}(\widetilde Y/\C_t^l)
\times
\Sigma_{i\lambda},
\]
for some $i\in L$, where $\Sigma_{i\lambda}$ is an irreducible component of
$\Sigma_i$.

The key observation, due to Brian\c{c}on--Maisonobe--Merle, is that
$q(\widetilde Z_i)$
contains an open subset of $\C_t^l$ if and only if $i\in L'$. We now take the
pure filtration $T_\bullet\sM$ of $\sM$. By Lemma \ref{lemApp:pureloc} and
the uniqueness of the pure filtration,
$T_j\sM$ and
$T_j\sM/T_{j+1}\sM$
are $\cO_{\widetilde Y_R}(*\widetilde D_R)$-modules for all $j$. By relative
holonomicity, for every $j$ we can write
\[
\Ch(T_j\sM/T_{j+1}\sM)
=
\sum_{i\in L_j\subseteq L}
T^*_{\widetilde Z_i}\widetilde Y\times \Sigma_i^j.
\]
Thus, by induction on the pure filtration, it is enough to assume that $\sM$
is pure. In particular, all irreducible components $\Sigma_{i\lambda}$ have
the same dimension; denote this dimension by $k$. By Theorem
\ref{thmApp:Sabrelch}, we have
\[
\dim \Ch_{q_R}(\sM)=\dim \widetilde Y+k.
\]
Since
$\sM|_{\widetilde Y_R\setminus \widetilde D_R}
=
\sN|_{\widetilde Y_R\setminus \widetilde D_R}$
and $\sM|_{\widetilde Y_R\setminus \widetilde D_R}$ is pure, $\sN|_{\widetilde Y_R\setminus \widetilde D_R}$
is pure and any nonzero
coherent
$\shD_{\widetilde Y_R/\C^l_{t,R}}
\big|_{\widetilde Y_R\setminus \widetilde D_R}$-submodule
of $\sN|_{\widetilde Y_R\setminus \widetilde D_R}$ has its characteristic variety
of dimension
$\dim \widetilde Y+k$,
again by Theorem \ref{thmApp:Sabrelch}. Since
\[
\sN\subseteq \sM=\sM(*\widetilde D_R),
\]
it follows that $\sN$ itself is pure. Consequently
$\Ch_{\widetilde{\sA^R_{\widetilde Y,\widetilde D}}}(\sN)$
is equidimensional.

By dimension counting, every irreducible component of $\Ch_{q_R}(\sM)$ must
therefore be of the form
\[
T^*_{q|_{\widetilde Z_i}}(\widetilde Y/\C_t^l)
\times
\Sigma_{i\lambda}
\]
for some $i\in L'$. Equivalently, $\Ch_{q_R}(\sM)$ has no irreducible
component supported over $\widetilde D_R$. This proves the desired
characteristic cycle formula and completes the proof of
Theorem \ref{thm:relGins-Sab}.
\end{proof}

\bibliographystyle{amsalpha} 
\bibliography{mybib,appendix_sabbah}
\end{document}